\documentclass[12pt,leqno]{amsart} 

\usepackage{amsmath}
\usepackage{amssymb}
\usepackage{amsthm} 
\usepackage{mathrsfs}
\usepackage[pagebackref]{hyperref} 
\hypersetup{
	colorlinks=true, % make hyperlinks different color
	linkcolor=blue,  % color of \ref
	citecolor=blue, % color of \cite
}
\usepackage{esint} 
\usepackage{enumerate} 
\usepackage[shortlabels]{enumitem}
\usepackage{mathtools}
\usepackage{comment} 
\usepackage[english]{babel}
\usepackage{tikz}

\numberwithin{equation}{section}

\newtheorem{lemma}{Lemma}[section]
\newtheorem{prop}[lemma]{Proposition}
\newtheorem{thm}[lemma]{Theorem}
\newtheorem{cor}[lemma]{Corollary}
\theoremstyle{definition}
\newtheorem{rmk}[lemma]{\bf Remark}
\newtheorem{defn}[lemma]{\bf Definition}

\newcommand{\N}{\mathbb{N}}

\newcommand{\R}{\mathbb{R}}

\renewcommand{\S}{\mathbb{S}}
\renewcommand{\epsilon}{\varpepsilon}
\renewcommand{\emptyset}{\varnothing}
\newcommand{\sing}{\mathcal{S}}

\newcommand{\ra}{\rightarrow}
\newcommand{\da}{\downarrow}

\newcommand{\HD}{\mathcal{H}}
\newcommand{\norm}[1]{\left \lVert #1 \right \rVert}
\newcommand{\abs}[1]{\left \lvert #1 \right \rvert}
\newcommand{\dotp}[1]{\left \langle #1 \right \rangle}
\newcommand{\ceil}[1]{\left \lceil #1 \right \rceil}
\newcommand{\floor}[1]{\left \lfloor #1 \right \rfloor}
\newcommand{\loc}{\mathrm{loc}}
\newcommand{\ol}[1]{\overline{#1}}
\newcommand{\cal}{\mathcal}

\DeclareMathOperator{\divv}{div}

\DeclareMathOperator{\dist}{dist}

\DeclareMathOperator{\Tr}{Tr}

\renewcommand{\L}{\mathcal{L}}
\renewcommand{\epsilon}{\varepsilon}

\newcommand{\A}{\mathcal{A}}

\newcommand{\cphii}{{C_2}}
\newcommand{\Jac}{\mathrm{Jac}}
\newcommand{\Hom}{\mathrm{H}}

\allowdisplaybreaks

\usepackage[pdftex,dvipsnames]{xcolor} % Coloured text etc.
\usepackage[colorinlistoftodos,prependcaption,textsize=tiny]{todonotes}

\title[Singular set estimates in higher co-dimension]{Singular set estimates for solutions to elliptic equations in higher co-dimension}
\author{Max Engelstein, Cole Jeznach, and Yannick Sire}
\thanks{M.E. was partially supported by NSF DMS CAREER 2143719. He would also like to thank Eugenia Malinnikova for her encouraging words at the beginning of the project. Y.S. was partially supported by NSF DMS 2154219. C.J. was partially supported by Simons Foundation grant 563916 SM, NSF DMS CAREER 2143719, the European Research Council (ERC) under the European Union's Horizon 2020 research and innovation programme (grant agreement 101018680), and would like to thank Eugenia Malinnikova and Mariana Smit Vega Garcia for helpful conversations regarding the Almgren frequency function.
}
\address{School of Mathematics, University of Minnesota, Minneapolis MN, 55455, USA.}
\email{mengelst@umn.edu} 
\address{Departament de Matem\`{a}tiques, Universitat Aut\`{o}noma de Barcelona, Bellaterra, 08193, Spain}
\email{colejeffrey.jeznach@uab.cat} 
\address{Department of Mathematics, Johns Hopkins University, Baltimore MD, 21218 USA.}
\email{ysire1@jhu.edu} 
\date{\today}
\addtocontents{toc}{\protect\setcounter{tocdepth}{1}}

\begin{document}

\begin{abstract}
Recent advances in quantitative unique continuation properties for solutions to uniformly elliptic, divergence form equations (with Lipschitz coefficients) has led to a good understanding of the vanishing order and size of singular and zero set of solutions. Such estimates also hold at the boundary, provided that the domain is sufficiently regular. In this work, we investigate the boundary behavior of solutions to a class of elliptic equations in the higher co-dimension setting, whose coefficients are neither uniformly elliptic, nor uniformly Lipschitz. Despite these challenges, we are still able to show analogous estimates on the singular set of such solutions near the boundary. Our main technical advance is a variant of the Cheeger-Naber-Valtorta quantitative stratification scheme using cones instead of planes.
\end{abstract}

\maketitle
\tableofcontents

\section{Introduction} 

The question of quantitative unique continuation, i.e. what can the zero or singular set of a non-trivial solution to an elliptic equation look like, has a long history, starting with foundational work in \cite{MR1639155,HS89, GL_86,MR1090434, MR943927}. In more recent years, motivated both by applications to geometric problems and by conjectures of Yau and Nadirashvili on the size of nodal domains for eigenfunctions in terms of their eigenvalues, there has been a tremendous amount of progress. In particular, the development of quantitative-stratification techniques introduced by Cheeger, Naber and Valtorta in e.g. \cite{CNV15, CHN15, CN13, NV17} (see also the recent paper dealing with PDEs with H\"older coefficients \cite{jiang}) as well as the more combinatorial arguments introduced by Logunov, Malinnikova and their collaborators in \cite{LM18, Logunov18_upp, Logunov18_low, LMNN21}, has led to several breakthroughs in quantitative unique continuation properties for solutions to linear and nonlinear elliptic PDE. The above-mentioned works concern {\sl interior} unique continuation and geometric measure estimates of {\sl interior} nodal sets.  

Many of the above techniques have also been applied to the related questions of {\sl boundary} unique continuation. Collectively, these sorts of results say that if $u$ satisfies
\begin{equation}\label{eqn:bdry_unique_harm}
\begin{alignedat}{3}
-\Delta u & = 0, & \qquad &  \text{ in }  \Omega, \\
u & = 0, & \qquad & \text{ on } V \subset \partial \Omega,
\end{alignedat}
\end{equation}
where $V \subset \partial \Omega$ is a relatively open subset of the boundary, and $\partial \Omega \cap V$ is sufficiently regular, then the set where normal derivative $\partial_n u$ vanishes, $\sing(u) \coloneqq V \cap \{ \partial_n u = 0\}$ cannot be too large. The precise smoothness required on $\partial \Omega$ (and what too large means) are major open questions with applications to control theory (see \cite{Bangbang}).  To date it has been proven that:
\begin{enumerate}[(i)]
\item if $\partial \Omega \cap V$ is given by a $C^{1,\mathrm{dini}}$ graph, then $\HD^{n-2}(\sing(u) \cap V) < \infty$ \cite{KZ22},
\item if $\partial \Omega \cap V$ is given by small-constant Lipschitz graph, then $\HD^{n-1}(\sing(u)) = 0$ \cite{Tolsa23}.
\end{enumerate}
Furthermore, an example of \cite{KZ23} shows that the Dini-smoothness assumption is sharp if by small one means finite $(n-2)$-Hausdorff measure. If by small one means zero surface measure, then examples of domains with mutually singular harmonic and surface measure show that some rectifiability of the boundary must be assumed. We described above the latest developments leading to some optimal smoothness of the boundary but previous results with smoother boundaries are proved in  \cite{AE,AEK, McCurdy} for instance. 

In this article, we study an analogous boundary unique continuation problem for a class of degenerate-singular elliptic PDEs in domains whose boundary is of low dimension. Roughly speaking, we consider domains $\Omega \subset \R^n$ whose boundaries $\partial \Omega$ are quantitatively $d$-dimensional with $d < n-1$,  and operators of the form
\begin{align}\label{eqn:op}
- \divv( \mathcal{A} \dist(X, \partial \Omega)^{-n+d+1} \nabla u) = 0,
\end{align}
where $\mathcal{A}$ is uniformly elliptic and satisfies a certain smoothness and structure conditions. Such operators were introduced in \cite{DFM21AMS} in order to provide an analogue of harmonic measure (the hitting measure of Brownian motion) which detects sets of co-dimension larger than 2. Notice that since $d < n-1$, the operator above is neither uniformly elliptic, nor uniformly Lipschitz no matter the choice of the (elliptic) matrix $\mathcal{A}$. Despite this, recent work (see for example, \cite{DFMDAHL, DM23, DFM23}) has shown that solutions of \eqref{eqn:op} in many ways exhibit similar behavior to their co-dimension one counterparts (i.e., harmonic functions). In this setting, it turns out we can also prove a boundary-unique continuation type result for solutions provided that $\partial \Omega$ and $\mathcal{A}$ are sufficiently regular. Before we state the theorem, we need to introduce two key notions:

For a $C^1(\Omega)$ solution $u$ of \eqref{eqn:op}, define the \textit{singular} set of $u$ by
\begin{align*} \sing(u) \coloneqq \{ X \; : \; u(X) = \abs{\nabla u(X)} = 0 \}.
\end{align*}
Of course the choice of the word singular here is to refer to the fact that $\sing(u)$ contains the set where the nodal set of $u$, $Z(u):=\{ X\; : \; u(X)=0 \}$, is not given locally by a smooth hypersurface. Our main concern shall be uniform estimates on the size of $\sing(u)$ near $\partial \Omega$, since a standard re-scaling argument allows one to reduce analogous estimates on compact subsets of $\Omega$ to the works in, say, \cite{NV17}. Finally, we introduce the notion of doubling index (\`a la \cite{GL_86})

\begin{align*}
\mathscr{N}_u(X_0, r) \coloneqq \dfrac{ \int_{B_{8r}(X_0)} u^2 \dist(X,\partial \Omega)^{-n+d+1} \; dX }{   \int_{B_r(X_0)} u^2 \dist(X,\partial \Omega)^{-n+d+1}\; dX  }.
\end{align*}

\begin{thm}[Main result]\label{thm:main}
Suppose:
\begin{enumerate}[(i)]
\item $0 \in \Gamma \subset \R^n$ is a $d$-dimensional $C^{1,1}$ graph parametrized by $\phi \in C^{1,1}(\R^d ; \R^{n-d})$ with $\norm{  \nabla \phi }_{\mathrm{Lip}(\R^d)} \le \cphii$, 
\item Let $u\in C(B_1)$ solve $$\begin{aligned} L_\infty(u):= -\divv(\dist(X,\Gamma)^{-n+d+1}\nabla u) =& 0, \qquad \text{in}\,\, B_1\setminus \Gamma\\ u|_{\Gamma \cap B_1} =&0.\end{aligned}$$
\end{enumerate}
Then there exists $r_0 = r_0(n,d,C_2) \in (0,1)$ small enough and $C = C(n,d,C_2, \mathscr{N}_u(0, 100r_0)) >1$ so that 
\begin{align*}
\L^n(B_s(\sing(u)) \cap B_{r_0}) \le C s^2, \qquad \HD^{n-2}(\sing(u) \cap B_{r_0}) \le C.
\end{align*}
Moreover, $\sing(u) \cap B_{r_0}$ is countably $(n-2)$-rectifiable: there exists countably many $(n-2)$-dimensional Lipschitz graphs $\Sigma_i \subset \R^n$ so that 
\begin{align*}
	\HD^{n-2} \left(  \left( \sing(u) \cap B_{r_0} \right) \setminus \bigcup_{i \in \N} \Sigma_i  \right) = 0.
	\end{align*}
\end{thm}

We remark first that our Theorem \ref{thm:main} also holds for operators as in \eqref{eqn:op} with $\mathcal A$ satisfying certain smoothness and structure conditions (see Corollary \ref{cor:main} for a precise statement). In fact, after a change of variables argument, Theorem \ref{thm:main} will follow from Theorem \ref{thm:main_flat}, which proves a similar result when $\Gamma = \mathbb R^d$. We show that this change of variables leads to an operator which satisfies the conditions of Theorem \ref{thm:main_flat} in Appendix \ref{sec:app_gr}, see also Corollary \ref{cor:main}. Unfortunately, we do not know how to show that the pushforward of the $L_\infty$ operators through this change of variables is sufficiently smooth without passing through estimates associated to the so-called regularized distances introduced in Appendix \ref{sec:app_gr}. That the euclidean distance does not behave well with respect to this change of variables is a phenomenon also encountered in the work of David-Feneuil-Mayboroda, see the end of Section 1.4 in \cite{DFMDAHL}. Finally, we should mention that we do not expect assumption on the smoothness of $\phi$ to be sharp but, given that the higher co-dimension setting already contains many new technical difficulties, we chose to work with $C^{1,1}$-graphs in order to highlight the main ideas. In line with the sorts of boundary unique continuation results that hold for harmonic functions \cite{KZ22}, it is reasonable to expect that the sharp regularity on $\phi$ should be that $\nabla \phi$ has a Dini modulus of continuity. However, extending Theorem \ref{thm:main} to such geometries would require new techniques and ideas which are not present in either this paper or the co-dimension 1 literature.

The main Theorem above says that the singular set $\sing(u)$ of solutions to $ L_\infty(u) =0$ is at most $(n-2)$-dimensional, and comes with $\HD^{n-2}$ measure estimates in $\overline{\Omega}$, at least locally.  In the case of harmonic functions $u$ in \eqref{eqn:bdry_unique_harm}, however, the results of \cite{KZ22, Gallegos22, Tolsa23} say that for smooth enough domains, the singular set $\sing(u)$ only saturates a small portion of $\partial \Omega$ (i.e., has zero surface measure, or is a co-dimension one subset of $\partial \Omega$). In our setting, since $\Gamma = \partial \Omega$ is of dimension $d \le n-2$, the $\HD^{n-2}$ estimate provided by Theorem \ref{thm:main} says nothing about the relative size of $\sing(u)$ \textit{inside} $\Gamma$, and indeed the conclusion obtained here is in some sense the best possible. In our study of solutions to equations of the type \eqref{eqn:op}, we construct examples of nontrivial, $\Lambda$-homogeneous solutions of an equation of the type \eqref{eqn:op} in the flat-space $\Omega = \R^n \setminus \R^d$ whose singular set $\sing(u)$ consists of a union of $(n-2)$-planes. {For example, there exists solutions of the form $u(x,t) = v(x) \abs{t}^\gamma \phi(t/\abs{t})$ where $v$ is a homogeneous harmonic polynomial, $\gamma >0$ and $\phi$ is a spherical harmonic on $\S^{n-d-1}$.} Moreover, for an appropriate choice of $\gamma > 2$ these planes contain $\partial \Omega = \R^d$, and so these examples show both that $\sing(u)$ can be $(n-2)$ dimensional, and moreover, it is possible that $\sing(u) \supset \Gamma$. This is one of several phenomena which distinguishes the higher co-dimension setting from the classic co-dimension one problem. We explore these differences through a series of examples in Section \ref{sec:examples} below. We should also state explicitly that these examples show the estimates in Theorem \ref{thm:main} are sharp in the sense that $\HD^{n-2-\epsilon}(\sing(u)\cap B_{r_0}) = +\infty$ is possible for any $\epsilon > 0$ (in fact, we show something even stronger, see the discussion at the end of Section \ref{sec:examples}). The exact dependence of the constants $C > 0$ on the doubling index $\mathscr{N}_u$ is an open problem even in the classical setting, and we make no attempt to optimize that dependence here. However, the new covering arguments introduced here do not introduce fundamentally different errors than those in \cite{NV17}, and we expect the same (non-optimal) dependence in terms of $\mathscr{N}_u$.

It is worth remarking that the (countable) rectifiability of $\sing(u)$ is a short consequence of the work of \cite{NV17}, and the fact that $d \le (n-2)$. Indeed, we will see that in regions far from $\R^d$ (i.e., balls $B_r(Y)$ so that $B_{2r}(Y) \cap \R^d = \emptyset$), $u$ satisfies a uniformly elliptic equation with Lipschitz coefficients, and so we can apply the main Theorem of \cite{NV17} to deduce the $(n-2)$-rectifiability of $\sing(u)$ there. Since $B_1 \setminus \R^d$ is contained in a countable of such balls $B_r(Y)$, we have no problem in deducing the countable $(n-2)$-rectifiability of $\sing(u) \setminus \R^d$. On the other hand, the techniques of this article do not immediately show that $\sing(u) \cap B_{1/2}$ is contained in \textit{finitely} many Lipschitz graphs of dimension $(n-2)$, which is indeed the conclusion shown in \cite{NV17}. The main reason our result differs here is the necessity of a quantitative stratification using cones instead of planes, so while locally finite rectifiability seems plausible, it remains outside the scope of this paper.

In the work \cite{STT20}, the third author with Terracini and Tortone also studied solutions to a class of degenerate-singular elliptic equations and introduced methods to study the nodal set structure of solutions to 
\begin{equation}\label{e:cafsilv}
-\divv(\abs{t}^\alpha \nabla u) = 0 \text{ in } B_1 \subset \R^{n+1},
\end{equation}
where $X = (x,t) \in \R^n_x \times \R_t$ and $\alpha \in (-1,1)$. There, they prove nodal set estimates for such solutions $u$ \textit{inside} the characteristic manifold, $\{ t = 0\}$, where the equation becomes singular or degenerate, and as an application, they prove quantitative unique continuation results for solutions of the fractional Laplacian. While the spirit of their result is certainly related to that of Theorem \ref{thm:main}, a key difference is that we provide estimates on the singular set \textit{in a neighborhood of} the boundary. Of course, in \cite{STT20} the authors were motivated by the celebrated result of Caffarelli and Silvestre \cite{CS2007} identifying fractional powers of the Euclidean Laplacian as the Dirichlet-to-Neumann map of operators of the type above, and thus were not so interested in the singular set outside of the characteristic manifold. {In the case of the Dirichlet-to-Neumann map as in \cite{CS2007} the parameter ranges $\alpha \in (-1,1)$ and the weight falls into the well-known class of $A_2$-weights as investigated in e.g. \cite{fks}. For further applications, it is however important to consider larger powers.  We believe that the tools and techniques we develop in the present work could be used to investigate finer properties of the geometric measure for the singular set (and nodal sets ) of solutions to \eqref{e:cafsilv} in a neighborhood of the set $\{t = 0\}$}. We emphasize that in this context the characteristic manifold is a hypersurface. Interestingly, as a follow-up of \cite{STT20}, Terracini, Tortone and Vita investigated in the very nice paper \cite{TTV} applications of such degenerate equations (and the regularity of their solutions) to higher order boundary Harnack. Indeed in \cite{deSilvaSavin}, De Silva and Savin proved a very unexpected higher order (in H\"older spaces)  Boundary Harnack Principle. The starting point in \cite{TTV}, is that the quotient of two solutions satisfies a similar PDE with coefficients which on the zero set are degenerate.

\subsection{Technical comparison to prior work}
As mentioned above, away from $\partial \Omega$ our equation is uniformly elliptic (with smooth coefficients) and inside $\partial \Omega$ our estimates are trivial. Thus Theorem \ref{thm:main} is most interesting in that it gives uniform estimates on $\sing(u)$ in a \textit{neighborhood} of the boundary, $\partial \Omega$. There has been some work in the classical case which also gives estimates in a neighborhood of the boundary, in particular \cite{McCurdy, KZ22}. Our work (and that of \cite{KZ22}) differs from \cite{McCurdy} in that there is no quantity which is uniformly monotone both in the interior and on the boundary. That is to say, for interior points, our variant of the Almgren frequency formula, see \eqref{eqn:N_var_gen}, is only monotone at scales far smaller than the distance from the point to the boundary. 

This challenge is also present in \cite{KZ22} who overcome it through a clever change of coordinates at the boundary and adapting the powerful rectifiable-Reifenberg scheme of Naber-Valtorta \cite{RectReif}. However, in our problem there is an additional technical challenge, that the blowups on the boundary solve a completely different equation (and have different qualitative properties) from the blowups in the interior, see Section \ref{sec:examples} below. Indeed, in the interior, the model equation is an elliptic constant coefficient PDE, but on the boundary, the model equation is more akin to the degenerate operator \eqref{e:cafsilv}. This leads to many more technical difficulties, for example, at interior points the solution cannot ``look like its blowup" except at scales far smaller than the distance from the point to the boundary. We should mention that while this challenge is not present in \cite{KZ22}, they do work with rougher boundaries and prove finer structure (i.e. rectifiability) of the singular set. 

To overcome this technical difficulty, we introduce a new variation on the quantitative cone-splitting argument of Cheeger-Naber-Valtorta (see, e.g. \cite{CNV15}). Briefly\footnote{for a more in-depth discussion, see the beginning of Section \ref{sec:covering}}, the original argument quantifies the fact that if a solution is homogeneous around two different points, it must be invariant along the line connecting those points. It goes on to use this quantification to conclude that singular points which are near one another and ``look singular" at comparable scales must both lie close to a lower dimensional affine space. From there a covering argument shows that the singular set must be lower dimensional. As mentioned above, we cannot meaningfully apply this observation in our setting, since no interior singular point will ``look singular" at scales comparable to its distance from a boundary singularity. Instead, we show that in a neighborhood of a singular point on the boundary, singularities not on the boundary are constrained to live in a specific cone dictated by the blowup of the boundary singular point. See Corollary \ref{cor:pinch_cone} for the precise statement. The fact that the boundary is smooth tells us that these cones all point in approximately the same direction. This then gives us enough linear structure to employ a variant of the Cheeger-Naber-Valtorta covering argument and prove our Theorem \ref{thm:main}.

\subsection{Outline of the argument and the paper} 
As alluded to above, Theorem \ref{thm:main} is proved by applying a careful change of variables that flattens $\Gamma$ and instead studying solutions to equations of the type \eqref{eqn:op} with certain smoothness and structure assumptions on $\mathcal A$ in the flat space, $\R^n \setminus \R^d$  (see Theorem \ref{thm:main_flat} and Corollary \ref{cor:main}). Since this change of variables introduced in \cite{DFMDAHL} is quite technical, we leave its presentation to the Appendix \ref{sec:app_gr}, but the main takeaway is the existence of such a change of variables (Theorem \ref{thm:change_var}). We mention again that the reason we need to work with the regularized distances at all here is to ensure that after the change of variables the new equation has coefficients with sufficient regularity.  

Once we are in the flat setting, we show that an Almgren-type quantity is monotone increasing when centered on the boundary ( see Lemma \ref{lem:monotone_var}). As usual in these arguments, the existence of this monotone quantity controls the rate of vanishing of solutions. This argument is inspired by and closely follows the work of \cite{GL_86}, but with extra technical difficulties caused by the co-dimension $> 1$ setting.  The monotone quantity also allows us to perform blow-ups: taking a sequence of rescalings of a solution to the equation and showing it limits to a homogeneous solution, $u_\Lambda$, of the ``constant coefficient" operator $-\mathrm{div}(\delta^{-n+d+1}(X)\nabla u_\Lambda) = 0$ in $\mathbb R^n\backslash \mathbb R^d$ (where $\delta(X)$ is the distance to $\mathbb R^d$). The proof of monotonicity and the blowup analysis is done in Section \ref{sec:lip}.

When we classify the homogeneous solutions (in Section \ref{sec:const}) the analysis departs substantially from the classical case (where the blowups are essentially homogeneous harmonic polynomials). In fact, an extremely rich array of behavior can take place depending on the relative values of $n, d$ and the homogeneity. Still, we are able to classify all such solutions, in Theorem \ref{thm:homogen_sol}. Our argument here is inspired by an analogous classification in \cite{CSS_inv_08}, but again the higher co-dimension introduces substantial difficulties and necessitates new arguments. 

In Section \ref{sec:pinch} we quantify the blowup analysis of the previous section. In particular, we show that general solutions are well approximated by homogeneous solutions whenever the frequency is approximately constant in a large range of scales. It is in this section that we prove the key geometric estimate Corollary \ref{cor:pinch_cone}, which shows that singular points in $\mathbb R^n\backslash \mathbb R^d$ which are close to a singularity in $\mathbb R^d$, must be contained in a cone around a lower dimensional affine space. 

In Section \ref{sec:covering} we combine our estimates in the novel covering argument explained above to prove our main results.

\subsection*{Notations and definition of weak solutions}\label{sec:notation}
Let us recall rapidly for the convenience of the reader some notation regarding the correct function spaces in studying solutions of equations of the type \eqref{eqn:op}, though for a more in-depth review, we point the reader to \cite{DFM21AMS}. In addition, we introduce here useful notation that will be used in the remainder of the paper.

For $\Omega \subset \R^n$ so that $\Gamma \coloneqq \partial \Omega$ is the graph of a Lipschitz function $\phi \in \mathrm{Lip}(\mathbb R^{d}; \mathbb R^{n-d})$ we\footnote{{for the elliptic theory we only need $\Gamma$ to be $d$-Ahlfors regular but for our purposes, the Lipschitz graph assumption is not restrictive}} consider operators of the form 
\begin{align}\label{eqn:op2}
L \coloneqq - \divv( A \nabla  \, \cdot \,   \, ).
\end{align}
Here we make the standing assumption that the matrix $\mathcal{A} \coloneqq A \delta(X)^{n-d-1}$ is assumed to be symmetric, uniformly elliptic, and bounded, where $\delta(X) \coloneqq \dist(X, \Gamma)$. 
The important function spaces are 
\begin{align*}
W = W (\Omega) \coloneqq \{ u \in L^1_{\loc}(\Omega) \; : \; \nabla u \in L^2(\Omega, w)\} 
\end{align*}
where $w(X) = \delta(X)^{-n+d+1}$ for $X \in \Omega$. Since these weighted measures come up quite frequently, we write for convenience $$dm(X) = w(X)  dX$$ and $$d \sigma_w(X) = w(X) d\sigma(X) =  w(X) d \HD^{n-1}(X)$$ for the weighted Lebesgue measure and weighted surface measure respectively.

Finally we define a local version of $W$ suitable for defining solutions of \eqref{eqn:op2} as follows: for $E \subset \R^n$ open, the set $W_r(E)$ is defined as
\begin{align*}
W_r(E) \coloneqq \{ f \in L^1_\loc(E) \; : \;  \phi f \in W \text{ for all } \phi \in C_c^\infty(E) \}.
\end{align*}
It can be shown \cite{DFM21AMS} that each $W_r(E)$ is also characterized as 
\begin{align*}
W_r(E) \equiv \{ f \in L^1_\loc(E) \; : \; \nabla f \in L^2_{\loc}(E, w)\}.
\end{align*}
Finally, we say that $u \in W_r(E)$ is a solution to the equation $Lu = 0$ in $E$ provided that for any $\phi \in C_c^\infty(E)$, one has 
\begin{align*}
\int A(X) \nabla u(X) \cdot \nabla \phi(X) \; dX = 0.
\end{align*}
Similarly, we say that $u$ is a subsolution in $E$ provided that 
\begin{align*}
\int A(X) \nabla u(X) \cdot \nabla \phi(X) \; dX \le 0
\end{align*}
for any such $\phi \in C_c^\infty(E)$ with $\phi \ge 0$, and $u$ is a supersolution in $E$ provided that $-u$ is a subsolution in $E$.

Much of the standard elliptic theory (i.e., existence and uniqueness of solutions, Caccioppoli's inequality, H\"{o}lder continuity and De-Giorgi-Nash-Moser bounds for solutions, etc.) holds for equations of the type \eqref{eqn:op2} under these mild standing assumptions. In the end, we shall need several of these elliptic estimates (as well as others not proven in \cite{DFM21AMS}). To avoid cluttering the beginning of this paper with these standard arguments, we refer the reader to Appendix \ref{sec:app_reg}.

\section{\texorpdfstring{Examples of homogeneous solutions in $\R^n \setminus \R^d$}{Examples of homogeneous solutions in the flat space} }\label{sec:examples}
Here we give some examples of homogeneous solutions to the system:
\begin{equation}\label{eqn:soln} \begin{aligned} 
-\divv(\delta^{-n+d+1}(X) \nabla u(X) ) =& 0, \qquad \forall X\in \mathbb R^n\backslash \mathbb R^d\\
u(X_0) =&0, \qquad \forall X_0 \in \mathbb R^d\subset \mathbb R^n.
\end{aligned}\end{equation}
As mentioned above we hope these examples illustrate how solutions to these degenerate PDEs can have singular sets which differ in behavior dramatically from the singular sets of solutions to uniformly elliptic PDE. 

\subsection*{The trivial codimension-one case}Whenever $a(x,r)$ is a homogeneous harmonic polynomial in $\R^{d+1}$ vanishing on $\partial \R^{d+1}_+$, then $u(x, \abs{t})$ is a homogeneous solution to \eqref{eqn:soln} vanishing on $\R^d$.

\subsection*{Frequencies can be non-integral} Let $\Lambda > 0$ be chosen, let $\lambda_j$ be an eigenvalue of the Laplace-Beltrami operator on $\mathbb S^{n-d-1}$ and let $\phi_j$ be a eigenfunction associated to that eigenvalue. If $v(x)$ is a homogeneous harmonic polynomial of degree $\Lambda- \gamma_j$ in $\R^d$, where $\gamma_j$ is chosen so that $\gamma_j(\gamma_j-1) = \lambda_j$, then the function 
\begin{align*}
u(x,t) \coloneqq v(x)|t|^{\gamma_j} \phi_j(t/\abs{t})
\end{align*}
is a $\Lambda$-homogeneous solution vanishing on $\R^d$. As $\gamma_j$ is generally not an integer, this shows that the space of homogeneities, $\mathcal{F}$, is not contained in $\N \cup \{0\}$. Furthermore, since $\lambda_j \ra \infty$ as $j \ra \infty$ (so then $\gamma_j \ra \infty$ as well), this shows also that there exists homogeneous solutions of arbitrarily large (but finite) order of vanishing on $\R^d$. In particular, that $\sing(u) \cap \R^d = \R^d$ is possible for homogeneous solutions.

\subsection*{Non trivial structure of homogeneous solutions.} In the above construction, it is interesting to ask whether given a single homogeneity $\Lambda$, it is possible to find multiple different eigenvalues such that $\Lambda - \gamma_j$ is an integer. This would then allow us to create non-trivial linear combinations of solutions as above and indeed, such solutions may not necessarily be even or odd. 

 Even when $n-d=2$, this is possible, but to simplify computations we will give an example when $n-d-2 = 8$. In this dimension we look at the eigenvalues associated to harmonics of homogeneity two $\lambda_2 = 2(2 + 8)$ and twelve $\lambda_{12} = 12(12 + 8)$. If we let $\gamma_2 = 5$ and $\gamma_{12} = 16$ then $\gamma_i(\gamma_i-1) = \lambda_i$. Then for any pair of integers $k_i$ such that 
\begin{align*}
k_2 + \gamma_2 = k_{12} + \gamma_{12}, 
\end{align*}
and any homogeneous harmonic polynomials $v_i$ of degree $k_i$ in $\R^{d+1}$ vanishing on $\partial \R^{d+1}_+$, we have that the function
\begin{align*}
u(x,t) & \coloneqq \sum_{i=2, 12} v_i(x,\abs{t}) \abs{t}^{\gamma_i} \phi_i(t/\abs{t})
\end{align*}
is a $(k_2 + \gamma_2)$-homogeneous solution to \eqref{eqn:soln} vanishing on $\R^d$ that is neither even nor odd. This last fact can be checked from the observation that $\phi_i$ are even (since they come from the second and twelfth eigenfunctions on the sphere), while $k_i$ must have opposite parities since the $\gamma_i$ have opposite parities. Thus one of the terms $v_i(x,\abs{t}) \abs{t}^{\gamma_i} \phi_i(t/\abs{t})$ is even, while the other is odd. 

\subsection*{Large singular sets.} Given that in higher co-dimension the singular set can contain all of the boundary, it is natural to ask whether Hausdorff dimension or measure is really the correct way to capture the size of the singular set. That is, perhaps it is better to use a measure which varies as one approaches the boundary, to better reflect the degeneracy in the equation. The following examples shows that doing so naively will not work. 

 Consider the radial extension of $u(x,y,z) \coloneqq z(xy^2 - yx^2)$, a homogeneous harmonic polynomial in $\R^3$. If $v(x, y, t) = u(x,y,\abs{t})$ for $t \in \R^{n-2}$, then $v$ is a homogeneous solution of \eqref{eqn:soln} in $\R^{n}$ which vanishes on $\R^2 \times \{0\} \subset \R^{n}$ whose singular set contains $\{0,0\} \times \R^{n-2}$. This example demonstrates that the singular set of solutions vanishing on $\R^d$ can be $n-2$ dimensional, and in fact, that the $\HD^{n-2}$--measure of such solutions can be scale-invariant in $\abs{t}$. In particular, this example shows that if one wants to replace $\HD^{n-2}$ in Theorem \ref{thm:main} with a measure $\nu$ defined by $d\nu(X) = g(\delta(X)) d\HD^{n-2}(X)$, then $g$ must satisfy $\int_{\{0,0\} \times \R^{n-2}} g(\delta(X))  \; d\HD^{n-2}(X) < \infty$. In a similar manner, if one considers solutions of the form $a(x) r^\gamma \phi(\omega)$ where $a$ is harmonic in $x$, $\phi$ is a Laplace eigenfunction on the sphere, and $\gamma$ is appropriately chosen, one can construct examples where the singular set for this solution contains $\sing(a) \times \R^{n-d}$ as well as $\R^d \times [0, \infty) \times \sing(\phi)$, extending the previous example to produce examples of homogeneous singular sets having more complicated geometries (and in particular, showing that there can be a plane of singular points which is not perpendicular to $\R^d$).

\section{Homogeneous solutions to ``constant coefficient'' equations}\label{sec:const}
In this section, our goal is to provide a fine analysis of the order of vanishing of solutions to equations of the type \eqref{eqn:op2} which vanish on $\Gamma$ in the simplest setting: when the boundary $\Gamma$ is affine, and $\cal{A} = I$ is the identity matrix. That is, we are interested in solutions to \eqref{eqn:soln}, which we rewrite here
\begin{equation*} \begin{aligned} 
-\divv(\delta^{-n+d+1}(X) \nabla u(X) ) =& 0, \qquad \forall X\in \mathbb R^n\backslash \mathbb R^d\\
u(X_0) =&0, \qquad \forall X_0 \in \mathbb R^d\subset \mathbb R^n.\end{aligned}
\end{equation*} As mentioned above, homogeneous solutions to these equations model the infinitesimal behavior of solutions to \eqref{eqn:op2} in general, at least when $\Gamma$ is smooth enough. We will make more precise statements in later sections (see Lemma \ref{lem:cont_sing}).

In Theorem \ref{thm:homogen_sol} we provide a complete description of global \textit{homogeneous} solutions $u_\Lambda$ in this affine setting. Of particular importance is the fact that the singular sets $\sing(u_\Lambda)$ of these homogeneous solutions satisfy Minkowski estimates depending only on $\Lambda$ and the dimension, which is the content of Theorem \ref{thm:homogen_m}. 

 {Here is where our methods begin to depart from the more classical case of harmonic functions, where Hausdorff measure estimates for the nodal and singular set of harmonic polynomials follow rather quickly from classical algebraic-geometric facts (see for example, \cite[Theorem 2.1]{HS89}). These Minkowski estimates we obtain here will later be combined with the basic geometric argument of Lemma \ref{lem:cones} to deduce that the singular sets of \textit{homogeneous} solutions are contained in cones (of bounded aperture) about $(n-2)$ planes, which can further be passed to more general solutions via an approximation argument. These cone conditions, will be the most important part of our new covering argument in the proof of the main Theorem.}

{Before we begin the classification, we make two remarks which may be helpful to the reader:}
\begin{rmk}\label{rmk:strong_soln}
	We shall also frequently use the fact that for any $\Omega' \subset \subset \Omega$, any solution $u$ of \eqref{eqn:soln} is also smooth in $\Omega'$ since the coefficients $\delta(X)^{-n+d+1} \equiv \abs{t}^{-n+d+1}$ are smooth there. As a consequence, we see that $u$ satisfies the equation \eqref{eqn:soln} pointwise in the usual sense in $\Omega'$.
\end{rmk}

\begin{rmk}\label{rmk:harm_rot}
	It is worthwhile to point out the rotational symmetry of the equation \eqref{eqn:soln}, which in turns gives us a correspondence between harmonic functions and rotationally-invariant solutions of \eqref{eqn:soln}. In particular, a straight-forward computation shows that if $u(x,t)$ is rotationally-invariant in $t$ in the sense that 
	\begin{align*}
		u(x,t) = v(x,\abs{t})
	\end{align*}
	for some other function $v$, then $u$ solves \eqref{eqn:soln} if and only if $v$ is harmonic.
\end{rmk}

\subsection{Characterization of homogeneous solutions and their properties}
In this section, we provide a complete (explicit) description of all homogeneous solutions to the equation \eqref{eqn:soln} in the flat space $\Omega = \R^n \setminus \R^d$ that vanish on $\R^d $. Such solutions play a role analogous to the role homogeneous harmonic polynomials play with respect to harmonic functions. Our argument is adapted from \cite[Lemma 2.7]{CSS_inv_08}, though the main tool for our Liouville argument comes from Proposition \ref{prop:higher_order_moser} instead of H\"older continuity estimates. We shall see shortly by a separation of variables argument that Laplace eigenfunctions on the sphere (i.e., homogeneous harmonic polynomials) still play an important role in this setting, so for concreteness, let us denote by $\lambda_j \in \N \cup \{0\}$ the eigenvalues associated to $-\Delta_{\S^{n-d-1}}$:
\begin{align*}
\lambda_j \coloneqq j(j+n-d-2).
\end{align*}

\begin{thm}\label{thm:homogen_sol}
Let $u \in W_r(\R^n) \cap C(\R^n)$ be a non-trivial solution to \eqref{eqn:soln} which vanishes continuously on $\R^d  \subset \R^n$, and which is homogeneous of degree $\Lambda \ge 0$. Then $\Lambda$ lives in the discrete set 
\begin{align*}
\mathcal{F} \coloneqq \N \cup \{ \Lambda' \; : \; (\Lambda' -k)(\Lambda'-k-1) = \lambda_j \text{ for some } k,j \in \N \cup \{0\} \} \subset [1, \infty)
\end{align*}
and $u$ admits the following expansion:
\begin{align*}
u(x,t) & = \sum_{j=0}^{N(\Lambda)} a_j(x, \abs{t}) \phi_j(t/\abs{t}),
\end{align*}
where the $\phi_j$ are Laplace eigenfunctions on $\S^{n-d-1}$, in particular $-\Delta_{\S^{n-d-1}}\phi_j = \lambda_j \phi_j$. Furthermore the $a_j$ have the following properties:
\begin{enumerate}[(i)]
\item Each $a_j \in C^\infty(\R^{d+1}_+) \cap C(\overline{\R^{d+1}_+}) $ is homogeneous of degree $\Lambda$, and satisfies 
\begin{equation}\label{eqn:pde_a_j}
\begin{split}
\Delta_{x,r} a_j(x,r) - \dfrac{\lambda_j}{r^2} a_j(x,r) & = 0 \text{ in } \R^{d+1},   \\
a_j(x,0) & = 0. 
\end{split}
\end{equation}
In particular, each $a_j(x, \abs{t}) \phi_j(t/\abs{t})$ is a $\Lambda$-homogeneous solution of \eqref{eqn:soln} vanishing continuously on $\R^d$.
\item When $\Lambda \in \N$ then $a_j(x,r)$ is a polynomial, and otherwise $r^{\ceil{\Lambda} - \Lambda}a_j(x,r)$ is a polynomial in the variables $x$ and $r$. In either case this polynomial is either even or odd. Moreover, $a_0$ is a homogeneous harmonic polynomial in $\R^{d+1}$ vanishing on $\partial \R^{d+1}_+$. 
\item For $j > 0$, the smallest power of $r$ appearing in the expansion $a_j(x,r)$ is $r^{\gamma}$ for some $\gamma \ge \gamma^* = (1 + \sqrt{5})/2 > 3/2$. If $\Lambda \in \N$, then in fact $\gamma \ge 2$.
\item The sum can be taken so that $N(\Lambda) \in \N$ is the smallest integer with $\lambda_{N(\Lambda)} > \Lambda^2$. 
\end{enumerate}
\end{thm}
\begin{proof}
 We know by standard elliptic theory that $u$ is smooth in $\R^n \setminus \R^d$ and satisfies $Lu =0$ in the classical sense away from $\R^d \subset \R^n$. 
It is useful to use the cylindrical coordinates $(x, r, \omega)$, where $r = \abs{t}$ and $\omega = t/\abs{t}$. In these coordinates, for every $x, r$ fixed, we may expand $u(x, r, \cdot)$ in spherical harmonics (since the direct sum of spherical harmonics of differing degrees is $L^2(\S^{n-d-1}, \sigma)$):
\begin{align}\label{eqn:u_sph_expp}
u(x, r, \omega) & = \sum_{j=0}^\infty a_j(x, r) \phi_j(w),
\end{align}
where as in the statement of the Theorem, $\phi_j$ is a spherical harmonic of degree $j$ associated to the eigenvalue $\lambda_j = j(j+n-d-2)$:
\begin{align*}
-\Delta_{\S^{n-d-1}} \phi_j = \lambda_j \phi_j.
\end{align*}
Of course the coefficients $a_j$ are explicit:
\begin{align*}
a_j(x,r) & = \int_{\S^{n-d-1}}u(x, r, \omega) \phi_j(\omega) \; d\sigma(\omega),
\end{align*}
so that since $u$ is homogeneous of degree $\Lambda$ and smooth in the variables $x \in \R^d$ and $r$ for $r >0$, the same holds true for $a_j(x,r)$. Hence $a_j \in C^\infty(\R^{d+1}_+)$ and vanish continuously on $\partial \R^{d+1}_+$. 

{\it A priori}, the convergence in \eqref{eqn:u_sph_expp} is in the $L^2$ sense, but by the smoothness of $u$, we actually know more: inside compact subsets of $\R^n \setminus \R^d$, $a_j(x,r)$ decay faster than any polynomial in $j$ (see Theorem 5 in \cite{seeley_66}). Because of this, the expansion \eqref{eqn:u_sph_expp} converges uniformly, and similarly for all derivatives of the expansion (computed term by term). Hence one may differentiate term by term in $x$, $r$, and $\omega$ in the expansion \eqref{eqn:u_sph_expp} to compute the derivatives of $u$.

Now we re-write the operator $L = - \divv(r^{-n+d+1} \nabla \cdot )$ in cylindrical coordinates, so that outside $\R^d$, we have
\begin{align}
0 = Lu & = -\divv(\abs{t}^{-n+d+1} \nabla u) \nonumber \\
& = (n-d-1) \abs{t}^{-n+d-1} t \cdot \nabla u - \abs{t}^{-n+d+1} \Delta_{x,t} u  \nonumber \\
& =(n-d-1) r^{-n+d-1} t \cdot \nabla u - r^{-n+d+1} \left( \Delta_x u + \partial_r^2 u + \dfrac{n-d-1}{r} \partial_r u + \dfrac{1}{r^2} \Delta_{\S^{n-d-1}} u \right) \nonumber \\
& = -r^{-n+d+1} \left( \Delta_x u + \partial_r^2 u + \dfrac{1}{r^2} \Delta_{\S^{n-d-1}}u \right) , \label{eqn:cyl}
\end{align}
where $\Delta_{\S^{n-d-1}}$ is the spherical Laplacian in the $\omega$ coordinate. Here we have used the fact that $t \cdot \nabla_tu = r \partial_r u$. In particular, $u$ solves (classically) in $\R^n \setminus \R^d$ the equation
\begin{align*}
\Delta_x u + \partial_r^2 u + \dfrac{1}{r^2} \Delta_{\S^{n-d-1}}u =0.
\end{align*}
Applying this operator to the expansion \eqref{eqn:u_sph_expp} term by term, and using the fact that $\phi_j$ are orthogonal, we see readily that the $a_j(x, r)$ satisfy the linear elliptic equation \eqref{eqn:pde_a_j}.

Along with the homogeneity requirement of $a_j$, we can classify all such solutions of \eqref{eqn:pde_a_j}. The simplest case is when $j=0$, since then $\lambda_j = 0$ so $a_0(x,r)$ is a homogeneous harmonic function in $\R^{d+1}_+$ vanishing on $\partial \R^{d+1}_+$, and so we can extend $a_0(x,r)$ to be a homogeneous harmonic polynomial in all of $\mathbb R^{d+1}$. 

Next we move on to the other $\lambda_j$. Let $\ell = \ceil{\Lambda} + \tfrac{d+1}{2}$. Then Proposition \ref{prop:higher_order_moser} applied to $u$ in the ball $B_R = B_R(0)$ implies that 
\begin{align*}
\int_{B_{R/2}} \abs{\nabla_x^{\ell+1} u}^2 \; dm & \lesssim R^{-2(\ell+1)+2} \int_{B_R} \abs{\nabla u}^2 \; dm \\
& \lesssim R^{-2(\ell+1)} \int_{B_{2R}} u^2 \; dm \\
& \lesssim R^{-2 + 2(\Lambda - \ceil{\Lambda}) - d-1} m(B_{2R}) \sup_{B_1} u^2.
\end{align*}
where the second inequality follows from Caccioppoli's inequality \cite[Lemma 8.6]{DFM21AMS}, and the third follows from the fact that $u$ is homogeneous of degree $\Lambda$. It is easy to see that $m(B_R) = \int_{B_R} \abs{t}^{-n+d+1} \; dX$ is homogeneous of degree $d+1$, so $m(B_{2R}) \le C R^{d+1}$, and we have thus shown
\begin{align*}
\int_{B_{R/2}} \abs{\nabla_x^{\ell+1} u}^2 \; dm \le R^{-2} \sup_{B_1} u^2.
\end{align*}
Letting $R \ra \infty$ we obtain $\nabla_x^{m+1} u \equiv 0$, and thus this implies for $j \ge 0$ and $r >0$, 
\begin{align*}
\nabla_x^{\ell+1} a_j(x,r) \equiv 0, 
\end{align*}
or in other words, for each $r$ fixed, $a_j(x,r)$ is a polynomial of degree at most $\ell$. Thus we may expand even further,
\begin{align*}
a_j(x,r) & = \sum_{k=0}^\ell b^j_k(x)r^{\Lambda-k}
\end{align*}
where $b_k(x)$ is a homogeneous polynomial of degree $k$. Here we have used the fact that $a_j$ is homogeneous of degree $\Lambda$, so the coefficients in front of $b_k^j(x)$ (which a priori are just \textit{functions} of $r$) must be homogeneous in $r$.  Moreover, notice that since $a_j(x,r)$ are homogeneous of degree $\Lambda$ and vanish on $\partial \R^{d+1}_+$, it must be that $b_k^j(x) \equiv 0$ for $k > \Lambda$, and thus we may write 
\begin{align*}
a_j(x,r) & = \sum_{k=0}^m b^j_k(x)r^{\Lambda-k},
\end{align*}
where $m = \floor{\Lambda}$.

To recap we have proven that the $a_j$ satisfy properties (i) and (ii) in the theorem (except we have not shown that each $a_j$ must either be even or odd). We now turn towards proving each $a_j$ is either even or odd, properties (iii) and (iv) and the claim on the structure of the frequency set, $\mathcal F$.  At this stage, let us consider the case that $\Lambda \not \in \N \cup \{0\}$, so that computing derivatives in $r$ is simpler (since $\partial_r^2(r^{\Lambda - k}) \ne 0$). The case when $\Lambda \in \N \cup\{0\}$ can be considered similarly, though one needs some care in dealing with the terms $k =m$ and $k = m-1$. 

Let $k_j$ be the largest $k \le m$ so that $b_{k}^j(x) \not \equiv 0$. Since $\Delta_{x,r} a_j - (\lambda_j/r^2) a_j = 0$ classically, by comparing like powers of $r$ using the expansion for $a_j$, we obtain the relations 
\begin{align}
\left( (\Lambda - k_j)(\Lambda -k_j-1) -\lambda_j \right) b^j_{k_j}(x) & = 0, \label{eqn:fcond1} \\
\left( (\Lambda - k_j-1)(\Lambda - k_j-2) - \lambda_j \right) b_{k_j-1}^j(x) & = 0, \label{eqn:fcond2}
\end{align}
and similarly, for $0 \le k \le k_j$, 
\begin{align}
\Delta b_{k}^j(x) + \left( (\Lambda - k-2)(\Lambda -k-3) -\lambda_j \right) b^j_{k-2}(x) & = 0, \label{eqn:ccond1}
\end{align}
with the understanding that by definition, $b_{-1}^j = b_{-2}^j \equiv 0$.
Since $b_{k_j}^j \not \equiv 0$, then $(\Lambda - k_j)(\Lambda - k_j-1) = \lambda_j$ and we have $\Lambda \in \mathcal{F}$ as in the theorem statement. Furthermore we note that 
\begin{align*}
(\Lambda - k_j-\ell) (\Lambda-k_j-\ell-1) \ne \lambda_j,
\end{align*}
for any $1 \le \ell \le \Lambda - k_j$. Thus $b_{k_j-1}^j \equiv 0$ from \eqref{eqn:fcond2} and, using the recursive relation \eqref{eqn:ccond1}, this gives $b_{k_j - \ell}^j \equiv 0$ for $\ell$ odd, $1 \le \ell \le k_j$. Finally, remark that since $\lambda_j \ge 1$, then the positive solution $\gamma$ of $\gamma(\gamma-1) = \lambda_j$ satisfies $\gamma \ge \gamma^* = (1 + \sqrt{5})/2$, and so in particular, $\Lambda - k_j \ge \gamma^*$. This shows that the lowest power of $r$ appearing in the expansion of $a_j$ is at least as large as $\gamma^*$. It further implies that $\mathcal{F} \subset [1, \infty)$ as claimed.

Moreover, we know that 
\begin{align*}
\Lambda^2 & \ge (\Lambda - k_j)(\Lambda - k_j-1) \\
& = \lambda_j,
\end{align*}
which implies that the sum in \eqref{eqn:u_sph_expp} is actually finite: whenever $j$ is such that $\lambda_j > \Lambda^2$, $a_j(x, r) \equiv 0$. This completes the proof of the theorem when $\Lambda \not \in \N$. 

When $\Lambda \in \N$ (so that $m = \Lambda$), then remark that again we may compare like powers of $r$ for the equation \eqref{eqn:pde_a_j} for $a_j$. In this case, when comparing the lowest powers of $r$ ($r^{-2}$ and $r^{-1}$) we obtain the relations
\begin{align*}
-\lambda_j b_m^j & \equiv 0 \\
-\lambda_j b_{m-1}^j & \equiv 0,
\end{align*}
from which we see that $k_j$ (defined as above) must satisfy $k_j \le m-2$. The analysis from \eqref{eqn:fcond1} onwards then applies, and in this case, $a_j(x,r)$ is a polynomial in $x$ and $r$, which completes the proof of the Theorem (notice that in this case, in fact, the lowest power of $r$ appearing in the expansion is at least $2$).
\end{proof}

\begin{rmk}\label{rmk:finer_homogen}
Slightly more analysis gives finer information on the expansion of $u$ in Theorem \ref{thm:homogen_sol} whenever $\Lambda \not \in \N$ or $n-d-2$ is even (so in particular the critical co-dimension 2 case). 

If $\Lambda \not \in \N$, then the representation of $u$ consists of a single term, $u = a_j\phi_j$. This can be shown as follows. If there are two solutions of the equation
\begin{align}\label{eqn:ki}
(\Lambda - k_i)(\Lambda - k_i -1) = j_i(j_i+n-d-2)
\end{align}
for $k_i, j_i \in \N$, then from the explicit formula
\begin{align*}
\Lambda -k_i = \dfrac{1}{2}\left(1 + \sqrt{1 + 4(j_i^2 +j_i(n-d-2))},  \right)
\end{align*}
we see that $2(\Lambda - k_1) - 2(\Lambda - k_2)$ is an integer, which happens only when $1 + 4(j_i^2 + j_i(n-d-2))$ are perfect squares. Then in this case, we would see that $2 \Lambda \in \N$. Now since we assume $\Lambda \not \in \N$, then $\Lambda = p/2$ with $p \in \N$ odd. Of course though, we readily see that $(p/2-k)(p/2-k-1)$ is not an integer for any $k \in \N$, from which we obtain a contradiction to the fact that $(\Lambda -k_i)(\Lambda-k_i-1)$ are integers. Altogether this proves the claim that when $\Lambda \not \in \N$, then there do not exist two solutions to \eqref{eqn:ki}, and from the proof of Theorem \ref{thm:homogen_sol} we see that only one of the $a_j$ can be nonzero, proving the claim.

On the other hand, when $\Lambda \in \N$ and $n-d-2$ is even, then we can show that the expansion of $u = \sum_j a_j \phi_j$ only includes $j$ even. This follows from examination of the equation \eqref{eqn:ki}: when $n-d-2$ is even and $\Lambda \in \N$, then the left-hand side is an even integer always, while the right-hand side is even only when $j_i$ is even. In particular, solutions only exist for $j_i$ even, and thus again the proof of Theorem \ref{thm:homogen_sol} shows that $a_j =0$ for all $j$ odd.
\end{rmk}

\begin{rmk}\label{rmk:F_discrete}
	Since $\cal{F}$ is a discrete set, we can write
	\begin{align}
		\cal{F} = \{\Lambda_1, \Lambda_2, \Lambda_3, \dotsc,\}
	\end{align}
	where $\Lambda_k <\Lambda_{k+1}$, and by courtesy of Theorem \ref{thm:homogen_sol}, $\Lambda_1 = 1< 3/2 < \Lambda_2$. We shall use this notation later in Section \ref{sec:pinch}.
\end{rmk}

\subsection{Volume estimates on singular set of homogeneous solutions}
With all homogeneous solutions characterized, we can now provide explicit volume estimates on the tubular neighborhoods of homogeneous solutions to equations closely related to \eqref{eqn:soln}. This will be key in our blowup analysis.  In the case of homogeneous harmonic polynomials, the quantitative estimates follow from algebraic geometric considerations (see Theorem 2.1 in \cite{HS89}). Here we use similar ingredients but with additional complications due to the fact that our solutions are not necessarily polynomials.

First, let us introduce the class of operators we consider, which shall appear later as blowups of equations of the type \eqref{eqn:op2} for matrices $A$ satisfying a certain smoothness condition (see Definition \ref{defn:c1a}).

\begin{defn}\label{defn:sep}
We say that a matrix $\A$ separates $x$ and $t$ if it takes the form
\begin{align*}
\A = \begin{pmatrix}
J & 0 \\
0 & h I_{n-d}
\end{pmatrix}
\end{align*}
where $J$ is a $d \times d$ matrix, $h \in \R$, and $I_{n-d}$ is the identity.
\end{defn}

Notice that the following space $H_\Lambda^\lambda$ depends only on parameters $\Lambda$ and $\lambda$, and not the particular matrix $\A_0$.
\begin{defn}\label{defn:homogen_solns}
For each $\Lambda >0$ and $\lambda >0$, denote by $H_\Lambda^\lambda$ the linear space of $\Lambda$-homogeneous solutions $u \in W_r(\R^n) \cap C(\R^n)$ which vanish continuously on $\R^d$ and satisfy 
\begin{align}\label{eqn:soln_const}
-\divv(\abs{t}^{-n+d+1} \A_0 \nabla u) = 0 \text{ in } \R^n,
\end{align}
for some constant matrix $\A_0$ which is uniformly elliptic with constant $\lambda > 0$ and separates $x$ and $t$.
\end{defn}

\begin{rmk}\label{rmk:h_compt}
Notice that for each $\lambda, \Lambda >0$, $H_\Lambda^\lambda \subset C(B_1)$ is a compact subset. Indeed we shall shortly see in the proof of Theorem \ref{thm:homogen_m} that each $\Lambda$-homogeneous solution of \eqref{eqn:soln_const} (after a constant change of variables through the matrix $\A_0$ in \eqref{eqn:soln_const}) coincides with a $\Lambda$-homogeneous solution of \eqref{eqn:soln}, which by Theorem \ref{thm:homogen_sol}, are contained in a finite dimensional subspace of $C(B_1)$. Since the space of such matrices $\A_0$ which are uniformly elliptic with constant $\lambda >0$ is also compact, one deduces that $H_\Lambda^\lambda$ is compact. Of course one could also deduce such compactness using Lemma \ref{lem:compact_soln} later.
\end{rmk}

Recall that for $u \in L^2(dm)$ and $X_0$ such that 
\begin{align*}
\limsup_{r \ra 0} \fint_{B_r(X_0)}u^2 \; dm = 0,
\end{align*}
we define the vanishing order of $u$ at $X_0$, $p$, to be the quantity
\begin{align*}
p \coloneqq \sup\{ k \ge 0 \; : \; \limsup_{r \da 0} r^{-k} \left(\fint_{B_r(X_0)} u^2 \; dm \right)^{1/2} = 0\}.
\end{align*}
Remark that when $u$ is smooth in a neighborhood of $X_0$, then Taylor's Theorem implies that the order vanishing of $u$ at $X_0$ is simply the smallest $k \in \N$ so so that $\nabla^k u(X_0) \ne 0$.

With this notation, let us define for any $u \in W_r(\Omega)$ the strata $\cal{O}^p(u)$ and the singular set of $u$, $\sing(u)$ by
\begin{align*}
\mathcal{O}^p(u) & \coloneqq \{ X \in \R^n \; : \; \text{ $u$ vanishes to order at least $p$ at } X\}, \qquad  \sing(u)  \coloneqq \mathcal{O}^{3/2}(u).
\end{align*} 
Away from the boundary solutions to \eqref{eqn:soln_const} are smooth so we can define the usual singular set 
\begin{align}\label{eqn:sing_eq}
\sing(u) \cap \Omega = \{ X \in \Omega \; : \;  u(X) = \nabla u(X) = 0\}
\end{align}
for $\Omega = \R^n \setminus \R^d$. 

With our notation explained, we now prove volume estimates on the singular set of homogeneous solutions of \eqref{eqn:soln_const} with explicit and optimal dependence on the degree of homogeneity. First we need an estimate on the volume of tubular neighborhoods of real semi-algebraic varieties. The following can be seen as a Corollary of the proof of the main Theorem in \cite{Wongkew93}. In what follows, whenever $A \subset \R^n$, we write
\begin{align*}
B_r(A) \coloneqq \{ X \in \R^n \; : \; \dist(x, A) < r \} 
\end{align*}
for the $r$-neighborhood of $A$, and $\dim_{\mathrm{H}}$ denotes Hausdorff dimension.

\begin{thm}[\cite{Wongkew93}, see also Theorem 5.9 in \cite{YC04}]\label{thm:semialg_vol}

There is a dimensional constant $C_n >0$ depending only on $n \in \N$ such that that following holds. Suppose that $p_1, \dotsc, p_m$ are real polynomials in $\R^n$, $d = \deg p_1 + \cdots + \deg p_m$, and that the set 
\begin{align*}
V \coloneqq \{ p_1(x) \ge 0, \dotsc, p_m(x) \ge 0\} \subset \R^n
\end{align*}
is such that $\dim_{\mathrm{H}} V = k \in \{0\} \cup \N$. Then, for any $\epsilon \in (0,1)$, we have
\begin{align*}
\L^n \left( B_\epsilon(V) \cap  B_2(0)  \right)  \le C_n (\epsilon d)^{n-k}.
\end{align*}
\end{thm}

To be clear, Wongkew proves an estimate for real algebraic varieties, though the proof really depends only on Milnor's estimate on the Betti numbers of a real algebraic variety, which bounds the number of its connected components. However, this result also holds for semi-algebraic varieties: if $V = \{ p_1(x) \ge 0, \dotsc, p_m(x) \ge 0 \} \subset \R^n$ is a real semi-algebraic variety and $d = \deg p_1 + \dotsc + \deg p
_m$, then the sum of the Betti numbers of $V$ (and hence the number of connected components of $V$) is at most $C_n d^n$ \cite[Theorem 3]{M64}. One can also see a formulation similar to Theorem \ref{thm:semialg_vol} in Section 5 of \cite{YC04}.

\begin{thm}\label{thm:homogen_m}
Let $u \in H_\lambda^\Lambda \setminus \{0\}$, so that $u$ is a non-trivial $\Lambda$-homogeneous solution to \eqref{eqn:soln_const} that vanishes continuously on $\R^d$, with $\Lambda \ne 0$. Then $\Lambda \in \mathcal{F}$, and there exists a universal constant $C >0$ depending only on the dimensions $n$ and $d$ and the ellipticity $\lambda$ so that for any $\epsilon \in (0,1)$,
\begin{align}\label{eqn:homogen_m}
\L^n( B_\epsilon(\sing(u)) \cap B_1) + \HD^{n-1}( B_\epsilon(\sing(u)) \cap \partial B_1) \le C \epsilon^2 \Lambda^2,
\end{align}
and thus $\HD^{n-2}\left( \sing(u) \cap B_1(0) \right) \le C \Lambda^2.$
\end{thm}
\begin{proof}
It suffices to assume that $u \in H_\Lambda^\lambda$ solves \eqref{eqn:soln}; the more general case follows via a standard change of variables. To give a bit more detail, if a homogeneous $u$ solves an equation of type \eqref{eqn:soln_const} then $v(X) \coloneqq u(\A_0^{1/2}X)$ satisfies \eqref{eqn:soln} and vanishes continuously on $\R^d$ with the same homogeneity, since $\A_0$ separates $x$ and $t$. Furthermore the change of coordinates $x\mapsto \A_0^{1/2}X$ is bi-Lipschitz with non-vanishing gradient, so it maps the singular set to the singular set and increases measure by at most a multiplicative constant (depending on $\lambda$).

Let $u$ be a non-trivial $\Lambda$-homogeneous solution to \eqref{eqn:soln}, so that by Theorem \ref{thm:homogen_sol}, $\Lambda \in \mathcal{F}$. Notice that since $\Lambda \in \mathcal{F}$ is nonzero, then $\Lambda \ge 1$. Our main estimate shall be on singular points outside $\R^d$, so define for any $v$,
\begin{align*}
\tilde{\sing}(v) \coloneqq \sing(v) \cap (\R^n \setminus \R^d).
\end{align*}
As in the proof of Theorem \ref{thm:homogen_sol}, let us adopt the notation $r = \abs{t}$ and $\omega = t/\abs{t}$. Then for any $\gamma >0$, notice that we have the set equality
\begin{align}
\tilde{\sing}(u) = \tilde{\sing}(r^\gamma u), \label{eqn:stilde_r}
\end{align}
simply because for $r \ne 0$, $u$ and $r^\gamma u$ have the same order of vanishing at any fixed point $X \not \in \R^d$. 

Define the transformation on functions, $T$, by $Tv(x,r, \omega) \coloneqq r^{2 ( N(\Lambda) + \ceil{\Lambda} - \Lambda)} v(x, r^2, \omega)$. This transform will be important for homogeneous solutions because it allows us to relate their singular sets to those of polynomials. One easily verifies that for $v$ smooth in $\R^{n} \setminus \R^d$, \eqref{eqn:sing_eq} and \eqref{eqn:stilde_r} imply that
\begin{align}\label{eqn:stilde_T}
\tilde{\sing}(v) & = F(\tilde{\sing}(Tv)),  
\end{align}
where $F: \R^n \ra \R^n$ is defined by $F(x,r,\omega) = (x,r^2, \omega)$. Notice that $F$ is locally bi-Lipschitz in $\R^n \setminus \R^d$, and in fact, $\norm{\nabla F}_{L^\infty(B_{1}(0))} \lesssim 1$. Along with \eqref{eqn:stilde_T}, such properties of $F$ guarantee that for any $v$ smooth in $\R^n \setminus \R^d$,
\begin{align*}
\L^n( B_\epsilon(\tilde{\sing}(v)) \cap B_1(0)) & = \L^n( B_\epsilon(F (\tilde{\sing}(Tv))) \cap B_1(0)) \\
& \le C \L^n( B_{C \epsilon} (\tilde{\sing}(Tv) ) \cap B_2(0))) 
\end{align*}
for some dimensional $C >0$. 
In view of this estimate, to show the first inequality in \eqref{eqn:homogen_m} it suffices now to show that for our solution $u$, 
\begin{align}\label{eqn:sing_red}
\L^n(  B_\epsilon(\tilde{\sing}(Tu)) \cap B_2(0)) \lesssim \epsilon^2 \Lambda^2,
\end{align}
simply because $\L^n(B_\epsilon(\sing(u) \cap \R^d) \cap B_1(0)) \le \L^n( B_\epsilon(\R^d) \cap B_1(0)) \lesssim \epsilon^2 \le \epsilon^2 \Lambda^2$.  We now turn to the proof of \eqref{eqn:sing_red}.

By Theorem \ref{thm:homogen_sol}, we know that $u$ admits the expansion
\begin{align*}
u(x,r,\omega) &= \sum_{j=0}^{N(\Lambda)} a_j(x,r) \phi_j(\omega).
\end{align*}
Decompose $u = u_e + u_o$ where 
\begin{align*}
u_e(x,r,\omega) \coloneqq \sum_{\substack{j=0, \\ j \text{ even }}}^{N(\Lambda)} a_j(x,r) \phi_j(\omega), \qquad 
u_0(x,r, \omega) \coloneqq \sum_{\substack{j=1, \\ j \text{ odd }}}^{N(\Lambda)}a_j(x,r) \phi_j(\omega),
\end{align*}
and notice that $u_e$ and $u_0$ are $\Lambda$-homogeneous solutions of equation \eqref{eqn:soln}, and so is $u_e - u_0$. Since $T$ is linear, we have that $Tu = Tu_e + Tu_o$. Now we further decompose
\begin{align*}
\tilde{\sing}(Tu) = \tilde{\sing}_1 \cup \tilde{\sing}_2,
\end{align*}
where $\tilde{\sing}_1 = \tilde{\sing}(Tu) \cap \{Tu_e =0\}$ and $\tilde{\sing}_2 = \tilde{\sing}(Tu) \setminus \tilde{\sing}_1$. 

To deal with $\tilde{\sing}_1$, recall that $\mathcal{O}^p$ denotes the points with vanishing order at least $p$ and observe 
\begin{align*}
\tilde{\sing}_1 & \subset \mathcal{O}^3 ( (Tu_e)^2 - (Tu_o)^2))  \\
& \subset \R^d \cup \tilde{\sing}(Tu_e + Tu_o) \cup \tilde{\sing}(Tu_e - Tu_o). 
\end{align*}
Indeed, if $X \in \tilde{\sing}_1$, then by definition $Tu = T u_e + Tu_o$ vanishes to order at least $2$ at $X$, and if $X\in \{Te_e = 0\}$ then $Tu_e(X) - Tu_0(X) = 0$. This proves the first inclusion. As for the second, if $(Tu_e)^2 - (Tu_o)^2$ vanishes to order $3$ at $X \not \in \R^d$, then at least one of $Tu_e + Tu_o$ or $Tu_e - Tu_o$ must vanish to order at least $2$ at $X$. 

Now since $u$ is a solution to \eqref{eqn:soln}, a second order elliptic equation with smooth coefficients away from $\R^d$, the fact that $T$ is Lipschitz gives that $\dim_{\mathcal{H}}(\tilde{\sing}(Tu))) = \dim_{\mathcal{H}} (\tilde{\mathcal S}(u)) \le n-2$ (see for example, Chapter 4 of \cite{HLNODAL}). In particular, we have from the inclusion above,
\begin{align} \label{eqn:dim_O}
\dim_{\mathcal{H}} \mathcal{O}^3((Tu_e)^2 - (Tu_o)^2) \le n-2.
\end{align}
In addition, notice that 
\begin{align*}
\left( Tu_e(x,r,\omega) \right)^2 & = \left( \sum_{\substack{j=0 \\ j \text{ even }}}^{N(\Lambda)} r^{2 (N(\Lambda) + \ceil{\Lambda}  - \Lambda)} a_j(x,r^2) \phi_j(\omega) \right)^2 \\
& = \left( \sum_{\substack{j=0 \\ j \text{ even }}}^{N(\Lambda)} r^{2 N(\Lambda) -j}  r^{2\ceil{\Lambda} - 2 \Lambda }a_j(x,r^2) r^j \phi_j(\omega) \right)^2
\end{align*}
is an honest polynomial in $X = (x,t)$, since by Theorem \ref{thm:homogen_sol}, $r^{2\ceil{\Lambda} - 2\Lambda} a_j(x,r^2)$ is a polynomial in $x, r^2 = \abs{t}^2$, and $r^j \phi_j(\omega)$ is a homogeneous harmonic polynomial of degree $j$ in $t$ since $\phi_j$ is a spherical harmonic in $\omega = t/\abs{t}$. It is also easy to verify on the condition of $N(\Lambda)$ from Theorem \ref{thm:homogen_sol} that the degree of this polynomial is $\lesssim \Lambda$. Similarly, $(Tu_o)^2$ is a polynomial in $X$ of degree $\lesssim \Lambda$. Thus along with \eqref{eqn:dim_O}, Theorem \ref{thm:semialg_vol} applies to $\mathcal{O}^3((Tu_e)^2 - (Tu_o)^2)$, which is the nodal set of a polynomial of degree $\lesssim \Lambda$ in $\R^n$ to give that 
\begin{align*}
\L^n ( B_\epsilon(\tilde{\mathcal S}_1) \cap B_2(0)) & \le \L^n( B_\epsilon(\mathcal{O}^3((Tu_e)^2 - (Tu_0)^2)) \cap B_2(0)) \\
& \lesssim \epsilon^2 \Lambda^2.
\end{align*}
This concludes the measure estimate for $B_\epsilon(\tilde{\mathcal S}_1) \cap B_2(0)$. 

The set $\tilde{\mathcal S}_2$ is slightly trickier, but we still have the inclusions
\begin{align*}
\tilde{\mathcal S}_2 & \subset \bigcup_{m \in \N} \left [ \sing( (Tu_e)^2 - (Tu_o)^2 )  \cap \{ (Tu_e)^2 \ge m^{-1} \}    \right ] \\
& \subset \R^d \cup \tilde{\sing}(Tu_e + Tu_o) \cup \tilde{\sing}(Tu_e - Tu_o).
\end{align*}
Indeed for the first inclusion, if $X \in \tilde{\mathcal S}_2$, then the same argument as for $\tilde{\mathcal S}_1$ shows $(Tu_e)^2 - (Tu_o)^2$ vanishes to order at least $2$ at $X$. Since $X \not \in Z(Tu_e)$ (i.e. the nodal set of $Tu_e$), then $X$ belongs to the union above. As for the second inclusion, suppose that $X \not \in \R^d$ has
\begin{align*}
X \in \bigcup_{m \in \N} \left [ \sing( (Tu_e)^2 - (Tu_o)^2 )  \cap \{ (Tu_e)^2 \ge m^{-1} \}    \right ].
\end{align*}
Then we have $Tu_e(X) \pm Tu_o(X) =0 $, and thus
\begin{align*}
0& = \nabla ((Tu_e)^2 - (Tu_o)^2)(X) \\
& = 2 Tu_e(X) (\nabla Tu_e(X) \pm \nabla Tu_o(X)).
\end{align*}
Since $Tu_e(X) \ne 0$, then $X \in \tilde{\sing}(Tu_e \pm Tu_o)$, showing the inclusion.

Now for any fixed $m$, just as for the argument for $\tilde{\mathcal S}_1$ we have
\begin{align*}
\dim_{\mathcal{H}} \sing((Tu_e)^2 - (Tu_o)^2) \cap \{(Tu_e)^2 \ge m^{-1} \} \le n-2.
\end{align*}
Using that $(Tu_e)^2$, $(Tu_0)^2$ are polynomials of degree $\lesssim \Lambda$, then we apply Theorem \ref{thm:semialg_vol} this time along with the increasing set theorem to conclude that 
\begin{align*}
\L^n ( B_\epsilon(\tilde{\mathcal S}_2 ) \cap B_2(0)) & \le \limsup_{m \ra \infty} \L^n( B_\epsilon(\sing((Tu_e)^2 - (Tu_o)^2) \cap \{(Tu_e)^2 \ge m^{-1}\}) \cap B_2(0)) \\
& \lesssim  \epsilon^2 \Lambda^2.
\end{align*}
Combining our estimates on $B_\epsilon(\tilde{\mathcal S}_1) \cap B_2(0)$ and $B_\epsilon(\tilde{\mathcal S_2}) \cap B_2(0)$ gives \eqref{eqn:sing_red}, completing the proof of the first estimate in \eqref{eqn:homogen_m}. The second estimate on $\HD^{n-1}(B_\epsilon(\sing(u)) \cap \partial B_1)$ is a quick consequence of the first estimate, the coarea formula, and the fact that $\sing(u)$ is scale-invariant since $u$ is $\Lambda$-homogeneous. We omit the details.
\end{proof}

\begin{rmk}
The proof of Theorem \ref{thm:homogen_m} becomes simpler when we have finer information on the expansion of the homogeneous solution $u = \sum_j a_j \phi_j$, such as the first two situations outlined in Remark \ref{rmk:finer_homogen}. Indeed in either of the two cases outlined there, it is much easier to find a polynomial $P$ for which $\tilde{\sing}(u) \subset \tilde{\sing}(P)$ without having to deal with squaring some part of the solution $u$: in fact $Tu$ or $r Tu$ is a polynomial in $X$ in either of these cases. In general, when $n-d-2$ is odd, though, the terms $a_j$ in the expansion of $u$ can be all nonzero, and such simplifications do not hold, and we are forced to consider the separation of $u= u_e + u_o$ as above.
\end{rmk}

\section{Solutions to the ``Lipschitz coefficient'' equations} \label{sec:lip}
In this section, we discuss monotonicity formulae for solutions to ``variable'' coefficient degenerate elliptic equations in the flat space $\Omega = \R^n \setminus \R^d$. In strong analogy with the case of uniformly elliptic equations, we require that the coefficients satisfy an appropriate Lipschitz condition as detailed in Definition \ref{defn:c1a}. In the end, the conjugated elliptic operator $\tilde{L}_\beta$ obtained from the operator $L_\beta := -\divv(D_{\beta}^{-n+d+1} \nabla \, \cdot)$ defined outside of a $C^{1,1}$, $d$-dimensional graph $\Gamma \subset \R^n$ will satisfy this condition, where $\rho$ is an appropriate change of variables flattening $\Gamma$ (see Theorem \ref{thm:change_var})\footnote{{Similarly if $L_\infty$ is conjugated by $\rho$ the resulting $\tilde{L}_\infty$ will satisfy these conditions as long as $\Gamma$ is $C^2$}} . One of the key pieces of this notion of regularity is the asymptotic block-type structure of $A$ near $\R^d$; when $A$ satisfies the $C^{0,1}$ condition we introduce, then its values can be made sense pointwise on $\R^d$, and these matrices separate the variables $x$ and $t$ (recall Definition \ref{defn:sep}).

\begin{defn}\label{defn:c1a}
We say that the matrix $A$ satisfies the higher co-dimensional $C^{0,1}$ condition in $B_R(X_0)$ (with constants $\lambda, C_0 > 0$) provided that $A \in C^{0,1}_{\loc}(B_R(X_0) \cap (\R^n \setminus \R^d))$ is symmetric, $\delta^{n-d-1}A$ is uniformly elliptic with constant $\lambda$, the estimate
\begin{align*}
\abs{ \nabla A(X)} \le C_0 \delta(X)^{-1}, \; \; X \in B_R(X_0) \cap \Omega
\end{align*}
holds, and there is a $d \times d$ uniformly elliptic symmetric matrix $J(x,t)$ with constant $\lambda$ and a scalar function $\lambda^{-1} \le h(x,t) \le \lambda$ so that for $\A = \delta^{n-d-1}A$ and $\mathcal{B}(X)$ defined by 
\begin{align} \label{eqn:B}
\mathcal{B}(x,t) \coloneqq \left ( \begin{array}{cc}
J(x,t) & 0 \\
0 & h(x,t) I_{n-d}
\end{array} \right ),
\end{align}
we have $\A, \mathcal{B} \in C^{0,1}(B_R(X_0))$ with the following estimates in $B_{R}(X_0)$:
\begin{align*}
\abs{\nabla \mathcal{A}} + \abs{\nabla J} + \abs{\nabla h} & \le C_0  \\
\abs{ \A - \mathcal{B}} & \le C_0 \delta.
\end{align*}
\end{defn}

{Note that in the co-dimension one setting, in the domain $\{(x', x_n)\mid x_n > \phi(x')\}$ for $\phi \in C^{1,1}(\mathbb R^{n-1})$, if the Laplacian is conjugated by the standard change of variables $(x',x_n) \mapsto (x', x_n - \phi(x'))$, the resulting operator will have the regularity and block structure detailed above. }

\begin{rmk}\label{rmk:A_bdry}
Notice that the estimates in Definition \ref{defn:c1a} imply that we can make sense of the values of $\mathcal{A}$ and $\mathcal{B}$ on $\R^d \cap B_R(X_0)$ in the sense that 
\begin{align*}
\mathcal{A}(Y_0) \coloneqq \lim_{r \da 0} \fint_{B_r(Y_0)} \A \; dX \equiv  \lim_{r \da 0} \fint_{B_r(Y_0)} \mathcal{B} \; dX \eqqcolon \mathcal{B}(Y_0)
\end{align*}
exists for $Y_0 \in \R^d \cap B_R(X_0)$. In addition, the appearance of the matrix $\mathcal{B}$ is in some sense cosmetic, since if $A$ satisfies the higher co-dimensional $C^{0,1}$ condition with respect to a matrix $\mathcal{B}$, then $A$ also satisfies the same condition with $\tilde{\mathcal{B}}(x,t) \coloneqq \A(x,0)$ in place of $\mathcal{B}$. However, when proving a matrix satisfies the higher co-dimensional $C^{0,1}$ condition, it is still sometimes useful to opt for the definition we choose above.
\end{rmk}

In the remainder of this section, we make the standing assumption that $u \in W_r(B_{10}) \cap C(B_{10})$ is a solution of 
\begin{equation}
\left \{
\begin{aligned}\label{eqn:soln_var}
 -\divv( A \nabla u ) & = 0, && \text{in } B_{10}, \\
 u & = 0, && \text{on } \R^d \cap B_{10},
\end{aligned} \right .
\end{equation}
and in addition that 
\begin{equation}\label{eqn:A_cond}
\text{$A$ satisfies the higher co-dimension $C^{0,1}$ condition in $B_{10}$ with constants $\lambda, C_0 >1$}.
\end{equation}

At some points, it will be convenient to assume that one has better control on the constants $\lambda$, $C_0$, (which will automatically be true locally at least after a linear change of variables) in the sense that the matrix $\A = \delta^{n-d-1} A$ satisfies for $\eta \in (0,1/2)$ sufficiently small,
\begin{align}\label{eqn:ellipse_cond}
 B_{(1-\eta)} \subset \A^{1/2}(X) B_{1} \subset B_{(1 + \eta)} \text{ and } \abs{\A(X) - \A(Y)} \le \eta \abs{X-Y} \text{ for all } X, Y \in B_{10}.
\end{align}
Of course the two conditions above automatically imply that in addition,
\begin{align}\label{eqn:ellipse_cond_cons}
\abs{\A(X)^{1/2} - \A(Y)^{1/2}} + \abs{\A(X)^{-1/2} - \A(Y)^{-1/2}} \le C \eta
\end{align}
for some universal constant $C > 0$ depending only on the dimension $n$. From Remark \ref{rmk:A_bdry}, it is easy to check that \eqref{eqn:A_cond} and \eqref{eqn:ellipse_cond} then imply that 
\begin{equation}\label{eqn:A_cond_prime}
\text{$A$ satisfies the higher co-dimension $C^{0,1}$ condition in $B_{10}$ with constants $1+ C\eta, \; \eta$},
\end{equation}
which in particular, are bounded constants, since $\eta < 1/2$.

\subsection{Monotonicity on the boundary}
 For solutions of \eqref{eqn:soln_var} we define the frequency function (inspired by \cite{Tao_02, TZ_08, GSVG_14, Yu_unique}, and the earlier work of \cite{GL_86}). Notice that $\A(X_0)$ is a positive, symmetric matrix taking the block form as in Definition \ref{defn:c1a}, so that $\A(X_0)^{1/2}$ and $\A(X_0)^{-1/2}$ are well-defined, positive symmetric matrices.

\begin{defn}\label{defn:freq_var_gen}
Suppose that $ X_0 \in \Gamma \coloneqq \R^d  \subset \R^n$ and $u \in W_r(B_R(X_0))$ is a solution to \eqref{eqn:soln_var} in $B_R(X_0)$ that vanishes continuously on $\Gamma \cap B_R(X_0)$. Suppose in addition that $A$ satisfies the higher co-dimension $C^{0,1}$ condition in $B_R(X_0)$, so that $\A_0 := \A(Y_0)$ exists for any $Y_0 \in B_R(X_0)$ by Remark \ref{rmk:A_bdry}. With the ellipsoid $E_r^{\A_0}(Y_0)= E_r(Y_0)$ defined by \[E_r(Y_0) \coloneqq \A_0^{1/2}B_r +Y_0 \equiv  \left \{ \dotp{ \A_0^{-1}(X-Y_0), (X-Y_0)  }< r^2 \right \} \equiv \left \{ \abs{  \A_0^{-1/2}(X-Y_0)  } < r  \right \},  \] we define the modified Almgren frequency $N(r) = N_u^A(Y_0, r)$ for $u$ and $0 < r < R$ small enough so that $E_r(Y_0) \subset B_R(X_0)$ by 
\begin{align}\label{eqn:N_var_gen}
N(r) \coloneqq \dfrac{r \int_{E_r(Y_0)} \dotp{A \nabla u, \nabla u} \, dY}{ \int_{\partial E_r(Y_0)} \mu_{Y_0} u^2 \tfrac{ r }{ \abs{ \A_0^{-1}(Y-Y_0) } } \, d \sigma(Y)}  \eqqcolon\dfrac{ r D_u^A(r)}{H_u^A(r)}.
\end{align}
Here $\mu_{Y_0}:= \mu$ is the scalar function defined by 
\begin{align*}
\mu(Y) \coloneqq \abs{ \A_0^{-1/2} (Y - Y_0)}^{-2} \dotp{ A(Y) \A_0^{-1}(Y - Y_0), \A_0^{-1}(Y - Y_0)   }.
\end{align*}
\end{defn}

\begin{rmk}\label{rmk:freq_id}
Notice in Definition \ref{defn:freq_var_gen} that when $\A_0 = \A(Y_0) = I$, is the identity matrix, we have the simpler formula
\begin{align}\label{eqn:freq_var_gen_i}
N(r) = \dfrac{r \int_{B_r(Y_0)} \dotp{A \nabla u, \nabla u} \, dY}{ \int_{\partial B_r(Y_0)} \mu u^2  \, d \sigma(Y)} ,
\end{align}
where $B_r(Y_0)$ is the usual round Euclidean ball. Since this formula is easier to work with, we shall often assume (after a suitable change of variables) that $\A_0 = I$.
\end{rmk}

Before proving the almost monotonicity of $N$ when $Y_0 \in \R^d$, we provide several useful estimates on matrices satisfying the higher co-dimensional $C^{0,1}$ condition. To ease notation in what follows, we frequently write $ A(X)X \equiv AX$ when the argument of $A$ is unambiguous.
\begin{lemma}\label{lem:c1_est}
Suppose that $A$ satisfies the higher co-dimensional $C^{0,1}$ condition in $B_{10}(0)$ with constants $\lambda, C_0 >1$ and that in addition $\A(0) = I$. Then  $\mu_0(X) = \mu(X) = \dotp{ A X/\abs{X}, X/\abs{X}}$ and we define
\begin{align*}
 \beta(X) \coloneqq A X/\mu.
\end{align*}
The following estimates hold in $B_{10}(0)$:
\begin{align}
\abs{ \delta^{n-d-1} \mu(X) - 1} & \le C \abs{X} \label{eqn:asm_1} \\
\abs{ \nabla (\delta^{n-d-1} \mu(X)) } & \le C  \label{eqn:asm_1b} \\
\abs{ \nabla (\delta^{n-d-1} \mu(X))^{-1} } & \le C \label{eqn:asm_1c}\\
\abs{ \A(X) - \A(\pi(X))} & \le C \delta(X) \label{eqn:asm_2} \\
\abs{ \dotp{\beta(X), \nabla \delta(X)}} & \le C  \delta(X) \label{eqn:asm_3} \\
\abs{ D\beta(X) - I} & \le C \abs{X},  \label{eqn:asm_4} \\
\abs{ \dotp{\beta(X) - X, \nabla \delta} } & \le C \abs{X} \delta(X) \label{eqn:asm_4b}
\end{align}
with $C = C(\lambda, C_0)$. Recall that $\pi(X)$ denotes the closest point on $\mathbb R^d$ to $X$. 
\end{lemma}

\begin{proof}
Each of these is a straight-forward computation. For the first, for example, we note that 
\begin{align*}
\delta^{n-d-1}\mu (X) - 1 & = \dotp{ \A(X)  X/\abs{X}, X/\abs{X}} - 1\\
& = \dotp{ (\A(X) - \A(0)) X/\abs{X}, X/\abs{X}},
\end{align*}
and use the Lipschitz nature of $\A$. Notice from the ellipticity assumption on $\A$, that $\delta^{n-d-1} \mu \simeq 1$ uniformly on $B_{10}(0)$. The estimate \eqref{eqn:asm_1b} then follows from our observation above that
\begin{align*}
\delta^{n-d-1} \mu(X) = 1 + \dotp{ (\A(X) - \A(0)) X/\abs{X}, X/\abs{X}},
\end{align*}
so differentiating and once again using that $\abs{\A(X) - \A(0)} \le C_0 \abs{X}$ and $\abs{\nabla \A} \le C_0$ gives \eqref{eqn:asm_1b}. Notice that \eqref{eqn:asm_1c} follows immediately as well since $\delta^{n-d-1}\mu \simeq 1$. The fourth inequality \eqref{eqn:asm_2} is an immediate consequence of the Lipschitz nature of $\A$. As for the fifth, we write
\begin{align*}
\beta(X) = \dfrac{\A(X)X}{\delta(X)^{n-d-1} \mu(X)} & = \dfrac{\A(\pi(X))X}{\delta(X)^{n-d-1} \mu(X)} + \dfrac{(\A(X) - \A(\pi(X))X}{\delta(X)^{n-d-1} \mu(X)},
\end{align*}
and use \eqref{eqn:asm_2} to estimate the inner product of the second term with $\nabla \delta$ since $\abs{\nabla \delta} \le 1$. In the first term, we use the fact that $\A(\pi(X)) = \mathcal{B}(\pi(X))$ has the structure as in \eqref{eqn:B}, and thus $\dotp{\A(\pi(X)), \nabla \delta(X))} \lesssim \delta(X)$. This brings us to the last estimates, \eqref{eqn:asm_4} and \eqref{eqn:asm_4b}.

To obtain these inequalities, we write
\begin{align*}
\beta(X) & = \A \left(\dfrac{1}{\delta^{n-d-1} \mu} - 1 \right)X +  \left( \A - \mathcal{B} \right)X + \mathcal{B} X,
\end{align*}
and differentiate term by term to get the estimates
\begin{align*}
\abs{D \left( (\A \left(\dfrac{1}{\delta^{n-d-1} \mu} - 1 \right)X +  \left( \A - \mathcal{B} \right)X  \right) } \lesssim \abs{X},
\end{align*}
by virtue of the fact that $\A$, $\mathcal{B}$ are Lipschitz, estimates \eqref{eqn:asm_1} and \eqref{eqn:asm_1c}, and that $\abs{ \A -\mathcal{B}} \le C_0 \delta \le C_0 \abs{X}$. As for the last term appearing in our decomposition of $\beta(X)$, we get the same estimate as above whenever $\partial_{X_i}$ hits a coefficient of $\mathcal{B}(X)$, and thus we see the the main contribution of $\partial_{X_i} \beta_k(X)$ comes from $\sum_j b_{jk} (X_j)_{X_i} = b_{ik}$, where $\mathcal{B} = (\ol{b}_{ij})$. Altogether this gives us that 
\begin{align*}
D \beta(X) = \mathcal{B}(X) + O(\abs{X}) = I + O(\abs{X}),
\end{align*}
since $\mathcal{B}(0) = \A(0) = I$.

Finally, the last inequality \eqref{eqn:asm_4b} is an improvement of \eqref{eqn:asm_3} obtained from from integrating \eqref{eqn:asm_4}. Indeed, we can write
\begin{align} \label{eqn:bx}
\dotp{\beta(X) - X, \nabla \delta(X)} &  = \dotp{ \left( \beta(X) - \beta(\pi(X)) \right) - (X - \pi(X)), \nabla \delta(X)}
\end{align}
since $\dotp{ \beta(\pi(X)) - \pi(X), \nabla \delta(X)} = 0$ from the assumption on the structure of the matrix $\A(\pi(X)) = \mathcal{B}(\pi(X))$ as in \eqref{eqn:B}. Next, we simply estimate crudely the right-hand side of \eqref{eqn:bx} using the fundamental theorem of calculus along the segment $[\pi(X), X]$ and \eqref{eqn:asm_4} to obtain \eqref{eqn:asm_4b}.
\end{proof}

With these estimates, we prove almost monotonicity of the frequency $N$.
\begin{lemma}[Almgren monotonicity, variable coefficients]\label{lem:monotone_var}
Assume \eqref{eqn:soln_var}, \eqref{eqn:A_cond} and in addition that $\A(0) = I$. Then for $0 < r < 1$ the frequency $N(r)$ as defined in \eqref{eqn:N_var_gen} is almost monotone in the sense that there is a constant $C = C(\lambda, C_0)$ for which $N(r) e^{Cr}$ is monotone increasing on $(0, 1)$. In particular, the limit $N(0^+) \coloneqq \lim_{r \da 0 } N(r)$ exists.
\end{lemma}

\begin{proof}

By definition of $N$, we compute
\begin{align*}
N'(r) = \dfrac{r D(r)}{H(r)} \left(  \dfrac{1}{r} + \dfrac{D'(r)}{D(r)} - \dfrac{H'(r)}{H(r)}  \right),
\end{align*}
and estimate $D'(r)$ and $H'(r)$ separately. Introduce the vector field $\beta$ as in Lemma \ref{lem:c1_est},
\begin{align*}
\beta(X) \coloneqq A(X)X/\mu(X),
\end{align*}
and notice that since the unit outer normal on $\partial B_r$ is $n(X) = X/\abs{X}$ , then the definition of $\beta$ and $\mu$ give the identities
\begin{align}\label{eqn:beta_id}
\beta  \equiv r An/\mu ,  \, \,  \dotp{ \beta, n} \equiv r, \text{ on } \partial B_r.
\end{align}
Let us perform all our computations first without justifying certain integration by parts (which involve singular weights such as $\delta^{-n+d}$). At the end of the proof we justify such integration by parts steps.

First, for $D'(r)$, recall that $A = \delta^{-n+d+1} \A$, use \eqref{eqn:beta_id} and compute for almost every $r$, 
\begin{equation}\label{eqn:D'_var1}
\begin{split}
D'(r) & = \int_{\partial B_r} \dotp{ A \nabla u, \nabla u} \; d\sigma  \\
& = r^{-1} \int_{\partial B_r} \dotp{A \nabla u, \nabla u} \dotp{\beta, n} \; d\sigma \\
& = r^{-1} \int_{B_r} \divv \left( \dotp{ A \nabla u, \nabla u} \beta \right) \; dX.
\end{split}
\end{equation}
Writing $\A = \left( \ol{a}_{ij} \right)$ for the entries of $\A$, we have from the symmetry of $\A$ that 
\begin{equation}\label{eqn:D'_var2}
\begin{split}
\int_{B_r} \divv \left( \dotp{ A \nabla u , \nabla u } \beta \right) \; dx &=  \int_{B_r} \dotp{ A \nabla u , \nabla u} \divv(\beta) \; dX \\
& +  \int_{B_r} \dotp{\nabla(\delta^{-n+d+1}) , \beta} \dotp{\mathcal{A} \nabla u, \nabla u} \; dX  \\
&  +  \int_{B_r} \delta^{-n+d+1} \sum_{ijk} (\ol{a}_{ij})_{X_k} u_{X_i} u_{X_j} \beta_k \; dX \\
& + 2 \int_{B_r} \delta^{-n+d+1} \sum_{ijk} \ol{a}_{ij} u_{X_i} u_{X_j X_k} \beta_k \; dX.
\end{split}
\end{equation}
By \eqref{eqn:asm_4} and \eqref{eqn:asm_4b} we see that $\divv(\beta) = n + O(r)$ and $\dotp{ \nabla \delta , \beta } = \delta + \delta \,  O(r)$ since $\dotp{X, \nabla \delta(X)} = \delta(X)$. Using also the fact that $\ol{a}_{ij}$ are Lipschitz and $\dotp{ \nabla( \delta^{-n+d+1}) , \beta} = ((-n+d+1) + O(r)) \delta^{-n+d+1}$, we thus obtain
\begin{align*}
\int_{B_r} \divv( \dotp{A \nabla u, \nabla u} \beta ) \; dX = (d+1)D(r)   + 2\int_{B_r} \delta^{-n+d+1} \sum_{ijk} u_{X_i} u_{X_j X_k} \beta_k \; dX +  O(r D(r)). 
\end{align*}
Integrating by parts the second term in $X_j$ and using that $u$ is a solution gives us
\begin{equation}\label{eqn:D'_var3}
\begin{split}
2\int_{B_r} \delta^{-n+d+1} & \sum_{ijk} \ol{a}_{ij} u_{X_i} u_{X_j X_k} \beta_k \; dX \\
&  = - 2 \int_{B_r} \delta^{-n+d+1} \sum_{ijk} \ol{a}_{ij} u_{X_i} u_{X_k} (\beta_k)_{X_j} \; dX   + 2 \int_{\partial B_r}  \dotp{A \nabla u, n} \dotp{\nabla u, \beta} \; d\sigma \\
& = 2 \int_{\partial B_r}  \dotp{A \nabla u, n} \dotp{\nabla u, \beta} \; d\sigma - 2 D(r) + O(r D(r)),
\end{split}
\end{equation}
where the second inequality follows from \eqref{eqn:asm_4}. Putting together our computations with that of $D'(r)$, applying \eqref{eqn:beta_id}, and using that $A$ is symmetric yields
\begin{align}\label{eqn:D'_var}
D'(r) = (d-1)D(r)/r + 2 \int_{\partial B_r} \mu^{-1} \dotp{A \nabla u, n}^2  \; d \sigma + O(D(r)).
\end{align}
Moreover, since $\divv( A u \nabla u) = \dotp{A \nabla u, \nabla u}$, we may write
\begin{align}\label{eqn:D_var}
D(r) = \int_{B_r} \divv( A u \nabla u) \; dX & = \int_{\partial B_r} u \dotp{A \nabla u, n} \; d\sigma.
\end{align}

Next, we move to $H'(r)$. Setting $\ol{\mu} = \delta^{n-d-1} \mu$, we have
\begin{align*}
H(r) & = r^{n-1} \int_{\partial B_{10}} \delta(rX)^{-n+d+1} \ol{\mu}(rX) u(rX)^2 \; d\sigma \\
& = r^{d} \int_{\partial B_{10}} \delta(X)^{-n+d+1} \ol{\mu}(rX) u(rX)^2 \; d\sigma.
\end{align*}
Notice that $\ol{\mu}$ is Lipschitz by \eqref{eqn:asm_1b}, and thus 
\begin{equation} \label{eqn:H'_var1}
\begin{split}
H'(r) & = dH(r)/r + 2\int_{\partial B_r} \mu u \dotp{\nabla u, n}  + \delta^{-n+d+1} u^2 \dotp{\nabla \ol{\mu}, n} \; d\sigma \\
& = dH(r)/r + 2\int_{\partial B_r} \mu u \dotp{\nabla u, n} \; d\sigma + O(H(r)),
\end{split}
\end{equation}
since $X/\abs{X} = n$ on $\partial B_r$. If one considers the vector field $An - \mu n$, then direct computation shows (by definition of $\mu$) that $\dotp{An - \mu n, n} = 0 $ on $\partial B_r$, and thus $An - \mu n$ is a tangent vector field to $\partial B_r$. Moreover, we have the uniform bound 
\begin{align}\label{eqn:div_est}
\abs{\divv_{\partial B_r}(A n - \mu n)} \le C \mu.
\end{align}
Indeed using \eqref{eqn:beta_id} we can write on $\partial B_r$, 
\begin{align*}
An - \mu n = r^{-1} \mu  \left( \beta(X) - X \right),
\end{align*}
and remark that $\divv( \beta(X) - X) = O(r)$ by \eqref{eqn:asm_4}, while 
\begin{align*}
\nabla \mu = \nabla( \delta^{-n+d+1} \ol{\mu}),
\end{align*}
and $\ol{\mu}$ is Lipschitz, so that 
\begin{align*}
\dotp{\nabla \mu, \beta(X) - X} = O(r) \mu
\end{align*}
by virtue of \eqref{eqn:asm_4} and \eqref{eqn:asm_4b} and the fact that $\mu \simeq \delta^{-n+d+1}$ again. Collectively, these prove \eqref{eqn:div_est}.

In particular the divergence theorem on $\partial B_r$ shows
\begin{equation}\label{eqn:H'_var2}
\begin{split}
\abs{ \int_{\partial B_r} 2 \mu u \dotp{ \nabla u, n} - 2 u \dotp{ \nabla u, A n} \; d\sigma } & = \abs{ \int_{\partial B_r} \dotp{\nabla (u^2), \mu n - A n } \; d\sigma } \\ 
& = \abs {\int_{\partial B_r} u^2 \divv_{\partial B_r} ( \mu n - A n) \; d\sigma } = O(H(r)).
\end{split}
\end{equation}
Substituting this in \eqref{eqn:H'_var1} gives
\begin{align}\label{eqn:H'_var}
H'(r) = dH(r)/r + 2 \int_{\partial B_r} u \dotp{A \nabla u, n} \; d\sigma + O(H(r))
\end{align}
since $A$ is symmetric.

In view of \eqref{eqn:D'_var}, \eqref{eqn:D_var}, and \eqref{eqn:H'_var}, we have shown that 
\begin{align*}
N'(r) & = 2N(r) \left \{   \dfrac{\int_{\partial B_r} \mu^{-1} (A \nabla u \cdot n)^2  \; d\sigma}{ \int_{\partial B_r} u (A \nabla u \cdot n) \; d\sigma }  - \dfrac{\int_{\partial B_r} u (A \nabla u \cdot n) \; d\sigma }{  \int_{\partial B_r}  \mu u^2   \; d\sigma   }      \right \} + O(N(r)),
\end{align*}
which proves the claims since the first term above is nonnegative by the Cauchy-Schwarz inequality. 

Thus to complete the proof, we need only justify our applications of the divergence theorem on $B_r$ (and $\partial B_r$) where they appear. Indeed, our arguments thus far have not used that $u$ vanishes continuously on $B_{10}$ which is an important part of the analysis. In what follows, we sketch the necessary steps to make such applications (and related integration by parts) rigorous.

To this end, denote for $0 < \epsilon \ll r$,
\begin{align*}
\Omega_{r, \epsilon} \coloneqq B_r \cap \{ \delta > \epsilon \}, \, \Gamma_{r, \epsilon} \coloneqq \partial \Omega_{r,\epsilon} \setminus \partial B_r \equiv \{ \delta = \epsilon\} \cap B_r. 
\end{align*}
Our first point to justify is the application of the divergence theorem in \eqref{eqn:D'_var1}, and in the end our justification comes from performing all our analysis on $\Omega_{r,\epsilon_m}$ for a suitable sequence $\epsilon_m \da 0$ and then taking limits. Using the Lipschitz nature of $u$ (Lemma \ref{lem:w2_u}) we can invoke the dominated convergence theorem, to write
\begin{align}
\int_{\partial B_r} \dotp{ A \nabla u, \nabla u} \dotp{\beta, n} \; d\sigma & = \lim_{\epsilon_m \da 0} \int_{\partial B_r \,  \cap  \, \{\delta > \epsilon_m \} } \dotp{A \nabla u, \nabla u} \dotp{\beta, n} \; d\sigma  \nonumber \\
& = \lim_{\epsilon_m \da 0} \int_{\partial \Omega_{r,\epsilon_m} } \dotp{A \nabla u, \nabla u} \dotp{\beta, n} \; d\sigma. \label{eqn:lim_just}
\end{align}
Indeed, the second inequality holds for an appropriate choice of $\epsilon_m$ because $n \equiv \nabla \delta$ on $\Gamma_{r, \epsilon_m}$, and so \eqref{eqn:asm_3} gives us that 
\begin{align*}
\int_{\Gamma_{r, \epsilon_m}} \abs{ \dotp{A \nabla u, \nabla u} \dotp{\beta, n} }\; d\sigma \le C \int_{\Gamma_{r, \epsilon_m}} \delta^{-n+d+1} \abs{\nabla u}^2 \delta \; d\sigma \equiv \epsilon_m C \int_{\Gamma_{r, \epsilon_m}} \delta^{-n+d+1} \abs{\nabla u}^2 \; d\sigma.
\end{align*}
Now if we set $\psi(t) \coloneqq t \int_{\Gamma_{r, t}} \delta^{-n+d+1} \abs{\nabla u}^2 \; d\sigma$, then we readily check from the co-area formula,
\begin{align*}
\fint_{\epsilon/2}^\epsilon \psi(t) \; dt \simeq  \int_{\epsilon/2}^\epsilon \psi(t) \; \dfrac{dt}{t} \lesssim \int_{B_r \cap \{ \epsilon/2 < \delta < \epsilon \}} \delta^{-n+d+1} \abs{\nabla u}^2 \; dX \ra 0
\end{align*}
as $\epsilon \da 0$ because $u \in W_r(B_{10})$. Hence there is a sequence $\epsilon_m \da 0$ for which $\psi(\epsilon_m) \ra 0$, which shows we may write such a limit as in \eqref{eqn:lim_just} for an appropriate choice of $\epsilon_m$. Alternatively, we could instead apply Lemma \ref{lem:w2_u} to see that $\abs{\nabla u}$ is bounded in a small neighborhood of $\R^d$, and thus for any $\epsilon_m \da 0$, the same holds true by the estimate above.

Now with the limit \eqref{eqn:lim_just}, we proceed with the computations as in \eqref{eqn:D'_var2}, but with the domain $B_r$ replaced by $\Omega_{r,\epsilon_m}$ for $\epsilon_m$ fixed and then pass to the limit. The estimates here remain exactly the same, since all integrals converge absolutely except for the one appearing in \eqref{eqn:D'_var3} where we integrate by parts again and need further justification. To deal with this term, we notice that in the corresponding computation \eqref{eqn:D'_var3} with $\Omega_{r, \epsilon_m}$ replacing $B_r$, integrating by parts yields a new boundary term, which we must justify vanishes in the limit as $\epsilon_m \da 0$:
\begin{align*}
\int_{\Gamma_{r, \epsilon_m}} \dotp{A \nabla u, n} \dotp{\nabla u, \beta} \; d\sigma.
\end{align*} 
Here, we brutally estimate $\abs{ \dotp{A \nabla u, n}} \lesssim \delta^{-n+d+1} \abs{\nabla u} $, and split the contributions
\begin{align*}
\abs{ \dotp{\nabla u, \beta}} \le \abs{ \nabla_x u} \abs{X} + \abs{\nabla_t u} \abs{\dotp{\beta, \nabla \delta}},
\end{align*}
where we use that $\nabla \delta \equiv t/\abs{t}$, to estimate again with, \eqref{eqn:asm_3}
\begin{align*}
\abs{\int_{\Gamma_{r, \epsilon_m}} \dotp{A \nabla u, n} \dotp{\nabla u, \beta} \; d\sigma } & \lesssim \int_{\Gamma_{r, \epsilon_m}} \delta^{-n+d+1} \abs{\nabla u} \abs{\nabla_x  u } \abs{X}  \; d\sigma +  \epsilon_m \int_{\Gamma_{r, \epsilon_m}}  \delta^{-n+d+1} \abs{\nabla u}^2  \; d\sigma \\
& \lesssim r \left( \int_{\Gamma_{r, \epsilon_m}} \delta^{-n+d+1} \abs{\nabla u}^2 \; d\sigma \right)^{1/2} \left(  \int_{\Gamma_{r, \epsilon_m}} \delta^{-n+d+1} \abs{\nabla_x  u}^2 \; d\sigma  \right)^{1/2} \\
& \qquad  + \epsilon_m \int_{\Gamma_{r, \epsilon_m}} \delta^{-n+d+1} \abs{\nabla u}^2 \; d\sigma \eqqcolon \phi(\epsilon_m) + \psi(\epsilon_m).
\end{align*}
We see as in the previous paragraph that $\psi(\epsilon_m) \da 0$ as $\epsilon_m \da 0$. On the other hand Lemma \ref{lem:w2_u} says that $\abs{\nabla u}$ is bounded in a neighborhood of $\R^d$, and thus we have the estimate
\begin{align}\label{eqn:phi_em}
\phi(\epsilon_m) \le C_{r, u} \left ( \int_{\Gamma_{r,\epsilon_m}} \delta^{-n+d+1} \abs{\nabla_x u}^2 \; d\sigma  \right)^{1/2}.
\end{align}
Finally, we recall from Proposition \ref{prop:c1a_dxu} that $\nabla_x u \in W_r(B_{10})$ with $\mathrm{Tr}(\nabla_x u) = 0$ on $\Gamma \cap B_{10}$, by the regularity assumptions on $A$, and the estimate \eqref{eqn:u_x_tr} implies that for an appropriate choice of $\epsilon_m \da 0$, the right-hand side of \eqref{eqn:phi_em} vanishes as $\epsilon_m \da 0$. Altogether, these estimates allow us to pass to the limit in the domain $\Omega_{r, \epsilon_m}$ and obtain the same estimate for $D'(r)$ as in \eqref{eqn:D'_var}. 

To justify the equality \eqref{eqn:D_var}, we use the same domain approximation, and write 
\begin{align*}
\int_{B_r} \divv(u A \nabla u) \; dX & = \lim_{\epsilon_m \da 0} \int_{\Omega_{r,\epsilon_m}} \divv(u A \nabla u) \; dX \\
& = \lim_{\epsilon_m \da 0} \int_{\partial B_r \cap \{\delta > \epsilon_m\}} u \dotp{A \nabla u, n} \; d\sigma + \int_{\Gamma_{r, \epsilon_m}} u \dotp{A \nabla u, n} \; d\sigma.
\end{align*}
The second integral above vanishes in the limit since Lemma \ref{lem:w2_u} implies the decay of $u$ away from $\R^d$: $\abs{u} \le C_u   \delta$, while $\abs{\nabla u}$ is bounded in a neighborhood of $\R^d$.

This brings us to the justification of our computation of $H'(r)$, which consists of many similar ideas. In \eqref{eqn:H'_var1}, one can argue that the differentiation in the integrand is allowed by the dominated convergence theorem, the Lipschitz nature of $\ol{\mu}$, and the estimates 
\begin{align}\label{eqn:H'_decay}
\abs{u} + \delta \abs{ \nabla u}\le C_u \delta 
\end{align}
from Lemma \ref{lem:w2_u}. Indeed this estimate above gives that the integral $\int_{\partial B_r} \delta^{-n+d+1} \abs{u} \abs{\nabla u} \; d\sigma$ converges absolutely by Cauchy-Schwarz. What is left is again to justify the estimate \eqref{eqn:H'_var2}, which one can argue again by truncating to the manifold with boundary $\partial B_r \cap \{\delta > \epsilon_m\}$ and passing to the limit as $\epsilon_m \da 0$. Here one need to take care of the error obtained from the boundary contribution, which is bounded by
\begin{align*}
\int_{\partial B_r \cap \{\delta = \epsilon_m\}} u^2 \delta^{-n+d+1} d\HD^{n-2} \le C_u \epsilon_m^{2 - n + d + 1} \epsilon_m^{n-2} = C_u \epsilon_m^{d+1},
\end{align*}
by \eqref{eqn:H'_decay} and the fact that $\HD^{n-2}(\partial B_r \cap \{\delta = \epsilon_m\}) \simeq \epsilon_m^{n-2}$ for $\epsilon_m$ small. The above estimate implies that these boundary integrals vanish as $\epsilon_m \da 0$, which concludes our sketch of justifications and thus the proof.
\end{proof}

In the case that $\A(0) \ne I$, then one can perform a straight-forward change of variables to reduce to the case of Lemma \ref{lem:monotone_var}, and so almost-monotonicity still holds in that case. Since we shall use this change of variables again in the future, let us set it aside as a remark.

\begin{rmk}\label{rmk:change_vars_flat}
Assume \eqref{eqn:soln_var} and \eqref{eqn:A_cond}. Writing $\A_0 = \A(0)$, and $v(X) \coloneqq u(\A_0^{1/2}X)$, then $v$ solves 
\begin{equation*}
\left \{
\begin{aligned}
-\divv(B \nabla v) & = 0 \text{ in } \A_0^{-1/2} B_{10}, \\
v & = 0 \text{ on } \R^d \cap \A_0^{-1/2}B_{10},
\end{aligned} \right . 
\end{equation*}
with $B(X) \coloneqq \A_0^{-1/2} A(\A_0^{1/2} X) \A_0^{-1/2}$. Moreover $B(X)$ satisfies the higher co-dimension $C^{0,1}$ condition in $B_{10/\sqrt{\lambda}}$ with constants $C \lambda, C C_0$ for some dimensional constant $C>1$. 
\end{rmk}
\begin{proof}
The fact that $v$ is a solution of the associated equation is a straight-forward computation using the definition of weak solution, so we omit the details. On the other hand, the asymptotic structure of $\A$ from Definition \ref{defn:c1a} implies that $\A_0^{1/2}, \A_0^{-1/2}$ map $\R^d$ bijectively to $\R^d$, so that indeed $v$ vanishes on $\A_0^{-1/2} B_{10} \cap \R^d$ since $u$ vanishes on $B_{10} \cap \R^d$. Finally, we may write $B(X) = \abs{t}^{-n+d+1} \A_0^{-1/2} \A(\A_0^{1/2}X) \A_0^{-1/2}$ and use the structure of $\A$ at the boundary $\R^d$ again as in Definition \ref{defn:c1a} to conclude that $B$ satisfies the higher co-dimension condition in $ B_{10/\sqrt{\lambda}} \subset \A_0^{-1/2}B_{10} $, since $\A$ is Lipschitz. 
\end{proof}

At this stage, it will make life easier to instead assume \eqref{eqn:ellipse_cond} with $\eta$ sufficiently small instead of requiring ourselves to shrink to a smaller scale, though the two are essentially equivalent.

\begin{cor}\label{cor:monotone_var_gen}
Assume \eqref{eqn:soln_var} and \eqref{eqn:ellipse_cond}. Then as long as $\eta$ sufficiently small, there is a constant $C = C(n,d) > 0$ so that $N(r) e^{Cr}$ is monotone increasing where it is defined (so in particular, on the interval $(0, 10(1-\eta))$).
\end{cor}
\begin{proof}
By Remark \ref{rmk:change_vars_flat}, we know for $v(X) \coloneqq u( \A_0^{1/2} X)$ that 
\begin{equation*}
\left \{
\begin{aligned}
-\divv( B \nabla v) & = 0 \text{ in } \A_0^{-1/2} B_{10}, \\
v & = 0 \text{ on } \R^d \cap \A_0^{-1/2} B_{10},
\end{aligned} \right .
\end{equation*}
where $B$ is the matrix defined by $B(X) \coloneqq \A_0^{-1/2} A(\A_0^{1/2}X)\A_0^{-1/2}$. Since $A$ satisfies the higher co-dimension $C^{0,1}$ condition in $B_{10}$ with constants bounded constants by \eqref{eqn:ellipse_cond_cons}, $B(X) = \abs{t}^{-n+d+1} \mathcal{B}(X)$ satisfies the higher co-dimension $C^{0,1}$ condition in $B_{10(1-\eta)}$ with bounded constants. Moreover, by construction we have that $\mathcal{B}(0) = I$, and so Lemma \ref{lem:monotone_var} applies to $v$ to give almost monotonicity of the function 
\begin{align*}
r \mapsto \dfrac{  r \int_{B_r} \dotp{B \nabla v, \nabla v} \; dX  }{  \int_{\partial B_r} \mu_B v^2 \; d\sigma  } , \qquad \mu_B (X) \coloneqq \dotp{ B \dfrac{X}{\abs{X}}, \dfrac{X}{\abs{X}}  }.
\end{align*}

Changing variables through the map $\A_0^{1/2}$ (using the coarea formula to compute the Jacobian for the change of variables over $\partial B_r$), we for any $f \in C(B_r)$ that
\begin{align}\label{eqn:ellipse_cov}
\int_{\partial B_r} f(\A_0^{1/2} X) \; d\sigma(X) & =  \det \A_0^{-1/2} \int_{\partial E_r} f(Y) \dfrac{r}{\abs{ \A_0^{-1} Y}} \; d\sigma(Y).
\end{align}
In particular, we have the equalities
 compute that 
\begin{align*}
\int_{B_r} \dotp{B \nabla v, \nabla v} \, dX & = \det \A_0^{-1/2} \int_{E_r} \dotp{A \nabla u, \nabla u} \, dX, \\
\int_{\partial B_r} \mu_B v^2 \; d\sigma & = \det \A_0^{-1/2} \int_{\partial E_r} \mu_B(\A_0^{-1/2}X) u(X)^2 \dfrac{r}{\abs{\A_0^{-1}X}} \; d\sigma(X),
\end{align*}
which concludes the proof since
\begin{align*}
B(\A_0^{-1/2} X) & = \A_0^{-1/2} A(X) \A_0^{-1/2}.
\end{align*}

\end{proof}

\begin{rmk}\label{rmk:H_doubling_var}
From \eqref{eqn:H'_var} and \eqref{eqn:D_var} in the proof of Lemma \ref{lem:monotone_var} we see that the following identity holds:
\begin{align}
\dfrac{H'(r)}{H(r)} & = \dfrac{d}{r} +  \dfrac{2D(r)}{H(r)} + O(1) \nonumber  \\
& = \dfrac{d}{r} + \dfrac{2 N(r)}{r} + O(1). \label{eqn:H'_H2_var}
\end{align}
In particular, this implies 
\begin{align*}
\dfrac{d}{dr} \log (H(r) r^{-d}) = \dfrac{2 N(r)}{r} + O(1).
\end{align*}
Along with the almost-monotonicity of $N(r)$ from Lemma \ref{lem:monotone_var}, we see from a straight-forward estimate integrating $\tfrac{d}{dr} \left( \log( H(r) r^{-d})  \right) $ that whenever $0 < r_1 < r_2 < 10(1-\eta)$, 
\begin{align*}
r_2^{-d} H(r_2) \le C r_1^{-d} H(r_1) \left( \dfrac{r_2}{r_1} \right)^{C N(r_2)}
\end{align*}
with constant $C$ depending only on the dimensions $n$ and $d$. 
\end{rmk}

In view of the preceding Remark, standard arguments as in \cite[Theorem 3.1.3]{HLNODAL}, \cite[Theorem 4.6]{NV17} (along with the fact that $\mu(X) \simeq_\lambda \abs{t}^{-n+d+1}$ and $r/\abs{\A_0^{-1} Y} \simeq_\lambda 1$ on $\partial E_r(X_0)$) give the following corollaries.
\begin{cor}[Doubling inequalities on the boundary]\label{cor:doubling_var}
	Assume $\eta$ is sufficiently small. Then there is some constant $C >1$ for which we have the following doubling inequalities whenever $E_{2r} \subset B_{10}$ (and in particular, whenever $0 < 2r < 10(1-2\eta)$)
	\begin{align*}
		C^{N(r) -1} \fint_{  \partial E_r } \abs{t}^{-n+d+1} u^2 d\sigma  \le \fint_{\partial E_{2r}} \abs{t}^{-n+d+1} u^2 \; d\sigma & \le C^{  N(2r) + 1 }   \fint_{\partial E_{r}} \abs{t}^{-n+d+1} u^2 \; d\sigma, \\
		C^{N(r) -1} \fint_{   E_r } \abs{t}^{-n+d+1} u^2 dX \le \fint_{E_{2 r}} \abs{t}^{-n+d+1} u^2 \; dX & \le C^{  N(2r) + 1 }  \fint_{E_r} \abs{t}^{-n+d+1} u^2 \; dX.
	\end{align*}
	Moreover, by taking $\eta$ smaller, we can also assume that same inequalities hold with $B_{2r}$ and $B_{r}$ in place of $E_{2r}$ and $E_{r}$.
\end{cor}

\begin{cor}[Uniform control on the boundary]\label{cor:uniform_var}
	Assume $\eta$ is sufficiently small. Then there is a $C = C(n,d) > 0$ so that for all $Y \in \R^d \cap B_2$, and all $0 < r < 2$, one has 
	\begin{align*}
		N_u(Y, r) \le CN_u(0,  10(1-\eta)).
	\end{align*}
	
	In addition, for each $\epsilon >0$, there is $\delta >0$ so that if $\eta$ is small enough, then
	\begin{align*}
		N_u(Y,r) \le (1+ \epsilon) N_u(0, 10(1-\eta)), \; Y \in \R^d \cap B_\delta, \; \; 0 < r < 2.
	\end{align*}
\end{cor}

Just as in the case of harmonic functions, we have stronger rigidity results for the frequency function when considering solutions of \eqref{eqn:soln}. A careful evaluation of the proof of Lemma \ref{lem:monotone_var} shows that if $\A(X) = I$ (i.e., in the setting of the equation \eqref{eqn:soln}), then there are no error terms involving $\nabla \A$, and we have true monotonicity, $\tfrac{d}{dr} N(r) \ge 0$. We summarize these findings in the following Lemma, since we shall later use these more precise estimates for homogeneous solutions.
\begin{lemma}[Almgren monotonicity, flat space]\label{lem:monotone_flat_bdry}
	Suppose that $u \in W_r(B_R)$ is a solution to \eqref{eqn:soln} in $B_R$, and that u vanishes continuously on $\R^d \cap B_R$. 
	Then the following hold:
	\begin{enumerate}[(i)]
		\item $N(r)$ is non-decreasing for $r \in (0, R)$,
		\item 	$\dfrac{d}{dr} \log( r^{-d} H(r)) = \dfrac{2 N(r)}{r}$, so that for $0 < r_1 < r_2 < R$, 
		\begin{align*}
			 \left(\dfrac{r_2}{r_1}\right)^{2N(r_1)} \le \dfrac{r_2^{-d} H(r_2)}{r_1^{-d} H(r_1)} \le \left(\dfrac{r_2}{r_1}\right)^{2N(r_2)},
		\end{align*}
		\item $N(r) \equiv \Lambda$ on a non-trivial interval $I \subset (0,R)$ if and only if $u$ is $\Lambda$-homogeneous. 
	\end{enumerate}
	In particular, $N(0^+) = \lim_{r \da 0}N(r)$ exists.
\end{lemma}
	
Finally, since we shall frequently use compactness arguments to shorten our exposition, let us introduce the following Lemma which says that solutions to equations of the type \eqref{eqn:soln_var} with bounded frequency are a compact class. Since the arguments involved in the proof of this Lemma are straight-forward (once one recalls the Lipschitz estimates for solutions of \eqref{eqn:soln_var} guaranteed by Lemma \ref{lem:w2_u}), we omit the details.
\begin{lemma}[Compactness of solutions with bounded frequency]\label{lem:compact_soln}

Assume $\eta$ is sufficiently small and fix $\Lambda_0 > 0$. Suppose that $A_m$ are a sequence of matrices satisfying the higher co-dimensional $C^{0,1}$ condition in $B_{8 R_m}$ with constants $1 + \eta, \eta$, and that the $R_m$ are non-decreasing. Assume also that $u_m \in W_r(B_{8 R_m}) \cap C(B_{8 R_m})$ are a sequence of functions such that 
\begin{enumerate}[(a)]
\item $-\divv(A_m \nabla u_m) = 0$ in $B_{8R_m}$, 
\item$u_m$ vanishes continuously on $\R^d \cap B_{8 R_m}$,
\item $N_{u_m}^{A_m}(0, r) \le \Lambda_0$ for $0 < r \le 5R_m$, and $H_{u_m}^{A_m}(0, R_1) = 1$,
\end{enumerate}
for all $m \in \N$. Then with $R_\infty =    \liminf_{m \ra \infty} R_m  \in [ R_1, \infty]$, there is a subsequence $m_k \ra \infty$, a matrix $A_\infty$ satisfying the higher co-dimensional $C^{0,1}$ condition with constants $1+ \eta, \eta$ in $B_{5R_\infty}$, and a solution $u_\infty \in W_r(B_{5R_\infty})$ so that 
\begin{enumerate}[(a)]
\item $\abs{t}^{n-d-1}A_m \ra \abs{t}^{n-d-1}A$ uniformly on compact subsets of $B_{5R_\infty}$,
\item $u_{m_k} \ra u_\infty$ uniformly on compact subsets of $B_{5R_\infty}$,
\item $-\divv(A_\infty u_\infty) = 0$ in $B_{5R_\infty}$,
\item $u_\infty$ vanishes continuously on $\R^d \cap B_{5R_\infty}$,
\item $N_{u_m}^{A_m}(0,s) \ra N_{u_\infty}^{A_\infty}(0,s)$ as $m \ra \infty$ for $s \in (0, 5 R_\infty)$, and 
\item $H_{u_\infty}^{A_\infty}(0, R_1) = 1$, and $N_{u_\infty}^{A_\infty}(0,r) \le C \Lambda_0$ for all $r \le 5 R_\infty$.
\end{enumerate}
\end{lemma}

As a first consequence of Lemma \ref{lem:compact_soln}, let us demonstrate how one can deduce doubling inequalities for solutions $u$ \textit{off} of the boundary at sufficiently small scales, assuming only a bound on $N_u$ in a ball centered on the boundary. Importantly, this estimate is \textit{not} contained in the conclusion of Corollaries \ref{cor:doubling_var}, \ref{cor:uniform_var} since these two results only give doubling information about the solution $u$ on ellipses centered on $\R^d$.
 
\begin{lemma}\label{lem:equi_distr_var}
Assume $\eta$ is sufficiently small. For each $\Lambda_0 > 1$ and $\gamma \in (0,1)$, there is a $C_1 = C_1(\Lambda_0, \gamma) > 1$ so that the following holds. Whenever $u$ is a solution of \eqref{eqn:soln_var} with \eqref{eqn:ellipse_cond} and  $N_u(0, 10(1-\eta)) \le \Lambda_0$, then for any $Z \in B_2$, we have
\begin{align}\label{eqn:equidistr_mass}
\sup_{E_{\gamma}(Z)} \abs{u}^2 \ge C_1^{-1} \fint_{ \partial E_{1}}u^2 \; d\sigma_w,
\end{align}
and the doubling inequalities
\begin{align}\label{eqn:qual_doubling}
\fint_{ \partial E_{2 \gamma}(Z)}u^2 \; d\sigma_w \le C_1 \fint_{\partial E_\gamma(Z)} u^2 \; d\sigma_w, \qquad \fint_{  E_{2 \gamma}(Z)}u^2 \; dm \le C_1 \fint_{ E_\gamma(Z)} u^2 \; dm.
\end{align}
Moreover, each of the following quantities are comparable (with constant $C_1$) to $\fint_{\partial E_{1}} u^2 \; d\sigma_w$:
\begin{align}\label{eqn:equiv_norms}
\sup_{ E_\gamma(Z)} \abs{u}^2,  \; \fint_{\partial E_\gamma(Z)} u^2 \; d\sigma_w,  \; \fint_{E_\gamma(Z)} u^2 \; dm.
\end{align}
\end{lemma}
\begin{proof}
We argue by contradiction and compactness. If the Lemma does not hold, then there is a choice of $\Lambda_0 > 1$, and $\gamma \in (0,1)$, for which one can find a contradicting sequence of matrices $A_m$, solutions $u_m$, and points $Z_m \in B_2$ satisfying the hypothesis of the Lemma but which fail the inequality \eqref{eqn:equidistr_mass}. Hence, after normalizing by a multiplicative constant so that $H_{u_m}^{A_m}(0, 1) = 1$, we have 
\begin{align}\label{eqn:degen}
\sup_{E_\gamma(Z_m)} \abs{u_m}^2 \le C m^{-1} \fint_{\partial E_{1}} \abs{t}^{-n+d+1} u_m^2 \; d\sigma = C m^{-1}.
\end{align}
By Corollary \ref{cor:monotone_var_gen}, we have for each $m$ that $N_{u_m}^{A_m}(0, r) \le C \Lambda_0$ for $r \le 10(1-\eta)$, and thus we are in the context to apply Lemma \ref{lem:compact_soln} with $R_m \equiv 1$ for all $m$ as long as $\eta$ is chosen small enough. Hence we may extract a subsequence (which we still label $u_m$ for convenience) for which $u_m \ra u$ locally uniformly in $B_{5}$ to some non-trivial solution $u$ of an equation of the type \eqref{eqn:soln_var} in $B_{5}$ with corresponding matrix $A$ which satisfies the higher co-dimensional $C^{0,1}$ in this same ball. In view of the inequality \eqref{eqn:degen}, we see that the solution $u$ must vanish identically on some open subset of $B_{4} \cap (\R^n \setminus \R^d)$. However, since $u$ satisfies a uniformly elliptic equation with Lipschitz coefficients locally in $B_{5} \cap (\R^n \setminus \R^d)$, this contradicts the strong unique continuation principle for such solutions \cite{GL_86}, so that the first claim is proved.

The doubling claim \eqref{eqn:qual_doubling} and the comparability of the quantities in \eqref{eqn:equiv_norms} follows from essentially the same compactness argument, so we omit the details.
\end{proof}

The following Corollary is proved in a similar way, using the fact that for solutions $u$ of \eqref{eqn:soln}, $N_u(r)$ is scale-invariant in the sense that whenever $u_r \coloneqq u(r \, \cdot \, )$, then $N_{u_r}(s) = N_u(sr)$. While it's true that we can prove a much stronger statement (namely, that the limit of the following Corollary is unique), for our current purposes, and for the sake of brevity, we only need the existence of tangent maps in the following sense.
\begin{cor}\label{cor:tangent_maps}
	Let $u \in W_r(B_R)$ be a solution to \eqref{eqn:soln} which vanishes continuously on $\R^d \cap B_R$. Then for each sequence of scales $r_m \da 0$, there is a subsequence $r_{m_k} \da 0$ so that for $\{u_{m_k}\}$ defined by
	\begin{align*}
		u_{m_k} \coloneqq \dfrac{ v_{m_k}}{ \left( \int_{\partial B_1} v_{m_k}^2 \; d\sigma_w   \right)^{1/2}}, \; \; v_{m_k} \coloneqq u(r_{m_k} \, \cdot \,),
	\end{align*}
	one has $\lim_{k \ra \infty} u_{m_k} = v$	for some $(N_u(0^+))$-homogeneous solution $v$ of \eqref{eqn:soln} which vanishes continuously on $\R^d$. Here the limit holds in uniformly on compact subsets of $\R^n$, and thus also in $L^2(\partial B_1, d\sigma_w)$. 
\end{cor}

\subsection{Monotonicity formulae off the boundary}
One important aspect of the definition of the frequency function $N_u(Y,r)$ as in \eqref{eqn:N_var_gen} is that when $Y \not \in \R^d$ and $r \ll \delta(Y)$, it coincides with the generalized frequency function introduced by Garofalo and Lin in \cite{GL_86} (and later by Naber and Valtorta in \cite{NV17}) to study unique continuation properties of solutions to elliptic, divergence form equations with Lipschitz coefficients. In particular, we may use the following rescaling trick to see that inside Whitney regions for $\Omega =\R^n \setminus \R^d$, we may always fall back on this now well-understood theory. 

\begin{rmk}\label{rmk:whitney_scaling}
Assume \eqref{eqn:soln_var} and \eqref{eqn:ellipse_cond} with $\eta$ sufficiently small.  Then for $Y \in B_{2}$ and $\rho \coloneqq \delta(Y)/4$, we have by linearity that $v(X) \coloneqq u(Y + \rho X)$ solves the equation
\begin{align*}
-\divv( \rho^{n-d-1} \tilde{A} \nabla v) = 0 \text{ in } B_2,
\end{align*}
where $\tilde{A}(X) = A(Y + \rho X)$ ({note that $X\in B_2$ implies that $Y+\rho X \in B_{\delta(Y)/2}(Y)$}). Moreover, it is straight-forward to verify that $\rho^{n-d-1} \tilde{A}$ is uniformly elliptic with constant at most $C \eta$ and Lipschitz with constant at most $C \eta$. Applying the results of \cite{GL_86} and \cite{NV17} to the solution $v$ then gives us the following estimates in Whitney scales.
\end{rmk}

\begin{lemma}[Almost monotonicity for Whitney scales]\label{lem:mon_whit}
As long as $\eta$ is sufficiently small, there exists $C = C(n,d) > 0$ so that $e^{Cr}N_u(Y, r)$ is monotone increasing in $r$ for $r \in (0, \delta(Y)/4)$ and $Y \in B_2$.
\end{lemma}

\begin{thm}[Whitney covering estimates]\label{thm:NV}
	Let $Y \in B_2$ and suppose that $B_{2r}(Y) \subset \R^n \setminus \R^d$. Then if $\eta$ is sufficiently small, then there exists a constant $C = C(n,d)$ so that $N_u(Z, t) \le \Lambda_0$ for $Z \in B_{r}(Y)$ and $0 < t \le r$ implies
	\begin{align*}
		\L^n \left( B_s(\sing(u)) \cap B_{r/2}(Y)   \right) \le C^{\Lambda_0^2} (s/r)^2, \; \; s < r/2.
	\end{align*}
\end{thm}

\subsection{\texorpdfstring{The behavior of $N$ through the boundary}{The behavior of N through the boundary}}
With the almost-monotonicity properties of $N_u(X_0,r)$ established for $X_0 \in \R^d$ or $X_0 \not \in \R^d$ but $ r \ll \delta(X_0)$, we now investigate the behavior of $N_u(X_0,r)$ ``through'' the boundary, i.e., the behavior of $N_u(X_0, r)$ for when $X_0 \not \in \R^d$ yet $B_r(X_0) \cap \R^d \ne \emptyset$. Our first result below says that assuming a bound on the frequency of the solution $u$ at some top scale, the values of $N_u(X_0, r)$ and $N_u(\pi(X_0), r)$ are sufficiently close for $r \gg \delta(X_0)$.

\begin{lemma}[Boundary hopping]\label{lem:hop_var}
Assume \eqref{eqn:soln_var} and \eqref{eqn:ellipse_cond} with $\eta$ sufficiently small. Suppose in addition that $N_u(0, 10(1-\eta)) \le \Lambda_0$ for some $\Lambda_0 > 1$. Then there exists $C_1 = C_1(\Lambda_0) > 0$ so that whenever $X_0 \in B_{10^{-1}}$, and $r \in (3 \delta(X_0), 10^{-1})$
\begin{align}\label{eqn:H_hop}
e^{-C_1\delta(X_0)/r} \le \dfrac{H_u(X_0,r)}{H_u(\pi(X_0), r + 2\delta(X_0))} , \dfrac{H_u(X_0,r)}{H_u(\pi(X_0), r- 2\delta(X_0) )}   \le e^{C_1 \delta(X_0)/r}.
\end{align}

In particular, if $r =  \gamma \delta(X_0) < 10^{-1}$ for $\gamma > 3$, then
\begin{align}\label{eqn:N_hop}
e^{-C_1/\gamma} N_u(\pi(X_0), (\gamma - 2) \delta(X_0)   ) \le   N_u(X_0, \gamma\delta(X_0)) \le e^{C_1/\gamma} N_u(\pi(X_0), (\gamma + 2) \delta(X_0)).
\end{align}
\end{lemma}
\begin{proof}
	First, let us fix $X_0 \in B_{10^{-1}}$, $r \in (3 \delta(X_0), 10^{-1})$, and normalize $u$ so that $\sup_{E_{2({r + \delta(X_0)})}} \abs{u} = 1$. Define for $X \in [\pi(X_0), X_0]$ (the segment connecting $\pi(X_0)$ to $X_0$), the radii $r(X) = r + 2 \abs{X - X_0}$. Our goal is to explicitly estimate 
	\begin{equation}\label{eqn:H_deriv_est}
		\abs{\nabla_X \log( H_u(X, r(X)) )} \le C r(X)^{-1}, \; \; X \in [\pi(X_0), X_0]
	\end{equation}
	 in order to obtain the first claimed inequality. This is largely computational, so let us sketch the main steps.

	By the choice of the normalization of $u$ and Lemma \ref{lem:equi_distr_var}, we see that if $\eta$ is sufficiently small, then there is a constant $C_1 = C_1(\Lambda_0) >1$ for which $\fint_{\partial E_{r(X)}(X)} u^2 \; d\sigma_w \simeq_{C_1} 1$, or in other words,
	\begin{align}\label{eqn:H_comp}
		C_1^{-1} \le r(X)^{-d}\int_{\partial E_{r(X)}(X)} u^2 \; d\sigma_w \le C_1, \; X \in [\pi(X_0), X_0].
	\end{align}
	Next, we remark as in the change of variables from \eqref{eqn:ellipse_cov} that we may rewrite
	\begin{align*}
		H_u(X, r) \equiv \det \A(X)^{1/2} r^{n-1} \int_{\partial B_1} u(\tilde{Y})^2 \delta(\tilde{Y})^{-n+d+1} \tilde{\mu}(Y) \; d\sigma(Y)
	\end{align*}
	where we define $\tilde{Y} \coloneqq X + r \A(X)^{1/2} Y \in \partial E_r(X)$ whenever $Y \in \partial B_1$ and $\tilde{\mu}(Y) \simeq 1$ is the Lipschitz function defined by 
	\begin{align}
		\tilde{\mu}(Y) \coloneqq  \dotp{\A(\tilde{Y})   \A(X)^{-1/2} Y, \A(X)^{-1/2} Y }.
	\end{align}
	Indeed $\tilde{\mu}(Y)$ is Lipschitz since $\A$ and $\tilde{Y}$ are. In fact, for the latter, we have the estimate
	\begin{align*}
		\abs{\nabla_X \tilde{Y} - I} \le C r \abs{Y} \le C r.
	\end{align*}
	With this change of variables, we then may rewrite 
	\begin{align*}
		H_u(X, r(X)) \equiv \det \A(X)^{1/2} r(X)^{n-1} \int_{\partial B_1} u(\tilde{Y})^2 \delta(\tilde{Y})^{-n+d+1} \tilde{\mu}(Y) \; d\sigma(Y),
	\end{align*}
	and so computing $\nabla_X H_u(X, r(X))$ becomes a routine application of the product rule. It is easy to see that the contribution coming from the terms where $\nabla_X$ is applied to $\det \A(X)^{1/2}$, $r(X)^{n-1}$, or $\tilde{\mu}(Y)$ are bounded in absolute value by
	\begin{align*}
		C H_u(X, r(X)), C H_u(X, r(X))/r(X), CH_u(X, r(X)),
	\end{align*}
	respectively, and so since $r(X) < C$, to prove \eqref{eqn:H_deriv_est} we need only to consider the terms where $\nabla_X$ is applied to $u(\tilde{Y})^2$ or $\delta(\tilde{Y})^{-n+d+1}$.
	
	As for the first term involving $\nabla_X (u(\tilde{Y})^2)$, we estimate using the Lipschitz nature of $\tilde{Y}$ and the gradient estimate Lemma \ref{lem:w2_u} to see
	\begin{align*}
		\left | \det \A(X)^{1/2} r(X)^{n-1} \int_{\partial B_1} \right. & \left . 2 u(\tilde{Y})  \nabla u(\tilde{Y}) \nabla \tilde{Y} \delta(\tilde{Y})^{-n+d+1} \tilde{\mu}(Y) \; d\sigma(Y)  \right | \\
		& \le C r(X)^{n-1} \int_{\partial B_1} \delta{(\tilde{Y})}^{-n+d+1} \; d\sigma(Y) \le C r(X)^d \le C C_1 H_u(X, r(X))
	\end{align*}
	by virtue of \eqref{eqn:H_comp}. As for the second term involving $\nabla_X (\delta(\tilde{Y})^{-n+d+1})$, one argues essentially in the same way with Lemma \ref{lem:w2_u} and the fact that $\abs{u(Y)} \le C \delta(Y)$ for $Y \in E_{2(r + \delta(X))}(X)$ to conclude that
	\begin{align*}
		\left | \det \A(X)^{1/2} r(X)^{n-1} \int_{\partial B_1} \right. & \left .  u(\tilde{Y})^2   \nabla (\delta(\tilde{Y})^{-n+d+1}) \tilde{\mu}(Y) \; d\sigma(Y)  \right | \\
		& \le C r(X)^{n-1} \int_{\partial B_1} u^2 \delta{(\tilde{Y})}^{-n+d} \; d\sigma(Y) \le C r(X)^d \le C C_1 H_u(X, r(X))
	\end{align*}
	as well, which altogether completes our proof of the estimate \eqref{eqn:H_deriv_est}. Integrating this inequality over the segment $[\pi(X_0), X_0)]$, one readily obtains
	\begin{align*}
		\abs{ \log \left( \dfrac{ H_u(X_0, r(X_0)) }{H_u(\pi(X_0), r(\pi(X_0))}  \right)} \le C C_1 \delta(X_0)/r,
	\end{align*}
	which is the first desired estimate in \eqref{eqn:H_hop}, while the other is obtained in essentially the same manner where we instead define $r(X) \coloneqq r - 2 \abs{X - X_0}$. 
	
	In order to obtain the estimate \eqref{eqn:N_hop}, we simply recall the definition of $N_u$ as in \eqref{eqn:N_var_gen}, apply \eqref{eqn:H_hop} and the containment
	\begin{align*}
		E_{r+2 \delta(X_0)}(\pi(X_0)) \supset E_r(X_0), \; \; E_{r - 2 \pi(X_0)}(\pi(X_0)) \subset E_r(X_0)
	\end{align*}
	which holds if $\eta$ is sufficiently small.
	Indeed if $Y \in E_r(X_0)$, then 
	\begin{align*}
		\abs{ \A( \pi(X_0) )^{-1/2} (Y - \pi(X_0)) } & \le \abs{ \A( x_0 )^{-1/2} (Y - X_0) } + \abs{ \left(\A( x_0 )^{-1/2} - \A(\pi(X_0))^{-1/2} \right) (Y - X_0) } \\
		& \qquad + \abs{ \A(\pi(X_0))^{-1/2} (X_0 - \pi(X_0))  } \\
		& \le r + C \eta \delta(X_0) + (1 + C \eta) \delta(X_0) \le r + 2 \delta(X_0),
	\end{align*}
	by courtesy of \eqref{eqn:ellipse_cond}, while the other containment is proved in essentially the same manner.
\end{proof}

Combining the previous Lemma with the doubling inequalities in Lemma \ref{lem:equi_distr_var} as well as Corollary \ref{cor:uniform_var}, we may now conclude uniform boundedness of the frequency $N_u$ in all small balls near the origin assuming only a bound on $N_u(0, 10(1-\eta))$. 

\begin{cor}\label{cor:uniform_var_off}
Suppose \eqref{eqn:soln_var} and \eqref{eqn:ellipse_cond} with $\eta$ is sufficiently small, and fix $\Lambda_0 >1$. Then there is a $C_1 = C_1(\Lambda_0) > 0$ so that $N_u(0, 10(1-\eta)) \le \Lambda_0$ implies $N_u(Y, r) \le C_1 $ for all $Y \in B_{10^{-2}}$ and $0 < r < 10^{-2}$.
\end{cor} 

\begin{proof}
	Let $Y \in B_{10^{-2}}$, and notice that for all $r > 0$ of the form  $r = \gamma \delta(Y) \in(3 \delta(Y), 10^{-1})$ with $\gamma > 3$, we have from Lemma \ref{lem:hop_var} and Corollary \ref{cor:uniform_var} 
	\begin{align}\label{eqn:N_large}
		N_u(Y, r) \le e^{C_1} N_u(\pi(Y), ( \gamma +2 ) \delta(Y)) \le e^{C_1} C \Lambda_0,
	\end{align}
	for some constant $C_1 = C_1(\Lambda_0)$. As for the range $r \in (4^{-1} \delta(Y), 4 \delta(Y))$, we use Lemma \ref{lem:equi_distr_var} and $E_r(Y) \subset E_{16r}(Y)$ to see that 
	\begin{align*}
		N_u(Y, r) \le \dfrac{Cr  \int_{E_r(Y)} \abs{\nabla u}^2 \; dm  }{ \int_{\partial E_r(Y)} u^2 \; d\sigma_w} \le \dfrac{C C_1 r  \int_{E_{16r}(Y)} \abs{\nabla u}^2 \; dm  }{ \int_{\partial E_{16r}(Y)} u^2 \; d\sigma_w} \le C C_1 N_u(Y, 16r) \le C(\Lambda_0),
	\end{align*}
	since $16r > 3 \delta(Y)$, so that \eqref{eqn:N_large} applies. Finally, when $r \in (0, \delta(Y)/4)$, we apply the monotonicity Lemma \ref{lem:mon_whit} in Whitney regions to conclude
	\begin{align*}
		N_u(Y, r) \le C N_u(Y, \delta(Y)/4),
	\end{align*}
	which concludes the proof.
\end{proof}

In view of the previous Lemma, in our efforts towards proving Theorem \ref{thm:main} we now may frequently use in addition to \eqref{eqn:soln_var} and \eqref{eqn:ellipse_cond} that $u$ is a solution for which
\begin{align}\label{eqn:unif_bound}
N_u(Y,r) \le \Lambda_0, \; \text{ whenever } E_r(Y) \subset B_{10}.
\end{align}
The important piece of \eqref{eqn:unif_bound} is that we have a uniform estimate on $N_u(Y,r)$ for $Y \not \in \R^d$, which is not guaranteed initially by Corollary \ref{cor:uniform_var}.

\section{Frequency pinching, cone conditions, and quantitative approximations}\label{sec:pinch}

In this section, we compile a list of results which say that when $u$ is a solution whose frequency $N_u(X_0, r)$ is almost constant on a sufficiently large interval, and $X_0 \in \R^d$, then the singular set of $u$ lives in conical neighborhood of an $(n-2)$-plane. This containment in a cone is an essential part of our main covering argument Lemma \ref{lem:cov} in the next section. Before our results, let us introduce two useful definitions.

\begin{defn}\label{defn:pinched}
Given a non-constant solution $u$ of \eqref{eqn:soln_var}, we say that its frequency is $\epsilon$-pinched about $\Lambda$ at $X_0 \in B_1$ in the scales $[r_1, r_2]$ if for all $r \in [r_1, r_2]$, we have 
\begin{align*}
\abs{N_u(X_0, r) - \Lambda} \le \epsilon.
\end{align*}
\end{defn}

In what follows, whenever $0 \not \equiv f \in L^\infty(B_r(X_0))$, define \[f_r^{X_0}(Y) \coloneqq f(X_0 + rY) \norm{  f(X_0 +  r  \; \cdot)   }^{-1}_{L^\infty(B_1)},\] so that $f_r^{X_0} \in L^\infty(B_1)$ with $\norm{f_r^{X_0}}_{L^\infty(B_1)} = 1$. With this notation we introduce a normalized distance to $H_\Lambda^\lambda$ as follows.

\begin{defn}\label{defn:dist_homogen}
When $0 \not \equiv f \in L^\infty(B_r(X_0))$, we define the normalized $L^\infty$ distance of $f$ to $H_\Lambda^\lambda$ in $B_r(X_0)$ by
\begin{align*}
\dist_r^{X_0}(f, H_\Lambda^\lambda) \coloneqq \inf_{ u \in H_\Lambda^\lambda} \norm{f_r^{X_0} - u}_{L^\infty(B_1)} ,
\end{align*}
and the normalized $L^{\infty}$ distance of $f$ to $H_\Lambda^\lambda$ over the interval $[r_1, r_2]$ by
\begin{align*}
\dist^{X_0}_{r_1, r_2} (f, H_\Lambda^\lambda) \coloneqq \inf_{u \in H_{\Lambda}^\lambda}  \sup_{s \in [r_1, r_2]} \norm{ f_s^{X_0} - u }_{L^\infty (B_1)}.
\end{align*}
\end{defn}

{The expectation is that if a solution is pinched around $\Lambda$ in an interval then it should grow approximately like a $\Lambda$-homogeneous function for scales in that interval. Our next Lemma confirms this as long as the solution is pinched around a point in $\mathbb R^d$. }

\begin{lemma}[Pinching specifies growth]\label{lem:H_size}
	Assume that $u$ is a solution of \eqref{eqn:soln_var}, and assume in addition \eqref{eqn:ellipse_cond} and \eqref{eqn:unif_bound}. Then there exists a constant $C = C(\lambda, C_0) > 1$ so that if $u$ is $\epsilon$-pinched about $\Lambda$ at {$X_0\in B_5\cap \mathbb R^d$} in the scales $[r_1, r_2]$ with $r_2 < 1$, then for $r \in [r_1, r_2]$,
	\begin{align*}
		 C^{-1} \left(\dfrac{r}{r_2} \right)^{2(\Lambda +\epsilon)}  H_u(X_0, r_2) r_2^{-d} \le  H_u(X_0, r) r^{-d} \le C \left(\dfrac{r}{r_2} \right)^{2(\Lambda -\epsilon)}  H_u(X_0, r_2) r_2^{-d}.
	\end{align*}
	In addition, for all $r \in [r_1, r_2]$ so that $2r < r_2$, we have the gradient estimate
	\begin{align*}
		r^2 \fint_{B_r(X_0)} \abs{\nabla u}^2 \; dm \le C \left(   H_u(X_0,r_2) r_2^{-d} \right) \left(\dfrac{2r}{r_2}\right)^{2(\Lambda - \epsilon)}
	\end{align*}
\end{lemma}
\begin{proof}
	To ease notation, we write $H(r) = H_u(X_0,r)$ and $N(r) = N_u(X_0,r)$. By Remark \ref{rmk:H_doubling_var} we see that
	\begin{align*}
		\dfrac{d}{dr} \log(H(r) r^{-d}) = \dfrac{2N(r)}{r} + O(1),
	\end{align*}
	with implicit constant depending only on $C_0$, $\lambda$, and the dimension. Integrating this inequality from $r$ to $r_2$ then gives for any $r_1 < r< r_2$, 
	\begin{align*}
	2(\Lambda - \epsilon)\log(r_2/r) - C \le  \log \left(  \dfrac{H(r_2)(r_2)^{-d}}{H(r) r^{-d}}  \right) \le 2(\Lambda +\epsilon) \log(r_2/r) + C	
	\end{align*}
	from which one easily deduces the first pair of desired inequalities. 
	
	As for the second estimate, notice that since $\tfrac{d}{dr} \log(H(r)r^{-d}) \ge -C$, by Remark \ref{rmk:H_doubling_var}, the same computation as above yields the estimate
	\begin{align*}
		H(t) \le C (t/r)^d H(r), \; \; 0 < t < r < r_2.
	\end{align*}
	In particular, given $ r \in [r_1, r_2]$ with $2r < r_2$, we can estimate using the fact that $\int_{\partial B_t(X_0)} u^2 \; d \sigma_w \simeq H(t)$,
	\begin{align*}
		\int_{B_{2r}(X_0)} u^2 \; dm & = \int_{0}^{2r} \int_{\partial B_t(X_0)} u^2 \; d \sigma_w dt    \le C \int_0^{2r} \int_{\partial B_{2r}(X_0)} u^2 \; d\sigma_w dt  \le C r H(2r),
	\end{align*}
	or in other words,
	\begin{align*}
		\fint_{B_{2r}(X_0)} u^2 \; dm \le C H(2r)(2r)^{-d},
	\end{align*}
	since $m(B_t(X_0)) \simeq t^{d+1}$ and $\sigma_w(\partial B_t(X_0)) \simeq t^d$ for all $t > 0$. Caccioppoli's inequality along with the first estimate of the Lemma then yields the second:
	\begin{align*}
		r^2 \fint_{B_r(X_0)} \abs{\nabla u}^2 \; dm \le  C  \fint_{B_{2r}(X_0)} u^2 \; dm \le C  H(2r)(2r)^{-d} \le C r^{-2} \left( H(r_2)r_2^{-d}  \right) (2r/r_2)^{2(\Lambda - \epsilon)}.
	\end{align*}
\end{proof}

{For points in $\mathbb R^d$ we can say a little bit more, namely that if the frequency of $u$ is pinched between scales, then $u$ is close to some homogeneous solution of the ``constant coefficient equation" in a slightly smaller range of scales. This is an analogue of \cite[Theorem 3.12]{CNV15}, though the richness of the space of homogeneous solutions in our setting adds additional complexity to the proof. Note that it will be more convenient for us to measure closeness in the $L^\infty$-sense. We do this below in Lemma \ref{lem:pinch_homogen}.}

\begin{lemma}[Approximation by homogeneous solutions]\label{lem:fourier}
	For each $\tau \in (0,1)$ {small enough depending on $\lambda,n,d$} and $\Lambda_0 > 1$, there exists $\gamma = \gamma(\tau, \Lambda_0) \in (0,1)$ and $\epsilon = \epsilon(\Lambda, \tau) \in (0,1)$ small enough so that the following holds.
	 
	Assume \eqref{eqn:soln_var}, \eqref{eqn:ellipse_cond}, \eqref{eqn:unif_bound} and that $u$ is $\epsilon$-pinched about $\Lambda \in \mathcal{F}$, $\Lambda \le \Lambda_0$ at $X_0 \in B_1 \cap \R^d$ in the scales $[r_1, 1]$ for $r_1 < 100^{-1}$. Assume in addition that $\A(X_0) = I$.  Then for $\eta$ chosen small enough, depending on $\tau$ and $\Lambda_0$, and $M$ chosen large enough (depending on $\tau$) there exists a $\Lambda$-{homogeneous} solution $w_0$ of the ``constant-coefficient equation'' \eqref{eqn:soln} for which
	\begin{align*}
		\fint_{\partial B_r} (u - w_0)^2 \; d\sigma_w \le \tau^{1/2} \fint_{\partial B_r} u^2 \; d\sigma_w, \; \; r \in [Mr_1, \gamma].
	\end{align*}

\end{lemma}
\begin{proof}
	To ease notation, we assume $X_0 = 0$ and write $H(s) = H_u(0, s)$. In addition, since the conclusion of  the Lemma is invariant under multiplication of $u$ by constants, we normalize $u$ so that $H(1) = H_u(0,1) = 1$. Since $\Lambda \in \mathcal{F}$, choose $i \in \N$ so that $\Lambda = \Lambda_i$. As per Remark \ref{rmk:sph_rep}, we may write $u(r,\theta)$ in polar coordinates as
	\begin{align*}
		u(r,\theta) & = \sum_{k \ge 1} \sum_{j=1}^{N_k} a_j^k(r) \phi_j^k(\theta),
	\end{align*}
	where the $\phi_j^k$ form an orthonormal basis of $L^2(\partial B_1, d\sigma_w)$ and $a_j^k(r) \in \R$. With this notation, we easily see that 
	\begin{align*}
		s^{-d} H(s) \simeq_\lambda \sum_{k \ge 1} \sum_{j=1}^{N_k} a_j^k(s)^2, \; \; s \in [r_1, 1],
	\end{align*}
	since the conformal factor $\mu$ defining $H(r)$ (recall Definition \ref{defn:freq_var_gen}) is bounded above and below by a constant depending only on the ellipticity constant of $\A = \delta^{n-d-1} A$.
	
	First, we fix $r_1 < s_0 < 1/2$, and let $v$ be the solution to the boundary value problem 
	\begin{equation}\label{eqn:v_appx}
		\begin{aligned}
			-\divv(\abs{t}^{-n+d+1}  \nabla v   ) & = 0  && \text{ in } B_{s_0} \cap \Omega, \\
			v & = u && \text{ on } \partial (B_{s_0} \setminus \Gamma).
		\end{aligned}
	\end{equation}
	Assuming the representation of $u$ as in the statement of the Lemma, a straightforward computation gives us the explicit representation for $v$:
	\begin{align}
		v(r,\theta) & = \sum_{k \ge 1} \sum_{j=1}^{N_k} a_j^k(s_0) (r/s_0)^{\Lambda_k} \phi_j^k(\theta),
	\end{align}
	simply because $v$ has boundary data $u$ on $\partial B_{s_0}$, and each term in the sum above is a (homogeneous) solution of \eqref{eqn:v_appx}. {Our goal now is to show that $s\mapsto a_j^k(s)$ grows like $(s/s_0)^{\Lambda_k}$. This will quickly imply, given a good choice of $s_0$, that a solution like $v$ above will satisfy the conclusion of Lemma \ref{lem:fourier}.}
	
	To this end, we estimate the error $u-v$ using Lemma \ref{lem:approx} followed by Lemma \ref{lem:H_size} to obtain
	\begin{align}\label{eqn:v_up}
		\fint_{B_{s_0}} \abs{\nabla (u-v)}^2 \; dm \le C  \eta s_0 \fint_{B_{s_0}} \abs{\nabla u}^2 \; dm \le C \eta s_0^{-1} (2s_0)^{2(\Lambda  - \epsilon)}.
	\end{align}
	On the other hand, we can estimate this difference from below using the orthogonality of the basis $\phi_j^k$ in $L^2(\partial B_1, d\sigma_w)$. Indeed, letting $\partial_r$ denote the radial derivative, fixing $\gamma \in (0, 1)$, and writing $b_j^k(r) = a_j^k(s_0)(r/s_0)^{\Lambda_k}$, we have
	\begin{equation}\label{eqn:v_low}
		\begin{split}
			\fint_{B_{s_0}} \abs{\nabla(u-v)}^2 \; dm & \ge \fint_{B_{s_0}} \left(  \partial_r(u-v)\right)^2 \; dm \\
			& \ge C^{-1} s_0^{-d-1} \int_{\gamma s_0}^{s_0} \int_{\partial B_t} \left (\partial_r (u-v)\right)^2 \; d\sigma_w dt \\
			& \ge C^{-1} s_0^{-d-1} \int_{\gamma s_0}^{s_0} t^{d} \sum_{k\ge 1} \sum_{j=1}^{N_k} \left( (a_j^k)'(t) - (b_j^k)'(t) \right)^2 \; dt \\
			& \ge C_\gamma^{-1} s_0^{-1} \sum_{k \ge 1} \sum_{j=1}^{N_k} \int_{\gamma s_0}^{s_0} \left( (a_j^k)'(t) - (b_j^k)'(t) \right)^2 \; dt,
		\end{split}
	\end{equation}
	where $C_\gamma$ now is a constant depending on $\gamma$. 
	
	Combining both of our estimates and using that $b_j^k(s_0) = a_j^k(s_0)$, we can estimate 
	\begin{align*}
		\abs{a_j^k(\gamma s_0) - b_j^k(\gamma s_0)} & \le \int_{\gamma s_0}^{s_0} \abs{ (a_j^k)'(t) - (b_j^k)'(t)} \; dt \le C s_0^{1/2} \left( \int_{\gamma s_0}^{s_0} \left(  (a_j^k)'(t) - (b_j^k)'(t) \right)^2 \; dt \right)^{1/2}.
	\end{align*}
	Squaring the above inequality, summing over all indices, and using \eqref{eqn:v_low}, \eqref{eqn:v_up}, and the definition of $b_j^k$ gives us that 
	\begin{align}\label{eqn:doubling_diff}
		\sum_{k \ge 1 } \sum_{j=1}^{N_k} \left(  a_j^k(\gamma s_0) -  \gamma^{\Lambda_k} a_j^k(s_0)  \right)^{2} \le C_\gamma \eta s_0 (2s_0)^{2(\Lambda - \epsilon)},
	\end{align}
	whenever $r_1 < s_0 < 1/2$ and $\gamma \in (0, 1)$.	Crudely estimating a single term, recalling the definition of $b_j^k$, and multiplying by $\gamma^{- \Lambda_k}$ gives us the term-wise bound
	\begin{align}\label{eqn:doubling_diff_term}
		\abs{a_j^k(s_0) - \gamma^{-\Lambda_k} a_j^k(\gamma s_0)  } \le 2^{\Lambda + \Lambda_k} C \,    \eta^{1/2} s_0^{\Lambda - \epsilon + 1/2} ,
	\end{align}
	for any $j, k$ and any $r_1 < s_0 < 1/2$ and $\gamma \in (1/2, 1)$.
	
	Using our previous computations, we first estimate the growth of $a_j^k(s)$ for $k < i$; we claim that for any $1 \le k< i $, $1 \le j \le N_k$, and $s \in [r_1, 1/2]$, we have 
	\begin{equation}\label{eqn:a_growth}
		\abs{ a_j^k(s) -  (s/r_1)^{\Lambda_k} a_j^k(r_1)  } \le C_\Lambda \eta^{1/2} (H(s) s^{-d})^{1/2}
	\end{equation}
	where $C_\Lambda$ depends only on $\Lambda$.  To prove \eqref{eqn:a_growth}, let us assume for simplicity that $r_1 = 2^{-\ell}$ and $r_1 < s =2^{-\ell + m} \le 1/2$ for some integers $\ell, m$, though it is easy to prove the general case of \eqref{eqn:a_growth} from the flexibility of \eqref{eqn:doubling_diff_term} with the parameter $\gamma$. Now fix any such triple $k$, $j$, $s$. We apply the inequality \eqref{eqn:doubling_diff_term} with $\gamma = 1/2$ to the telescoping series
	\begin{align*}
		\abs{ a_j^k(s) -  (s/r_1)^{\Lambda_k} a_j^k(r_1)  } & = \abs{ a_j^k(2^{-\ell+m}) - 2^{m \Lambda_k}  a_j^k(2^{-\ell})  } \\
		& \le \sum_{i=0}^{m-1} 2^{i \Lambda_k}\abs{ a_j^k(2^{-\ell+m-i})  -   2^{\Lambda_k}   a_j^k(2^{-\ell + m-i-1}) } \\
		& \le 2^{\Lambda + \Lambda_k} C  \eta^{1/2}   \sum_{i=0}^{m-1} 2^{i (\Lambda + \epsilon)} (2^{-\ell +m - i})^{\Lambda - \epsilon + 1/2} \\
		& \le C_\Lambda \eta^{1/2} 2^{(-\ell + m)(\Lambda - \epsilon +1/2)}  \sum_{i=0}^{m-1} 2^{i(2\epsilon - 1/2)} \\
		& \le C_{\Lambda} \eta^{1/2} s^{\Lambda +1/3},
	\end{align*}
	since $\epsilon < 1/10$, which gives \eqref{eqn:a_growth} after a simple application of Lemma \ref{lem:H_size}.
	
	Next we move on to the estimate $a_j^k(s)$ for $k > i$. Fix $\gamma \in (0,1)$ small and to be determined. We show that for each $s \in [r_1, \gamma]$, we have the estimate
	\begin{equation}\label{eqn:a_growth_high}
		\sum_{k > i} \sum_{j=1}^{N_k} a_j^k(s)^2 \le \tau H(s)s^{-d}
	\end{equation}
	provided that $\epsilon$ and $\eta$ are chosen small enough depending on $\Lambda$ and $\tau$. Indeed, by virtue of \eqref{eqn:doubling_diff}, the representation $u$ in polar coordinates, and Lemma \ref{lem:H_size}, we may estimate
	\begin{align*}
	\sum_{k >i} \sum_{j=1}^{N_k} a_j^k(s)^2 & \le 2 \sum_{k > i} \sum_{j=1}^{N_k} \left( a_j^k(s) - \gamma^{\Lambda_k} a_j^k(\gamma^{-1}s)  \right) ^2 + \gamma^{2{\Lambda_k}}	(a_{j}^k(\gamma^{-1} s))^2  \\
	& \le  C_{\Lambda, \gamma} \eta s^{2(\Lambda - \epsilon) + 1} + C \gamma^{2\Lambda_{i+1}} \sum_{i > k} (a_j^k(\gamma^{-1} s))^2 \\
	& \le  C_{\Lambda, \gamma}  \eta s^{2(\Lambda-\epsilon) + 1} + C \gamma^{2\Lambda_{i+1}} H(\gamma^{-1} s) (\gamma^{-1} s)^{-d} \\
	& \le H(s) s^{-d} \left(  C_{\Lambda, \gamma}  \eta  + C \gamma^{2(\Lambda_{i+1} - \Lambda - \epsilon)}  \right).
	\end{align*}
	Setting $\delta_0 \coloneqq \Lambda_{i+1} - \Lambda >0$, then provided that $\epsilon < \delta_0 /2$, we first choose $\gamma$ small enough so that $C \gamma^{\delta_0} < \tau/2$, and then $\eta$ small enough so that $C_{\Lambda, \gamma} \le \tau/2$, which proves \eqref{eqn:a_growth_high}.
	
	With our estimates thus far we conclude. We define explicitly our approximation: set
	\begin{align*}
		\tilde{w}(s, \theta) \coloneqq \sum_{1 \le k \le i} \sum_{j=1}^{N_k} a_j^k(r_1)(s/r_1)^{\Lambda_k} \phi_j^k(\theta),
	\end{align*}
	so that by Corollary \ref{cor:ann_soln}, we know that $w$ is a solution of \eqref{eqn:soln} (and in fact, is a finite sum of homogeneous solutions). Using the orthogonality of $\phi_j^k$ in $L^2(\partial B_1, d\sigma_w)$, we readily estimate using \eqref{eqn:a_growth} and \eqref{eqn:a_growth_high},
	\begin{align*}
		s^{-d} \int_{\partial B_s} (u - \tilde{w})^2 \; d\sigma_w & = \sum_{1 \le k \le i} \sum_{j=1}^{N_k} (a_j^k(s) -  (s/r_1)^{\Lambda_k} a_j^k(r_1)    )^2 +  \sum_{k > i} \sum_{j=1}^{N_k} a_j^k(s)^2 \\
		& \le C_\Lambda \eta \left( \sum_{1 \le k \le i} N_k \right) H(s)s^{-d} + \tau H(s)s^{-d} = H(s)s^{-d} (\tilde{C}_\Lambda \eta + \tau)
	\end{align*}
	for any $s \in [r_1, \gamma]$. Taking $\eta$ small enough depending on $\tau$ and $\Lambda$ then proves the estimate:
	\begin{equation}\label{e:tildeestimate}
	s^{-d} \int_{\partial B_s} (u - \tilde{w})^2 \; d\sigma_w \leq \tau H(s)s^{-d}
	\end{equation}

	We now turn to showing that the lower order terms are not so significant, that is: \begin{equation}\label{e:nolower}\sum_{k < i} \sum_{j=1}^{N_k} a_j^k(s)^2 \leq \tau^{1/2} H(s)s^{-d}\end{equation} for $s \in [Mr_1, \gamma]$ (recall again that the frequency around which we are pinched, $\Lambda = \Lambda_i$), as long as $\tau$ is chosen small enough (depending only on the ellipticity and the ambient dimensions) and $M$ is large enough depending on $\tau$. 
		
	Assume that \eqref{e:nolower} does not hold for some $s\in [Mr_1, \gamma]$, and use the triangle inequality to get $$C\left(s^{-d}\int_{\partial B_s} (u-\tilde{w})^2\; d\sigma_w + \sum_{k < i} \sum_{j=1}^{N_k} (a_j^k(r_1))^2(s/r_1)^{2\Lambda_k} \right) \geq \sum_{k < i} \sum_{j=1}^{N_k} a_j^k(s)^2 \gtrsim \tau^{1/2} H(s)s^{-d}.$$ Using \eqref{e:tildeestimate} this implies that $$\sum_{k < i} \sum_{j=1}^{N_k} (a_j^k(r_1))^2(s/r_1)^{2\Lambda_k} \gtrsim \tau^{1/2} H(s)s^{-d}$$ as long as $\tau$ is chosen small enough.
	
From here, we can use Lemma \ref{lem:H_size} (but on the interval $[r_1, s]$) to get that $$\begin{aligned} \left(\frac{s}{r_1}\right)^{2\Lambda_{i-1}} \sum_{1 \leq k < i} \sum_{j=1}^{N_k} a_j^k(r_1)^2 &\gtrsim  \tau^{1/2} H(s)s^{-d}\\
&\gtrsim \tau^{1/2} \left(\frac{s}{r_1}\right)^{2(\Lambda -\epsilon)}H(r_1)r_1^{-d}\\
&\gtrsim \tau^{1/2}\left(\frac{s}{r_1}\right)^{2(\Lambda -\epsilon)}\sum_{1 \leq k < i} \sum_{j=1}^{N_k} a_j^k(r_1)^2.\end{aligned}.$$

Since $\epsilon > 0$ is small enough we have that $2(\Lambda - \epsilon) > 2\Lambda_{i-1}$ (see Theorem \ref{thm:homogen_sol}) which is a contradiction since $s/r_1 \geq M$ which is large, depending on $\tau$. Thus we have shown \eqref{e:nolower}. 

To finish set $$w_0(s, \theta) \coloneqq \sum_{j=1}^{N_i} a_j^i(r_1)(s/r_1)^{\Lambda_i} \phi_j^i(\theta).$$ Then using the orthogonality of the $\phi_j$'s we estimate that $$s^{-d} \int_{\partial B_s} (u - w_0)^2 \; d\sigma_w \leq \sum_{k < i} \sum_{j=1}^{N_k} a_j^k(s)^2 +  \sum_{j=1}^{N_i} (a_j^i(s) -  (s/r_1)^{\Lambda_i} a_j^i(r_1))^2.$$ The first term is small by \eqref{e:nolower} and the second term we have already shown is small when proving \eqref{e:tildeestimate}. So we are done. 

\end{proof}

To transfer the smallness from $L^2$ to $L^\infty$ requires just a bit of PDE. We separate the argument into this lemma for the reader's convenience. 

\begin{lemma}\label{lem:pinch_homogen}
	For each $\tau \in (0,1)$ {small enough} and $\Lambda_0 > 1$, there exists $\gamma = \gamma(\tau, \Lambda_0) \in (0,1)$, $\epsilon = \epsilon(\Lambda_0, \tau) \in (0,1)$, and $\eta = \eta(\tau, \Lambda_0) \in (0,1)$ small enough so that the following holds.
	
	Suppose that $u$ solves \eqref{eqn:soln_var}, and in addition we have \eqref{eqn:ellipse_cond} and the bound \eqref{eqn:unif_bound}. If $u$ is $\epsilon$-pinched about $\Lambda \in \mathcal{F}$ at $X_0 \in B_1 \cap \R^d$ in the scales $[r_1, 1]$ for $r_1 < 100^{-1}$ (with $\Lambda \leq \Lambda_0$), then $\dist_{[ \gamma^{-1} r_1, \gamma]}^{X_0}(u, H_{\Lambda}^2) < \tau$.

\end{lemma}

\begin{proof}
We can assume for notational simplicity that $X_0 = 0$. What remains is to transfer the $L^2(\partial B_s, d\sigma_w)$ estimate in Lemma \ref{lem:fourier} to an $L^\infty$ estimate {(note we are abusing the notation $\tau$ here a little bit, but it does not matter what power $\tau$ comes with).} We do this by first passing to an $L^2$ estimate on the bulk $B_s$, and then appealing to the Moser-type inequality in Lemma \ref{lem:moser} for the difference $u-w_0$. 

We note that $v := (u-w_0)$ is a strong solution to 
\begin{align}
	\begin{alignedat}{4}
		-\divv(A \nabla v) & = -\mathrm{div}(\delta(X)^{-n+d+1}f) & \text{ in } & B_1 \cap (\R^n \setminus \R^d), \\
		v & = 0 & \text{ on }  & B_1 \cap \R^d,
	\end{alignedat}
\end{align}
with $f = ( I - \cal A)\nabla w_0$ and $0 < r <1$. At this stage, it is convenient to make the normalizing assumption that $H_u(\gamma) \gamma^{-d} = 1$, so that $H_{w_0}(\gamma) \gamma^{-d} \simeq 1$. Using Moser inequality and perhaps shrinking $\gamma$ we have that for any $\rho \leq \gamma$ $$\sup_{B_\rho} |v| \leq \left \{ \left(\fint_{B_{2\rho}} v^2 \; dm  \right)^{1/2} + \rho \left( \fint_{B_{2\rho}}  \abs{f}^{q_0} \; dm \right)^{1/q_0}     \right\}$$ for some $q_0 > 1$. Using the homogeneity of $w_0$ we have that $$|f|(X) \leq C\eta|X|^{\Lambda-1}\rho^{-\Lambda}|\nabla w_0(\rho \frac{X}{|X|})|.$$ Using that $H^\lambda_{\Lambda_0}$ is compact and $w_0$ is homogeneous we have  \begin{equation}\label{e:estonf1} \left( \fint_{B_{2\rho}}  \abs{f}^{q_0} \; dm \right)^{1/q_0} \leq C\eta \rho^{-1}\left(H_{w_0}(\rho)\rho^{-d}\right)^{1/2}. \end{equation} As long as $\rho \in [Mr_1, \gamma]$ we can use the comparability of $H_{w_0}(\rho)$ and $H_u(\rho)$ (which follows from Lemma \ref{lem:fourier}) and then apply Lemma \ref{lem:equi_distr_var} to get \begin{equation}\label{e:estonf2} \left( \fint_{B_{2\rho}}  \abs{f}^{q_0} \; dm \right)^{1/q_0} \leq C\eta \rho^{-1}\sup_{B_\rho} |u|, \qquad \forall \rho \in [Mr_1, \gamma]. \end{equation}

Fix $k \in \N$ to be determined, and, assuming $\rho > \gamma^{-1}r_1$ we may assume (again possibly shrinking $\gamma$) that $2^{-k}\rho, \rho \in [Mr_1, \gamma]$ (where $M$ is given by Lemma \ref{lem:fourier}). First, remark that Corollary \ref{cor:doubling_var} and homogeneity tells us that there is a constant $C_2 > 1$ for which 
\begin{align}\label{eqn:small_ball}
	\fint_{B_{\rho/2^k}} u^2 \; dm \le C C_2^{-k(\Lambda - \epsilon)} \fint_{B_\rho} u^2 \; dm, \qquad \fint_{B_{\rho/2^k}} w_0^2 \; dm \le C C_2^{-k(\Lambda - \epsilon)} \fint_{B_\rho} w_0^2 \; dm.
\end{align}
On the other hand, integrating the conclusion of Lemma \ref{lem:fourier} for $s \in [\rho/2^k, \rho]$ and applying Lemma \ref{lem:equi_distr_var} tells us that if $\eta$ is chosen sufficiently small, then there is a constant $C_1 >1$ (depending on our upper bound of the frequency, $\Lambda_0$) for which 
\begin{align}\label{eqn:annulus}
	\fint_{B_{\rho} \setminus B_{\rho/2^k}} (u-w_0)^2 \; dm \le C \tau^{1/2} \fint_{B_\rho} u^2 \; dm \le C C_1 \tau \sup_{B_{\rho/2}} \abs{u}^2.
\end{align}
Choosing $k$ large enough depending on $\tau$ and $\Lambda$ so that $C_2^{-k(\Lambda - \epsilon)} < \tau^{1/2}$, one combines the previous two inequalities to deduce
\begin{align} \label{eqn:L2_bulk}
	\fint_{B_\rho} (u-w_0)^2 \; dm \le C C_1 \tau^{1/2} \sup_{B_{\rho/2}}\abs{u}^2, \; \; \rho \in [\gamma^{-1}r_1, \gamma].
\end{align}
Note that $\gamma$ ultimately depends on $M, k$ which in turn depend on $\tau, \Lambda$ so there is no circularity here. 

Plugging this estimate, and \eqref{e:estonf2} into our Moser inequality gives that $$\sup_{B_\rho} |u-w_0| \leq C\tau^{1/4}\sup_{B_\rho} |u| + C\eta \sup_{B_\rho}|u|.$$ Recalling the normalization in Definition \ref{defn:dist_homogen} and letting $\eta$ be small enough, we are done. 
\end{proof} 

One consequence of our pinching estimate is that if the frequency is close enough to 1 in a ball, there can be no singular points in a smaller interior ball. Before we present this corollary, we remind the reader that by Theorem \ref{thm:homogen_sol}, we know that $\Lambda_1 = 1$ and $\Lambda_2 \ge (1 + \sqrt{5})/2 > 3/2$. Since the only non-trivial $\Lambda_1$-homogeneous solution of \eqref{eqn:soln} vanishing on $\R^d$ is $\abs{t}$ (up to a multiplicative constant), and since $\abs{\nabla (\abs{t})}$ is non-vanishing in $\R^n$, we can deduce the following. 
\begin{lemma}[Lack of singular points]\label{lem:sing_empty}
Assume $u$ is a solution of \eqref{eqn:soln_var}, with \eqref{eqn:ellipse_cond} and $N_u(0, 1) \le 1 + \epsilon$ with $\epsilon$ and $\eta$ sufficiently small. Then there is some $r_0 = r_0(\eta, \epsilon) \in (0,1)$ so that $Z(u) \cap B_{r_0} \subset \R^d$ and moreover $\sing(u) \cap B_{r_0} = \emptyset$. Recall that $Z(u)$ denotes the nodal set of $u$.
\end{lemma}
\begin{proof}
First, we apply Corollary \ref{cor:uniform_var} to see that if $\eta$ is sufficiently small, then in fact we have for some $\delta > 0$ small,
\begin{align}\label{eqn:unif_bd2}
	N_u(Y, r) \le 1 + 2\epsilon, \; \; 0 < r < \delta, \; Y \in \R^d \cap B_\delta. 
\end{align}
Notice that this automatically gives us that $\sing(u) \cap B_\delta \subset \R^n \setminus \R^d$, since $\sing(u) \cap \R^d$ coincides with the set of points $Y$ for which $N_u(Y, 0^+) \ge \Lambda_2$. To prove the Lemma, it suffices to show the restricted cone containment,
\begin{align}\label{eqn:cone_zeros}
	\{u=0\} \cap B_{\delta}(Y) \cap \{\delta(X) \ge \abs{X-Y}/2\} = \{0\},
\end{align}
since then one may union the above over all $Y \in \R^d \cap B_\delta$ to obtain $Z(u) \cap B_\delta \subset \mathbb R^d$. After a harmless change of variables, we may prove \eqref{eqn:cone_zeros} for $Y= 0$ and also assume that $\A(0) = I$.

By the almost-monotonicity of $N_u$ as in Lemma \ref{lem:monotone_var}, we have
\begin{align*}
	N_u(0,r) \ge N_u(0, s) e^{-C (r-s)} \ge N_u(0,s) e^{-C \delta}
\end{align*}
for $0 < s < r < \delta$. It follows that if there is $0 <r < \delta$ with $N_u(0, r) \le (1-2 \epsilon)$, then $N_u(0, 0^+) \le (1 - 2 \epsilon)(1 + C \delta)$. By choosing $\delta$ even smaller if necessary (which still preserves \eqref{eqn:unif_bd2}) we see that $N_u(0, 0^+) \le 1-\epsilon$, which is a clear contradiction of the fact that $N_u(0, 0^+)$ agrees with the order of vanishing of $u$ at $Y$, which must be greater than or equal to $1$. We thus have shown that the frequency of $u$ is $(2\epsilon)$-pinched about $1 = \Lambda_1 \in \cal{F}$ at $0 \in \R^d$ in the scales $[0, \delta]$. If we fix $0 < \tau < 1/10$ and choose $\epsilon$ small enough depending on $\tau$, then we may apply Lemma \ref{lem:pinch_homogen} to see that there is a $\gamma \in (0,1)$ and a $1$-homogeneous solution $w$ of the constant coefficient equation \eqref{eqn:soln} so that 
\begin{align}\label{eqn:lin_appx}
	\sup_{B_r} \abs{u-w} \le \tau \sup_{B_r} \abs{u}, \; \;  0 <r <  \gamma \delta.
\end{align}
It is easy to see that the above implies 
\begin{align}
	\sup_{B_r} \abs{w} \le (1+\tau) \sup_{B_r} \abs{u}, \; \; \sup_{B_r} \abs{u} \le (1-\tau)^{-1} \sup_{B_r} \abs{w}, \; \; 0 < r < \gamma \delta.
\end{align}
This implies that $w$ is a non-trivial $1$-homogeneous solution of \eqref{eqn:soln} since $u$ is non-trivial. As per Theorem \ref{thm:homogen_sol}, it follows that $w(X) = c \delta(X) = c \abs{t}$ for some $c \in \R$. Notice that by renormalizing $u$, we may assume that $c = 1$, and thus from \eqref{eqn:lin_appx} we see that if $X \in B_{\gamma \delta}$ satisfies $\delta(X) \ge \abs{X}/2$, then
\begin{align*}
	\abs{u(X) - \delta(X)} \le \tau (1 - \tau)^{-1} \sup_{B_{\abs{X}}} \abs{w} \le \tau (1-\tau)^{-1} \sup_{B_{2\delta(X)}} \abs{w} \le 2 \tau (1-\tau)^{-1} \delta(X).
\end{align*}
Since $\tau < 1/10$, we see that $2^{-1} \le u(X) / \delta(X) \le 2$ for $X \in B_{\gamma \delta}$ satisfying $\delta(X) \ge \abs{X}/2$, which proves our claim \eqref{eqn:cone_zeros} and thus the Lemma.
\end{proof}

With the previous Lemma in hand, we can prove by a contradiction-compactness argument the following.
\begin{lemma}[Upper semi-continuity of the singular set]\label{lem:cont_sing}
Fix $\Lambda_0 >0$, $\tau \in (0,1)$, and $v \in H_\Lambda^2$, and assume that $u$ is a solution of \eqref{eqn:soln_var} with \eqref{eqn:ellipse_cond}, \eqref{eqn:unif_bound}, and $\norm{u}_{L^\infty(B_1)} = 1$. If $\epsilon$ and $\eta$ are sufficiently small, depending only on $\tau$, $\Lambda_0$, and $v$, then $\norm{ u - v}_{L^\infty(B_1)} < \epsilon$ implies
\begin{align*}
\sing(u) \cap B_1 \subset B_\tau(\sing(v)).
\end{align*}
\end{lemma}
\begin{proof}
We argue by contradiction and compactness. If the conclusion of the Lemma is not valid, then there exists $\tau_0$, $\Lambda_0$, and $v \in H_\Lambda^2$ with $\Lambda \le \Lambda_0$ and a sequence of functions $u_m \in W_r(B_{10})$ solving
\begin{align*}
	-\divv( A_m \nabla u_m) & = 0 \text{ in } B_{10}, \\
	u_m & = 0 \text{ on } \R^d \cap B_{10},
\end{align*}
where $A_m$ satisfies the condition \eqref{eqn:ellipse_cond} with $\eta_m = 1/m$, $ \norm{u}_{L^\infty(B_1)} =1$, $\norm{u_m -v}_{L^\infty(B_1)} < 1/m$, and $u_m$ is such that 
\begin{align*}
	N_{u_m}(Y, r) \le \Lambda_0, \; \;  \text{ whenever } E_r(Y) \subset B_{10}.
\end{align*}
However, for each $m \in \N$, there is a point $X_m \in \sing(u_m) \cap B_1$ for which $\dist(X_m, \sing(v)) \ge \tau_0$. 

Applying Lemma \ref{lem:compact_soln} (along with Lemma \ref{lem:equi_distr_var} to guarantee non-degeneracy), we see that up to taking a subsequence (still labeled $u_m$ for convenience), the $u_m$ converge uniformly on compact subsets of $B_5$ to some solution $u_\infty$ of \eqref{eqn:soln} vanishing continuously on $\R^d$. Moreover, we may assume $X_m \ra X$ for some point $X \in \overline{B_1} \setminus \sing(v)$. Of course since $\norm{u_m -v}_{L^\infty(B_1)} <m^{-1}$, it must be that $u_\infty = v$. Since the coefficients $\A_m = A_m \abs{t}^{n-d-1}$ are uniformly Lipschitz, we can apply standard elliptic regularity theory and assume also that $\nabla u_m \ra \nabla v$ uniformly on compact subsets of $B_5 \cap (\R^n \setminus \R^d)$. In the case that $X \not \in \R^d$, then the uniform $C^1$ convergence of $u_m$ to $v$ away from $\R^n \setminus \R^d$ readily gives that $v(X) = \nabla v(X) = 0$, and thus $X \in \sing(v)$, which contradicts our choice of the points $X_m$.

Hence it must be the case that $X \in \R^d \cap \overline{B_1}$, with $X \not \in \sing(v)$. As such, we know that $N_v(X, 0^+) = 1$, and so since $N_v(X,r)$ is monotone in $r$ by Lemma \ref{lem:monotone_flat_bdry}, given any $\epsilon >0$, there is $s_0 > 0$ for which $N_v(X, r) \le 1 + \epsilon$ for all $0 < r < s_0$. The conclusion of Lemma \ref{lem:compact_soln} shows that
\begin{align*}
	N_{u_m}^{A_m}(X, s_0) \le 1 + 2\epsilon
\end{align*}
for all $m$ sufficiently large. However on the other hand, applying Lemma \ref{lem:sing_empty} in $B_{s_0}(X)$ to the solution $u_m$ for $\epsilon$ sufficiently small, we see that 
\begin{align*}
	\sing(u_m) \cap B_{r_0 s_0}(X) = \emptyset,
\end{align*}
which is a clear contradiction to the fact that $X_m \in \sing(u_m) \ra X$. This finishes the proof of the Lemma.																								 
\end{proof}

Before stating one of the main results of this section, we need to introduce some definitions. Whenever $V \in G(n,d)$, $Y \in \R^n$, and $0 < \alpha < 1$, we use the notation
\begin{align*}
	X(Y, V, \alpha) \coloneqq \{ Z \in \R^n \; : \; \dist(Z-Y, V) < \alpha \abs{Z-Y } \},
\end{align*}
to denote the cone about $V$ centered at $Y$ with aperture $\alpha$. We also denote $X(0, V, \alpha) \equiv X(V, \alpha)$ for the cone centered at the origin, so that in this notation, $X(Y,V,\alpha) \equiv Y + X(V, \alpha)$. We shall also need two geometric estimates (Lemmas \ref{lem:ball covering} and \ref{lem:cones}) which are independent of the solutions of our PDE at hand, and only concern Hausdorff measure estimates and conical properties of sets. As such, we leave them for the end of the Section, but shall soon apply them in our covering arguments, such the following Corollary (see Figure \ref{fig:cone} for a schematic picture).

\begin{figure}
	\begin{tikzpicture}[scale=2]
		% planes and dashed cone
		\draw[->] (-2,0)--(2,0) node[right] {$\R^d$};
		\draw[->] (-1, 1.5)--(1, -1.5) node[below right] {$V_j$};
		\draw[dashed] (0,0)--(1.5,-0.9);
		\draw[dashed] (0,0)--(0.2,-1.5);
		
		% circles
		\draw[red, thick] (0,0) circle (1.2);
		\draw[red, thick] (0,0) circle (0.3);		
		\node[red, above right] at (0.9, 0.9) {$B_\gamma(X_0)$};
		\node[red, above right] at (-0.25, 0.3) {$B_{\gamma^{-1} r_1}(X_0)$};
		
		% singular curve
		\draw[blue, thick] plot [smooth] coordinates{
		(0,0) (0.17, -0.03) (0.14, -0.12) (0.35, -0.4) (0.2, -1) (1, -1) (1.3, -0.5)} node[below right] {$\sing(u)$};
	\end{tikzpicture}
	\caption{The relative position of the singular set of $u$ (blue) inside an annulus, guaranteed by Corollary \ref{cor:pinch_cone}.}\label{fig:cone}
\end{figure}

\begin{cor}\label{cor:pinch_cone}
For each $\Lambda_0 > 0$, there exists $\alpha = \alpha(\Lambda_0) \in (0,1)$ (potentially close to $1$), $\gamma = \gamma(\Lambda_0) \in (0,1)$, and finitely many $(n-2)$-planes $\{V_j\}_{j=1}^K \subset G(n, n-2)$ with $K = K(\Lambda_0) \in \N$ so that the following holds. 

Suppose that $u$ is a solution of \eqref{eqn:soln_var} with \eqref{eqn:ellipse_cond} and \eqref{eqn:unif_bound}. Assume in addition that the frequency of $u$ is $\epsilon$-pinched about $\Lambda$ at $X_0 \in B_1 \cap \R^d$ in the scales $[r_1, 1]$ for $r_1 < 100^{-1}$. Then if $\eta$ and $\epsilon$ are small enough, depending on $\Lambda_0$, then there exists an index $1 \le j \le K$ for which 
\begin{align*}
\sing(u_s^{X_0}) \cap \partial B_1 \subset X(V_j, \alpha)
\end{align*}
for all $s \in [\gamma^{-1} r_1, \gamma]$, i.e.,
\begin{align*}
\sing(u) \cap \left( B_{\gamma}(X_0) \setminus B_{\gamma^{-1} r_1} (X_0)  \right) \subset X(X_0, V_j, \alpha).
\end{align*}
\end{cor}
\begin{proof}
Fix $\tau_1 \in (0,1)$ small and to be determined. Recall from Remark \ref{rmk:h_compt} that $H_\Lambda^2 \subset C(B_1)$ is compact, and thus we may choose $K(\tau_1) \in \N$ elements, $\{v_1, \dotsc, v_K \} \subset H_\Lambda^2$ which are $\tau_1$-dense in $H_\Lambda^2$. That is, for each $v \in H_\Lambda^2$, there is some $1 \le j \le K$ for which $\norm{v - v_j}_{L^\infty(B_1)} < \tau_1$. Combining Lemma \ref{lem:cones} for $F = \sing(v_j)$ with the measure estimates from Theorem \ref{thm:homogen_m}, we see that there is some small $\tau_2 = \tau_2(n, \Lambda_0) \in (0,1)$ so that for each $1 \le j \le K$, one can find a $2$-plane, $W_j \in G(n, 2)$ for which 
\begin{align*}
\sing(v_j) \cap \partial B_1 \cap B_{\tau_2}(W_j) = \emptyset.
\end{align*}
In other words, with $V_j = W_j^\perp \in G(n,n-2)$, we have 
\begin{align} \label{eqn:vj_sep}
\sing(v_j) \cap \partial B_1 \subset B_{\sqrt{1 - \tau^2_2}}(V_j).
\end{align}

Now let $u$ be such a solution as in the statement of the Corollary. By Lemma \ref{lem:pinch_homogen} and the choice of the elements $\{v_i\}_{i=1}^K$, we know that if $\eta$ and $\epsilon$ are chosen small enough depending on $\Lambda_0$ and $\tau_1$, then there exists some index $1 \le j \le K$ for which 
\begin{align*}
\norm{ u_{s}^{X_0} - v_j  }_{L^\infty(B_1)} < 2 \tau_1
\end{align*}
for all $s \in [\gamma^{-1} r_1, \gamma]$ where $\gamma \in (0,1)$ is a small parameter depending on $\Lambda_0$ and $\tau_1$. By taking $\tau_1$ and $\eta$ small enough depending on $\tau_2$ and $\Lambda_0$, Lemma \ref{lem:cont_sing} then says that for all such $s$, we have
\begin{align*}
\sing(u_s^{X_0}) \cap B_1  \subset B_{\tau^2_2/8}(\sing(v_j)).
\end{align*}
The scale invariance of $\sing(v_j)$ (since $v_j$ is $\Lambda$-homogeneous) then implies
\begin{align*}
 \sing(u_s^{X_0}) \cap \partial B_1 \subset B_{\tau^2_2/4} ( \sing(v_j) \cap \partial B_1 )
\end{align*}
as long as $\tau_2$ is chosen sufficiently small. Combining this containment with \eqref{eqn:vj_sep}, we see that 
\begin{align*}
\sing(u_s^{X_0}) \cap \partial B_1 \subset B_{1 - \tau^2_2/4}(V_j),
\end{align*}
which completes the proof of the first claim with $\alpha = 1 - \tau^2_2/4$. Of course, the second claim follows from the first one, once one realizes that $\sing( u_s^{X_0}) \cap \partial B_1 = \sing(u) \cap \partial B_s (X_0)$.
\end{proof}

We end our discussion on frequency pinching with a lemma on the frequency drop. Essentially, $N_u$ must drop in frequency a definite amount once it dips below one of the frequencies $\Lambda \in \mathcal{F}$. For solutions to the constant coefficient equation \eqref{eqn:soln}, this follows from some straightforward computations using the Fourier decomposition (see Lemma \ref{lem:N_polar}). For solutions to more variable coefficient equation, one can reduce to the case of ``constant coefficients" using contradiction-compactness techniques as above. Since these results are well know to experts, we omit the proof. 

\begin{lemma}[Frequency drop]\label{lem:freq_drop}
Fix $\Lambda_0 > 0$ and $\epsilon > 0$ sufficiently small (depending on $\Lambda_0$), and assume \eqref{eqn:soln_var}, \eqref{eqn:ellipse_cond}, and \eqref{eqn:unif_bound}. Suppose that $\Lambda_{k+1} \in \mathcal{F}$, and that the solution $u$ satisfies
\begin{align*}
N_u(0, 1) \le \Lambda_{k+1} - \epsilon.
\end{align*}
Then if $\eta = \eta( \epsilon, \Lambda_0)$ is chosen small enough, there is some $\rho_1 = \rho_1(\epsilon, \Lambda_0) \in (0,1)$ small for which $N_u(0, s) \le \Lambda_{k} + \epsilon$ for all $0 \le s \le \rho_1$.
\end{lemma}

Finally, for completeness, we provide the proofs of the following geometric lemmas.

\begin{lemma}\label{lem:ball covering}
	Fix $L   > 2 $, and $\alpha \in (0,1)$. Let $X_1, \dotsc, X_N \in B_1$ be any finite number of points, and suppose that $\{B_i = B_{t_i}(Y_i)\}_{i=1}^K$ is a finite family of pairwise disjoint balls with radii $t_i \coloneqq L^{-1} \min_{1 \le j \le N} \abs{Y_i - X_j}$. If there is a $d$-plane $V \in G(n,d)$ so that for each $1 \le i \le K$ and $1 \le j \le N$, 
	\begin{align*}
		B_i \subset  B_1(0) \cap X(X_j, V, \alpha), 
	\end{align*}
	then there is a constant $C = C(n, d, \alpha, L) > 0$ so that $\sum_{i=1}^K t_i^{d} \le C$, with $C$ independent of $N$ and $K$.
\end{lemma}
\begin{proof}
	Notice that it suffices to show that for any $X \in V$,
	\begin{align}\label{eqn:bdd_proj}
		\abs{ \{ 1 \le i \le K \; : \;  \pi_V^{-1}(X) \cap  B_i \ne \emptyset   \} } \le C
	\end{align}
	for some constant $C >0$ depending only on $n,d, \alpha$ and $L$. Indeed \eqref{eqn:bdd_proj} is equivalent to $\sum_{1 \le i \le K} \chi_{\pi_V(B_i)} \le C$, which along with $B_i \subset B_1(0)$ implies that 
	\begin{align*}
		\sum_{i=1}^K t_i^{d} & = C \sum_{i=1}^K \HD^{d}(\pi_V(B_i))  =  C\int_{V} \sum_{k=1}^K \chi_{\pi_V(B_i)}  \; d\HD^{d} \le C \int_{B_1(0) \cap V} d \HD^{d} \le C.
	\end{align*}
	
	To show \eqref{eqn:bdd_proj}, we will show first that the balls $B_i$ with overlapping projections have comparable radii. To this end, let $Y, Z \in B_1$  be points for which $\pi_V(Y) = \pi_V(Z)$ and $Y, Z \in X(X_j, V, \alpha)$ for each $1 \le j \le N$. Let $1 \le i \le N $ be an index minimizing $\{\abs{Y - X_j} \; : \; 1 \le j \le N\}$. First, we claim that the cone condition on $Z$ implies that 
	\begin{align}\label{eqn:min_point}
		\min \{ \abs{Z - X_j} \; : \; 1 \le j \le N\} \simeq \abs{Z - X_i},
	\end{align}
	with implicit constant depending only on $\alpha$. Indeed, for any index $1 \le j \le N$, notice that $Y \in X(X_j, V, \alpha)$ implies that 
	\begin{align*}
		\abs{\pi_V(Y - X_j)} \le \abs{Y- X_j} \le \abs{ \pi_V(Y - X_j)  } + \abs{ \pi_{V^\perp}(Y- X_j)} \le \abs{\pi_V(Y -X_j)} + \alpha \abs{ Y - X_j },
	\end{align*}
	from which we conclude $\abs{Y- X_j} \simeq_\alpha \abs{ \pi_V(Y - X_j)}$. Since the same holds for $Z$, the fact that $\pi_V(Z) = \pi_V(Y)$ readily gives \eqref{eqn:min_point}. 
	
	Since $L > 2$, then the definition of $t_i$ readily implies for any $1 \le i \le K$ and $1 \le j \le N$,
	\begin{align*}
		\abs{Z- X_j} \simeq \abs{Y_i - X_j}
	\end{align*}
	for any choice of indices $1 \le i \le K$,$ 1 \le j \le N$ and any $Z \in B_i$. In particular, if $Y \in B_{i}$ and $Z \in B_{i'}$ where $1 \le i, i' \le K$ are such that $\pi_V(Y) = \pi_V(Z)$, then \eqref{eqn:min_point} and the previous observation imply that \[t_{i} \simeq \min_{1 \le j \le N} \abs{Y - X_j}  \simeq \min_{1 \le j \le N} \abs{Z - X_j} \simeq t_{i'},\] so that $B_i$ and $B_{i'}$ have comparable radii. Finally we conclude: fix $1 \le i \le K$ and choose $1 \le j \le N$ so that 
	\begin{align*}
		L t_i = \abs{Y_i - X_{j}} = \min_{1 \le k \le N} \abs{Y_i - X_k}.
	\end{align*}
	If we denote by $I \coloneqq \{ 1 \le i' \le K \; : \; \pi_V(B_{i'}) \cap \pi_V(B_i) \ne \emptyset \}$, then there is a constant $C  >1$ depending on $n$, $d$, $\alpha$, and $L$ for which $C^{-1} t_i \le \abs{Y_{i'} - X_j} \le C t_i$ and $C^{-1} t_i \le t_{i'} \le C t_i$ for each $i' \in I$. Using the disjointness of the $B_{i'}$ we then see that 
	\begin{align*}
		\L^n(B_{C t_i}(X_j)) & \ge \sum_{i' \in I} \L^n(B_{i'})\ge C^{-1} \abs{ B_{t_i}} \abs{I},
	\end{align*}
	so that $\abs{I} \le C$. Clearly \eqref{eqn:bdd_proj} follows from this estimate, completing the proof.
\end{proof}

\begin{lemma}[Volume estimates imply cones on spheres]\label{lem:cones}
	For each $\Lambda > 0$, there is a constant $\delta = \delta(\Lambda, n) >0$ so that the following holds. If $F \subset \S^{n-1}$ is a subset for which \[ \HD^{n-1}(B_\epsilon(F) \cap \S^{n-1}) \le \Lambda \epsilon^2, \qquad 0 < \epsilon < 1, \]
	then there is a $2$-plane $V \in G(n,2)$ for which $F \cap B_{\delta}(V) = \emptyset$.
\end{lemma}
\begin{proof}
	Notice that we may as well assume that $F = -F$, since $F \cup (-F) \subset \S^{n-1}$ also satisfies the main assumption of the Lemma with $2\Lambda$ in place of $\Lambda$. First, let us show the existence of $\delta_0 = \delta_0(\Lambda, n) >0$ and a point $Y \in \S^{n-1}$ for which 
	\begin{align}\label{eqn:disj_antipode}
		\left( B_{\delta_0}(Y) \cup B_{\delta_0}(-Y)  \right) \cap F = \emptyset.
	\end{align}

	Fix any direction $Z_0 \in \S^{n-1}$ and denote by $H = H(Z_0) \subset \S^{n-1}$ a spherical cap, $H \coloneqq \{ X  \in \S^{n-1} \; : \; X \cdot Z_0 > 1/2\}$. Assuming $0 < \delta_0 < 10^{-1}$ small to be determined below, we choose $m \ge c \delta_0^{-n+1}$ points $X_1, \dotsc, X_m \in H$ so that $\{B_{\delta_0}(X_i)\}_{i=1}^m$ is a collection of pairwise disjoint balls with centers in $H$. Denote by $I = \{ 1 \le i \le m \; : \; B_{\delta_0}(X_i) \cap F \ne \emptyset \}$, and notice by the disjointedness of the $B_{\delta_0}(X_i)$ and the fact that $B_{\delta_0}(X_i) \subset B_{C \delta_0}(F)$ for $i \in I$, we have
	\begin{align*}
		C^{-1} \abs{I} \delta_0^{n-1} & \le \sum_{i \in I} \HD^{n-1}( B_{\delta_0}(X_i) \cap \S^{n-1})  \\
		&  =  \HD^{n-1}\left( (\cup_{i \in I} B_{\delta_0}(X_i) ) \cap \S^{n-1} \right)  \le  \HD^{n-1} \left( B_{C \delta_0}(X_i)(F) \cap \S^{n-1} \right) \le  C \Lambda \delta_0^{2},
	\end{align*}
	and thus $\abs{I} \le C \Lambda \delta_0^{- n +3}$ for some constant $C >0$ depending only on $n$. Choosing $\delta_0$ small enough so that $C \Lambda \delta_0^{-n+3} < c \delta_0^{-n+1}$ (or in other words, $\delta_0^2 < (C\Lambda)^{-1}$), we see that $\abs{I} < m$. Thus there is some $1 \le j \le m$ with $j \not \in I$, so that \eqref{eqn:disj_antipode} holds for $Y = X_j$ since $F = -F$.
	
	Fix such a $Y$ and parametrize a family of $2$-planes containing $Y \in \S^{n-1}$ as follows: let $E$ denote the equator $E \coloneqq \{W \in \S^{n-1} \; : \; \dotp{Y, W} = 0\}$, so that $E \simeq \S^{n-2}$, and define $V(W) \coloneqq \mathrm{span}\{Y, W\} \in G(n, 2)$ for $W \in E$. Next, fix $0 < \delta \ll \delta_0$ small and to be determined, and choose $m \ge c \delta^{-n+2}$ points $W_1, \dotsc, W_m \in E$ so that $\{B_{\delta}(W_i)\}_{i=1}^m$ is a family of pairwise disjoint balls centered on $E$. Similar to before, denote by $I = \{ 1 \le i \le m \; : \; B_{\delta}(V(W_i)) \cap F \ne \emptyset\}$, and choose for each $i \in I$ some $Z_i \in B_\delta(V(W_i)) \cap F \setminus B_{\delta_0}(Y)$ and $X_i \in V(W_i) \cap \S^{n-1} \setminus B_{\delta_0/2}(Y)$ with $\abs{Z_i - X_i} \le C \delta$. Notice that such a choice of $X_i \in \S^{n-1}$ is possible provided $\delta$ is much smaller than $\delta_0$. It is a straight-forward computation to see that $X_i \in V(W_i) \cap \S^{n-1} \setminus B_{\delta_0/2}(Y)$ in fact implies $\abs{X_i - X_j} \ge c \delta_0 \abs{W_i - W_j} \ge c \delta \delta_0$ for some $c >0$. Hence we may conclude as before; for $r \coloneqq c \delta \delta_0 /2$, we see that $\{B_r(X_i)\}_{i \in I}$ is a family of pairwise disjoint balls and contained in $B_{C\delta}(F)$, and thus we estimate 
	\begin{align*}
		C^{-1} \abs{I} r^{n-1} & \le \HD^{n-1}(\cup_{i \in I} B_{r}(X_i) \cap \S^{n-1}) \le \HD^{n-1}(B_{C \delta}(F) \cap \S^{n-1}) \le C \Lambda \delta^2.
	\end{align*}
	Recalling the definition of $r$, one sees that $\abs{I} \le C \Lambda \delta^{-n+3} \delta_0^{-n+1}$, and so if $C \Lambda \delta^{-n+3} \delta_0^{-n+1} < c \delta^{-n+2}$, then there is some $1 \le j \le m$ for which $j \not \in I$, which completes the proof of the Lemma.
\end{proof}

\section{The main covering argument}\label{sec:covering}

In this section, we prove our main result regarding the size of $\sing(u)$ for solutions $u$ of equations of the type \eqref{eqn:soln_var}, i.e., we prove Theorem \ref{thm:main_flat}. The main tool used in the proof of this Theorem is the covering Lemma \ref{lem:cov}, which provides Minkowski estimates on $B_1 \cap (\sing(u) \setminus \R^d)$ at a discrete scale, much like Proposition 3.36 in \cite{NV17}. In our setting however, the proof is (and necessarily must be) considerably different. Since this covering argument is one of the main difficulties presented by Theorem \ref{thm:main}, let us spend more time explaining this difficulty and contrasting it with the case of harmonic functions (or more generally, solutions to elliptic PDEs with Lipschitz coefficients).

Suppose that $v$ is a harmonic function in $B_1$, and there are two distinct points $X_1, X_2 \in B_1$ for which the Almgren frequency function $N_v(X_1, r)$, $N_v(X_2, r)$ is constant on nontrivial intervals $I_1, I_2$ respectively. Then by a rigidity result for $N_v$, we know that $v( \, \cdot \, + X_1), v( \, \cdot \, + X_2)$ are homogeneous harmonic polynomials, and thus $v$ is homogeneous with respect to the points $X_1$ and $X_2$. As a consequence, $v$ is \textit{invariant} in the $X_2 -X_1$ direction:
\begin{align*}
v( Y + t(X_2 - X_1)) = v(Y), \; t \in \R,
\end{align*}
and thus the same holds true for $Z(v)$ (the nodal set of $v$) and $\sing(v)$. A quantification of the previous statement says that if there are many points $X_i \in B_1$ for which $N_v(X_i, r)$ is approximately constant on a sufficiently large interval, then the $X_i$ must be contained in a small neighborhood of an affine set (see Lemma 3.22 and Corollary 3.24 in \cite{NV17}). This quantitative cone-splitting is one of the key arguments used in \cite{NV17} to conclude $\sing(v)\cap B_{1/2}$ is contained in finitely many $(n-2)$-dimensional Lipschitz graphs, and thus has finite $\HD^{n-2}$ measure. The same sort of analysis can be done for solutions $v$ of $-\divv( A \nabla v) =0$ when $A$ is Lipschitz, because at small scales, $A$ is approximately constant and so solutions $v$ are well-approximated by harmonic functions. 

For solutions $u$ to equations of the type \eqref{eqn:soln_var} (or even  \eqref{eqn:soln_const}), we immediately have a problem at the level of the rigidity result for $N_u$. While it is true that $N_u(X, r) \equiv \Lambda$ for a nontrivial interval $r \in I \subset [0,\infty)$ implies that $u$ is $\Lambda$-homogeneous when $X \in \R^d \cap B_1$, (Lemma \ref{lem:monotone_flat_bdry}), we do not have such rigidity for $X \in B_1 \setminus \R^d$. Even worse, the coefficients $A$ are not uniformly Lipschitz, and the computation in Remark \ref{rmk:whitney_scaling} tells us that the best behavior we can expect of $u$ in the Whitney ball $B_{\delta(X)/4}(X)$ is that of a solution $v$ to a divergence form equation $-\divv (\tilde{A} \nabla v) = 0$ in the ball $B_2$ with $|\nabla \tilde{A}| \le C$. In particular, this behavior \textit{does not improve} as $\delta(X) \ra 0$, and so $u$ is only well-approximated by a harmonic function in $B_{\tau \delta(X)}(X)$ for $\tau  \in (0,1)$ sufficiently small. Of course the balls, $B_{\tau \delta(X_i)}(X_i)$, can be far from overlapping, and so this harmonic approximation tells us very little about the global structure (i.e., Lipschitz parametrization) of the points $X_i$ in $B_1$. As such, we must abandon this strategy of obtaining Lipschitz parametrizations and find another way to estimate the size of all of the balls centered at the points $X_i$ where the frequency is pinched. Our approach is to show that such balls must be contained in many cones centered at $\R^d$ pointing in roughly the same direction.

With the tools developed thus far, we can prove our main covering Lemma used in the proof of Theorem \ref{thm:main}, but first let us introduce some notation. In what follows, assume $u$ is a solution of \eqref{eqn:soln_var}.

\begin{defn}\label{defn:good_scale}
Fix $\epsilon > 0$. We say that a boundary ball $B_t(X_0)$ for $X_0 \in \R^d$ is $(\epsilon, \Lambda, \Lambda_0)$-good if for each $Y \in B_t(X_0)$ and $ 0 < s \le t$, we have 
\begin{align}\label{eqn:good_ball}
N_u(Y,s) \le \begin{cases}
	\Lambda + \epsilon, & Y \in \R^d, \\
	\Lambda_0, & Y \not \in \R^d.
\end{cases}
\end{align}
\end{defn}
\begin{defn}\label{defn:radii}
Fix $\bar{r} > 0$. Whenever $B_t(X_0)$ is $(\epsilon, \Lambda, \Lambda_0)$-good, define for $X \in B_t(X_0) \cap \R^d$, $r_{X}' \coloneqq \sup \{    s \ge 0 \; : \; N_u(X, s) \le \Lambda - \epsilon  \}$, and $r_X \coloneqq \max \{ \bar{r}, r_X'    \} $ (here, $r_X' = -\infty$ if the set in its definition is empty).
\end{defn}

Our main covering result, in analogy with Proposition 3.36 in \cite{NV17}, is the following.

\begin{lemma}\label{lem:cov}
Let $\Lambda_0, \bar{r} > 0$ and suppose $\Lambda_k \in \mathcal{F}$ is such that $\Lambda_k \le \Lambda_0$ and $k > 1$. Then there exists $C = C(\Lambda_0) > 1$ and a choice of $\epsilon = \epsilon(\Lambda_0), \eta = \eta(\Lambda_0) \in (0,1)$ sufficiently small so that the following holds.

Suppose that $u$ is solution of \eqref{eqn:soln_var}, with \eqref{eqn:ellipse_cond} and \eqref{eqn:unif_bound}, and $B_1(0)$ is $(\epsilon, \Lambda_k, \Lambda_0)$-good. Then there is a finite collection of balls $\{B_{r_i}(X_i)\}_{i \in I}$ with $X_i \in (\sing(u) \cup \R^d) \cap B_1(0)$  so that 
\begin{align*}
( \sing(u) \cup \R^d) \cap B_1(0) \subset \bigcup_{i \in I} B_{r_i}(X_i), \; \; \; \sum_{i \in I} r_i^{n-2} \le C.
\end{align*}
Moreover, for each ball $B_{r_i}(X_i)$, one of the following conditions holds:
\begin{enumerate}[(i)]
\item$C^{-1} \bar{r} \le r_i \le C \bar{r}$, 
\item $B_{2r_i}(X_i) \subset \R^n \setminus \R^d$,
\item $X_i \in \R^d$, and $B_{r_i}(X_i)$ is $(\epsilon, \Lambda_{k-1}, \Lambda_0)$-good.
\end{enumerate}
\end{lemma}
\begin{proof}
	
	{Since the proof is technical, let us provide an overview. It may be useful for the reader to refer back to the parameters introduced in this Lemma, as well as their uses, below:
		\vspace{0.5em}
		\begin{itemize}
			\item $\Lambda_0 > 1$ : a uniform bound on the Almgren frequency $N_u$
			\item $\Lambda_k$, ($\Lambda_{k-1}$) : the current (next) discrete frequency, as in the Lemma statement
			\item $\epsilon \in (0,1) $ : the small ``pinching'' parameter around the frequency, see Definitions \ref{defn:pinched} and \ref{defn:good_scale}
			\item $\eta \in (0,1)$ : how close $\mathcal{A}$ is to the identity matrix, see \eqref{eqn:ellipse_cond}
			\item $\rho_1 \in (0,1)$ : a small scale for which the frequency-dropping property, Lemma \ref{lem:freq_drop}, holds
			\item $\alpha, \gamma \in(0,1)$ : geometric parameters quantifying the containment of $\sing(u)$ in a cone provided that $u$ has $\epsilon$-pinched frequency  (see Corollary \ref{cor:pinch_cone})
			\item $\tau \in (0,1)$, $M, L_1, L_2 > 1$ : fixed parameters that ensure certain geometric containments hold
		\end{itemize}
		\vspace{0.5em}
		The main idea of the proof is as follows. We want to find a cover of $\sing(u) \cap \R^d \cap B_1(0)$ consisting of smaller balls $B_{r_i}(X_i)$ satisfying the conclusions of the Lemma. Since $B_1(0)$ is $(\epsilon, \Lambda_k, \Lambda_0)$-good, then the frequency  satisfies $N_u(X,s) \le \Lambda_k + \epsilon$ for $X \in B_1(0) \cap \R^d$, $0 < s \le 1$, and thus the frequency of $u$ is $\epsilon$-pinched about $\Lambda_k$ at $X$ in the scales $[r_X, 1]$, by the definition of $r_X$. Suppose that $r_X = r_X'$. First, notice that $N_{u}(X, r_X) = \Lambda_k-\epsilon$, and so the frequency-dropping principle, Lemma \ref{lem:freq_drop}, implies (for $\epsilon$ and $\eta$ small) that $N_u(X, s) \le \Lambda_{k-1} + \epsilon$ for $0 < s \le \rho_1 r_X$. This is essentially (iii), our desired ``goodness'' condition for the next set of balls (though at the moment this only shows that bound for $N_u(X,s)$ at $X$, and not near $X$). Moreover, Corollary \ref{cor:pinch_cone} says that as long as $\epsilon$ and $r_X'$ are sufficiently small, then singular set $\sing(u)$ \textit{away} from the boundary $\R^d$ is contained in a conical region about an $(n-2)$-hyperplane. For such singular points, we will choose a certain class of Whitney-type balls to cover them, so that these balls satisfy (ii). Finally, we will use Lemma \ref{lem:ball covering} to obtain a Minkowski estimate on these Whitney-type balls away from $\R^d$, while the balls centered on $\R^d$ sum for free, as long as we can find an essentially disjoint collection. The arguments thus far only accounts for the types of points $X$ for which $r_X = r_X'$, but for the other points, we have that $r_X = \bar{r}$, which will give balls for which (i) holds. The biggest technicality in the proof comes from the fact that for the aforementioned balls for which (iii) holds, we only had that $N_u(X,s) \le \Lambda_{k-1} + \epsilon$ for $0 < s \le \rho_1 r_X$. In order to guarantee such bounds for all $Y \in \R^d$ nearby $X$, we will need to separate boundary points $X \in \R^d \cap B_1(0)$ into good points, where $r_Y \gtrsim r_X$ near $B_{r_X}(X)$, and deal with the rest separately.
		
		\vspace{0.5em}
		
		At this point, let us begin the rigorous proof in earnest. 
	}With $\Lambda_0$ given, we may choose some $\alpha = \alpha(\Lambda_0)  \in (0,1)$, a finite family of $(n-2)$-planes $\{V_1, \dotsc, V_K\} \subset G(n, n-2)$ with $K  = K(\Lambda_0) \in \N$, and $\epsilon = \epsilon(\Lambda_0), \eta = \eta(\Lambda_0)$, and $\gamma = \gamma(\Lambda_0) \in (0,1)$ small enough so that the conclusion of Corollary \ref{cor:pinch_cone} holds. By choosing $\eta$ smaller (now also depending on $\epsilon$), we also may assume that conclusion of Lemma \ref{lem:freq_drop} holds with $\rho_1 = \rho_1(\epsilon, \Lambda_0) \in (0,1)$. Finally, we fix parameters $M, L_1, L_2 > 1$ large which shall be specified later for geometric arguments, but in the end only $M$ and $L_2$ will depend on $\Lambda_0$ (implicitly through $\epsilon$, $\eta$, and $\gamma$ as well).

\vspace{0.5 em}
\textbf{Step one: covering the boundary.} First, we subdivide $ B_1 \cap \R^d$ into a good piece $G$ and a bad piece $F$ for the covering:
\begin{align*}
G & \coloneqq \{ X \in B_1 \cap \R^d  \; : \; r_Y \ge r_X/10 \text{ for all } Y \in \R^d  \cap B_{10 M  r_X}(X) \} \\
F & \coloneqq (B_1 \cap \R^d) \setminus G.
\end{align*}
By Vitali's covering lemma, we may choose a finite collection of points $X_i \in G$, $i \in I_1$ so that with $r_i = r_{X_i}$, we have $B_{Mr_i}(X_i)$ are disjoint and $\cup_{i \in I_1} B_{5M r_i}(X_i) \supset G$. It is clear that since the $B_{Mr_i}(X_i)$ are disjoint and centered on $\R^d$ with $d \le n-2$, then 
\begin{align}\label{eqn:cov_sum1}
\sum_{i \in I_1} (Mr_i)^{n-2} \le C.
\end{align}
Moreover, for each $i \in I_1$, either $\bar{r} \le r_i \le 10 \bar{r}$ or otherwise $r_i > 10 \bar{r}$ and the definition of $G$ implies that $r_{Y} = r_Y' \ge r_i/10$ for $Y \in B_{10 M r_i}(X_i) \cap \R^d$.  Thus in this latter case, we have that for all $Y \in B_{10 M r_i}(X_i) \cap \R^d$, $N_u(Y, r_{Y}') \le \Lambda_k - \epsilon$ by definition, and thus Lemma \ref{lem:freq_drop} implies that
\begin{equation}\label{eqn:drop_g}
N_u(Y, s) \le \Lambda_{k-1} + \epsilon \text{ for all } 0 \le s \le \rho_1 10 M r_i, \; Y \in B_{10 M r_i}(X_i) \cap \R^d.
\end{equation}
Note that the factor of $\rho_1$ in \eqref{eqn:drop_g} means these balls do not yet satisfy the conditions of the Lemma, but we will modify them at the end of the proof to get our cover of $G$.

Next, we move to the covering of the bad part of $\R^d \cap B_1$, $F$. Consider any point $Y \in F $. By the definition of a bad point, we can find some $Q_1 \in \R^d \cap B_{10 M r_{Y}}$ with $r_{Q_1} < r_Y/10$. Then either $Q_1 \in G$ and we stop, or $Q_1 \in F$ and we repeat to find $Q_2 \in B_{10 M r_{Q_1}}(Q_1)$ with $r_{Q_2} < r_{Q_1}/ 10 < 10^{-2} r_Y$. We repeat the process (which terminates, since the $r_X$ are bounded from below) until we find some $Q \in G$ for which 
\begin{align*}
\abs{Y - Q} \le  20 M r_Y, \; \; r_Q < r_Y / 10.
\end{align*}
Since $Q \in G$, there exists some index $i$ (depending on $Y$) so that $Q \in B_{5 M r_i}(X_i)$. By definition of $G$, we see that $r_i/10 \le r_Y$ and thus $\abs{Y - X_{i}} \le 30 M r_Y$. Altogether this shows that 
\begin{align}\label{eqn:ry_bd}
\min_{j \in I_1} \abs{Y - X_j} \le 30 M r_Y \text{ for all } Y \in F.
\end{align}
For $Y \in F$, denote by $i(Y)$ an index minimizing the distance above.

Now for $Y \in F \setminus \cup_{i \in I_1} B_{5 M r_i}(X_i)$, let us set $t_Y \coloneqq    \min_{j \in I_1}\abs{Y - X_j}/L_1 < 30M r_Y/L_1$. For such $Y$ and for $Z \in B_{ 10 t_Y}(Y)$, then choosing $L_1 > 10 $ large enough, we have 
\begin{align*}
\min_{j \in I_1} \abs{Z - X_j} \ge  \abs{Y - X_{j}} - |Z-Y| \ge \dfrac{1}{2} \abs{Y - X_{i(Y)}} = \dfrac{L_1 t_Y}{2}.
\end{align*}
In particular, with \eqref{eqn:ry_bd} applied to $Z \in F \cap B_{10 t_Y}(Y)$ combined with the above, we see that 
\begin{align*}
30 M r_Z \ge \min_{j \in I_1} \abs{Z - X_j} \ge \dfrac{L_1 t_Y}{2},
\end{align*}
i.e., $r_Z \ge (L_1 t_Y) /(60M) \ge t_Y/(60M)$. If $Z \in G \cap B_{10 t_Y}(Y)$, then there is an $i \in I_1$ for which $Z \in B_{5 M r_i}(X_i)$, and so by definition of $G$, we have $r_Z \ge r_i/10$. On the other hand, since $Y \not \in B_{5M r_i}(X_i)$ but $B_{5 M r_i}(X_i)$ meets $B_{10 t_Y}(Y)$, we readily see that $L_1 t_Y \le \abs{Y - X_i} \le 10 t_Y + 5Mr_i$, so for $L_1 > 15 $, we have $t_Y \le M r_i$. Hence we see $r_Z \ge r_i/10 \ge t_Y/(10M)$ in this case, and so in either case we have shown that 
\begin{align}\label{eqn:rz_bd}
r_Z \ge t_Y/(60 M) \text{ for any } Z \in B_{10 t_Y}(Y) \cap \R^d.
\end{align}

Now we repeat essentially the same procedure as for $G$ but with the radii $t_Y$ instead. By Vitali's covering Lemma, we may choose a finite number of points $Y_i \in F \setminus \cup_{j \in I_1} B_{5 M r_j}(X_j)$, $i \in I_2$ and radii $t_i = t_{Y_i}$ so that the $B_{t_i}(Y_i)$ are pairwise disjoint and $\cup_{i \in I_2} B_{5 t_i}(Y_i) \supset F \setminus \cup_{j \in I_1} B_{5 M r_j}(X_j)$. Similar to before, since the $B_{t_i}(Y_i)$ are disjoint and centered on $\R^d$, we have the estimate
\begin{align}\label{eqn:cov_sum2}
\sum_{i \in I_2} (10t_i)^{n-2} \le C.
\end{align}
First recall from above that for any such $Y_i$, $|Y-X_j| \geq 5Mr_j \geq 5M \bar{r}$. Using \eqref{eqn:rz_bd} just as before, we obtain that either for each $i \in I_2$, either $\frac{5M}{L_1} \bar{r} \leq t_i < 60M\bar{r}$, or otherwise for all $Z\in B_{10t_i}(Y_i) \cap \R^d$, we have $r_Z \geq t_i/(60M) \geq \bar{r}$ and so by definition $N_u(Z, r_Z) \leq \Lambda_k - \epsilon$ and, by Lemma \ref{lem:freq_drop},  
\begin{align}\label{eqn:drop_f}
N_u(Z, s) \le \Lambda_{k-1} + \epsilon \text{ for all } 0 \le s \le \rho_1 t_i / (60M), \; \; Z \in B_{10 t_i}(Y_i) \cap \R^d.
\end{align}
Just as for $G$, these balls do not quite satisfy the conditions of the Lemma (due to the $\rho_1$ factor) but we will modify them later to cover $F$, and thus all of $B_1 \cap \R^d$. 

\vspace{0.5 em}
\textbf{Step two: the covering away from the boundary.} We move to our covering for the rest of $\sing(u)$, which we write as
\begin{align*}
\sing'(u) \coloneqq (\sing(u) \cap B_1) \setminus \left(  \bigcup_{i \in I_1} B_{10 M r_i}(X_i) \cup \bigcup_{i \in I_2} B_{10 t_i}(Y_i)  \right).
\end{align*} 
For $Z \in \sing'(u)$ we define similarly to before 
\begin{align*}
s_{Z,1} \coloneqq \min_{i \in I_1} \abs{Z - X_i},  \; s_{Z,2} \coloneqq \min_{i \in I_2}  \abs{Z - Y_i}, \; s_Z \coloneqq \min \{s_{Z,1}, s_{Z,2} \}/L_2.
\end{align*}
First we remark that with this choice of $s_Z$, we have 
\begin{align}\label{eqn:sz_bd}
s_Z \le 2 \delta(Z)/L_2
\end{align}
for all $Z \in \sing'(u)$. Indeed given such a $Z$, we know that $\pi(Z) \in \R^d$ is contained in some $B_{5M r_i}(X_i)$ or $B_{5 t_i}(Y_i)$, so assume it is in the first set (the argument for the second case is exactly the same). Since $Z \not \in B_{10 M r_i}(X_i)$, we have that 
\begin{align*}
10 M r_i < \abs{Z - X_i} \le \abs{ Z - \pi(Z)} + \abs{\pi(Z) - X_i} < \delta(Z) + 5 M r_i,
\end{align*}
and thus $\delta(Z) > 5 M r_i > \abs{\pi(Z) - X_i}$. Finally, the triangle inequality, the previous inequality, and the definition of $s_Z$ give \eqref{eqn:sz_bd}. As before, we choose a finite number of points $Z_i \in \sing'(u)$, $i \in I_3$ and radii $s_i = s_{Z_i}$ so that the $B_{s_i}(Z_i)$ are disjoint and $\cup_{i \in I_3} B_{5 s_i}(Z_i) \supset \sing'(u)$. In view of \eqref{eqn:sz_bd} for $L_2$ large, we see that the balls $B_{5 s_i}(Z_i)$ are relatively far from the boundary $\R^d$ (i.e. satisfy the second condition in the Lemma for balls in our cover), so we need only to estimate $\sum_{i \in I_3} s_i^{n-2}$, which we do now.

First, we claim that for every $Z \in \sing'(u)$, we have the estimate
\begin{align}\label{eqn:sz_bd2}
\min_{i \in I_1} \abs{Z - X_i} \le 2 L_1 \min_{i \in I_2} \abs{Z - Y_i},
\end{align}
Indeed, suppose otherwise, and choose $j \in I_2$ so that $\abs{Z - Y_j} = s_{Z,2}$ and $i \in I_1$ so that $L_1 t_{Y_j} =  \abs{Y_j - X_i}$. We immediately obtain
\begin{align*}
2 L_1 \abs{Z - Y_j} < \abs{Z - X_i} \le \abs{Z - Y_j} + \abs{Y_j - X_i} \le \abs{Z - Y_j} + L_1 t_j,
\end{align*}
which says that $\abs{Z- Y_j} \le 2t_j$, a clear contradiction to the fact that $Z \in \sing'(u)$. Notice that \eqref{eqn:sz_bd2} and the definition of $s_i$ implies for every $i \in I_3$,
\begin{align}\label{eqn:sz_bd3}
s_i \ge \min_{j \in I_1} \abs{Z-X_j}/(2 L_1 L_2) = s_{Z_i, 1} /(2 L_1 L_2).
\end{align}
Define for $i \in I_3$, $j(i) \in I_1$ an index for which $\abs{Z_i - X_{j(i)}} \le 2L_1 L_2 s_i$, which exists by \eqref{eqn:sz_bd3}. Now we separate the index set $I_3$ into $K+1$ families as follows. 

First set $I_1^\tau = \{ j \in I_1 \; : \; r_j \ge \tau\}$, $I_3^\tau  = \{ i \in I_3 \; \; : \; j(i) \in I_1^\tau\}$ for some fixed $0<\tau < \gamma^2/2$, where $\gamma \in (0,1)$ was chosen at the beginning of the proof. By definition of $(\epsilon, \Lambda_k, \Lambda_0$)-good and the radii $r_X'$, we know that $u$ is $\epsilon$-pinched about $\Lambda_k$ in the scales $[r_X', 1]$ when $X \in \R^d \cap B_1$. Hence by Corollary \ref{cor:pinch_cone}, we know that for each $j \in I_1 \setminus I_1^\tau$, there is an index $1 \le \ell(j) \le K$ so that 
\begin{align}\label{eqn:cone_cont}
\sing(u) \cap ( B_{\gamma}(X_j) \setminus B_{\gamma^{-1} r_{j}}(X_j)) \subset X(X_{j}, V_{\ell(j)}, \alpha),
\end{align}
where the $\{V_\ell\}$ are the planes chosen in the beginning of the proof. For convenience, write $\ell(i) = \ell(j(i))$, and introduce the partitions
\begin{align*}
I_3^\ell \coloneqq \{ i \in I_3 \setminus I_3^\tau \; : \; \ell(i) = \ell\},  \; \;I_1^\ell \coloneqq \{ j \in I_1 \; : \; \ell(j) = \ell\}.
\end{align*}
Now we provide an estimate for the sums $\sum_{i \in I_3^\ell} s_i^{n-2}$ and $\sum_{i \in I_3^\tau} s_i^{n-2}$ separately. The second sum is easy to estimate though, once we recall \eqref{eqn:sz_bd3}. Indeed we have that since $Z_i \not \in B_{10 M r_{j(i)}}$, 
\begin{align*}
s_i \ge \abs{Z_i - X_{j(i)}}/(2 L_1 L_2) \ge (5 M \tau)/(L_1 L_2),
\end{align*}
which is a constant. Hence using the disjointedness of the $B_{s_i}(Z_i)$ (and the fact that they are contained in $B_2$),
\begin{align}\label{eqn:cov_sum3}
\sum_{i \in I_3^\tau} s^{n-2} \le C \sum_{i \in I_3^\tau} s^n \le C.
\end{align}

Moving to the other sums, we claim that as long as $M >\gamma^{-1}$, then for any pair $i \in I_3$ and $j \in I_1 \setminus I_1^\tau$ for which $Z_i \in B_{\gamma}(X_j)$, \eqref{eqn:cone_cont} implies $Z_i \in X(X_{j}, V_{\ell(j)}, \alpha)$. Indeed \eqref{eqn:cone_cont} applies to $Z_i$ provided that the annulus there is non-degenerate, and $Z_i$ is contained in it. This constitutes two conditions: $\gamma^{-1} r_j < \gamma$, which is guaranteed by the choice of $\tau$ and the fact that $j \in I_1 \setminus I_1^\tau$, and $\abs{Z_i - X_j} \ge \gamma^{-1} r_j$, which is true provided we choose $M > \gamma^{-1}$ (recall that $Z_i \in \sing'(u)$, so $Z_i \not \in B_{10M r_j}(X_j)$). We have shown the claim, but in fact even more is true: if $L_2$ is chosen small enough, depending on $\alpha$, then 
\begin{align}\label{eqn:cone_cont2}
B_{\abs{Z_i - X_j}/L_2}(Z_i) \subset X(X_{j}, V_{\ell(j)}, (1 + \alpha)/2).
\end{align}
Indeed this is a simple geometric fact, which at the unit scale reads as follows: if $Z \in \partial B_1 \cap X(V, \alpha)$, where $V \in G(n, n-2)$ and $L = L(\alpha) >1$ is sufficiently large, then $B_{L^{-1}}(Z) \subset X(V, (1+\alpha)/2)$, independently of the plane $V$. We omit this computation. 

Finally, we are in the position to conclude. If we set for $\ell$ fixed and $i \in I_3^\ell$, $\tilde{s}_{i, \ell} \coloneqq \min_{j \in I_1^\ell} \abs{Z_i - X_j}/L_2 \ge s_i$, then with the containment \eqref{eqn:cone_cont2} we are able to apply Lemma \ref{lem:ball covering} to the family of balls $\{B_{ \tilde{s}_{i, \ell} }(Z_i)\}_{i \in I_3^\ell}$, the boundary points $\{X_j\}_{j \in I_3^\ell}$ and the fixed plane $V_\ell$ to conclude 
\begin{align*}
\sum_{i \in I_3^\ell} s_i^{n-2} \le C \sum_{i \in I_3^\ell} \tilde{s}_{i,\ell}^{n-2} \le C (\alpha).
\end{align*}
To be clear, we only have the containment \eqref{eqn:cone_cont2} for when $Z_i \in B_{\gamma}(X_j)$, and so we must apply Lemma \ref{lem:ball covering} several times (with balls centered on the boundary of radius $\gamma/2$, say), and use the fact that if $Z_i \not \in B_{\gamma}(X_j)$ for each $j \in I_1^{\ell}$, then $\tilde{s}_{i, \ell} \gtrsim_\gamma 1$. Altogether with \eqref{eqn:cov_sum3} then,
\begin{align}\label{eqn:cov_sum4}
\sum_{i \in I_3} s_i^{n-2} \le \sum_{i \in I_3^\tau} s_i^{n-2} + \sum_{\ell = 1}^K \sum_{i \in I_3^\ell} s_i^{n-2} \le C +C(\alpha) K,
\end{align}
is a constant depending only on $\Lambda_0$.

\vspace{0.5 em}
\textbf{Step three: refining the cover.} Finally, we can produce our desired cover from our two main steps thus far. For the piece of the cover coming from the balls $\{B_{10M r_i}(X_i)\}_{i \in I_1}$, we choose for each $i \in I_i$ boundedly many points $X_{i, 1}, \dotsc, X_{i, N(i)} \in \R^d \cap B_{10 M r_i}(X_i)$, $X_{i, 1}', \dotsc, X_{i, N(i)}' \in B_{10 M r_i}(X_i)$ with $N(i) \le C = C(M, \rho_1)$ so that 
\begin{align*}
\bigcup_{j=1}^{N(i)} B_{\rho_1 10 M r_i}(X_{i, j}) \supset B_{10 M r_i}(X_i) \cap \{Y\mid \delta(Y) \le \rho_1 5 M r_i \} , \\
\bigcup_{j=1}^{N(i)} B_{\rho_1 r_i}(X_{i,j}') \supset B_{10 M r_i}(X_i) \cap \{Y\mid \delta(Y) > \rho_1 5 M r_i\}, 
\end{align*}
and for each $1 \le j \le N(i)$ we have that $B_{2\rho_1  r_i}(X_{i,j}') \subset \R^n \setminus \R^d$. Notice that for each $i \in I_1$, either $\bar{r} \le r_i \le 10 \bar{r}$, or otherwise \eqref{eqn:drop_g} implies that $B_{\rho_1 10 M r_i}(X_{i, j})$ are $(\epsilon, \Lambda_{k-1}, \Lambda_0)$-good for $1 \le j \le N(i)$. We perform essentially the same construction for the balls $\{B_{10t_i}(Y_i)\}_{i \in I_2}$ with \eqref{eqn:drop_f} to find boundedly many balls $B_{\rho_1 t_{i}/(60M)}(Y_{i,j}), B_{\rho_1 t_i/(60M^2)}(Y_{i,j}')$ with $1 \le j \le N(i) \le C(M, \rho_1)$, $Y_{i,j} \in \R^d \cap B_{10 t_i}(Y_i)$, and so that 
\begin{align*}
\bigcup_{j=1}^{N(i)} \left( B_{\rho_1 t_i/(60M)}(Y_{i,j}) \cup B_{\rho_1 t_i/(60M^2)}(Y_{i,j}') \right) \supset B_{10 t_i}(Y_i),\\
B_{2\rho_1 t_i /(60M^2)}(Y_{i,j}') \subset \R^n \setminus \R^d.
\end{align*}
Moreover, for each $i \in I_2$, either $\frac{5M}{L_1}\bar{r} \le t_i \le 60 M \bar{r}$ or otherwise $B_{\rho_1 t_i/(60M)}(Y_{i,j})$ is $(\epsilon, \Lambda_{k-1}, \Lambda_0)$-good. Finally, for each of the balls $B_{5 s_i}(Z_i)$, $i \in I_3$ we recall \eqref{eqn:sz_bd} which tells us that for $L_2$ sufficiently large that $ B_{10 s_i}(Z_i) \subset \R^n \setminus \R^d$. In view of the estimates \eqref{eqn:cov_sum1}, \eqref{eqn:cov_sum2}, and \eqref{eqn:cov_sum4} and the fact that $N(i) \le C$, our Lemma is complete if we cover by the collection of balls
\begin{align*}
\{ B_{5s_i} (Z_i) \}_{i \in I_3}, \; \;  \{ B_{\rho_1 10 M r_i} (X_{i,j}), B_{\rho_1 r_i} (X_{i,j}') \}_{i \in I_1,  \, 1 \le j \le N(i)}, \\ 
\{ B_{\rho_1 t_i/(60M)} (Y_{i,j}), B_{ \rho_1 t_i/(60M^2)} (Y_{i,j}') \}_{i \in I_2,  \,  1 \le j \le N(i)}.
\end{align*}
\end{proof}

With the help of Lemma \ref{lem:cov}, we may now provide a proof of our main result.
\begin{thm}\label{thm:main_flat}
Suppose that $A$ satisfies the higher co-dimension $C^{0,1}$ condition in $B_{10}$ with constants $\lambda, C_0 > 1$, and that $w \in W_r(B_{10}) \cap C(B_{10})$ is a solution of \eqref{eqn:soln_var}. Then there exists $r_0 = r_0(\lambda, C_0) \in (0, 10)$ small enough and $C = C(\lambda, C_0, N_w(0, 100r_0)) > 1$ so that for $0 < s < r_0$, we have 
\begin{align*}
\L^n \left( B_{s}(\sing(w)) \cap B_{r_0}  \right) \le C s^2.
\end{align*}
In particular,
\begin{align*}
\HD^{n-2}(\sing(w) \cap B_{r_0}) \le C(\lambda, C_0, N_w(0, 100r_0)).
\end{align*}
Moreover, $\sing(w) \cap B_{r_0}$ is countably $(n-2)$-rectifiable: there exists countably many $(n-2)$-dimensional Lipschitz graphs $\Sigma_i \subset \R^n$, $i \in \N$ so that
\begin{align*}
	\HD^{n-2} \left(  \sing(w) \cap B_{r_0} \setminus \bigcup_{i \in \N} \Sigma_i  \right) = 0.
\end{align*}
\end{thm}
\begin{proof}
Let $w$ be such a solution, and define the function $v(X) \coloneqq w(\A_0^{1/2}X)$ where $\A_0 = \A(0)$. By Remark \ref{rmk:change_vars_flat}, we have that $v$ is a solution of 
\begin{align*}
-\divv(B \nabla v) &  = 0 \text{ in } \A_0^{-1/2} B_{10}, \\
v & = 0 \text{ on } \R^d \cap \A_0^{-1/2} B_{10}.
\end{align*}
where $B$ satisfies the higher co-dimension $C^{0,1}$ condition in $B_{10/\sqrt{\lambda}}$ with constants $C \lambda, C C_0$ and $\mathcal{B}(0) = I$. Moreover, since $\A_0^{1/2}$ is bi-Lipschitz and the frequency function is constructed precisely so that $N_w^A(0, r) = N_v^B(0, r)$,
then it suffices to prove the result for $v$ in place of $w$. Of course here we are also using the fact that $\sing(v) = \A_0^{-1/2} \sing(w)$.

To this end, notice that for $v_{r_0}(X) = v(r_0 X)$ with $r_0 > 0$ sufficiently small, $v_{r_0}$ solves an equation of the type \eqref{eqn:soln_var} with matrix $B_{r_0} = \abs{t}^{-n+d+1} \mathcal{B}(r_0 \, \cdot)$, which satisfies the small-constant condition \eqref{eqn:ellipse_cond} with $\eta = r_0 C C_0$. Along with Corollary \ref{cor:uniform_var_off}, we may thus assume by taking $r_0$ sufficiently small that $u \equiv v_{r_0}$ is a solution of \eqref{eqn:soln_var} satisfying \eqref{eqn:ellipse_cond} for $\eta$ as small as we wish, with the additional bound \eqref{eqn:unif_bound}. Here $\Lambda_0 > 1$ is a constant which depends on $\lambda, C_0$, and also $N_w^A(100 r_0) = N_v^B( 100 r_0) = N_u^B(100)$.

Now we begin the covering argument. Fix $0 < s < r_0$, and choose $\epsilon(\Lambda_0) >0$ and $\eta'(\Lambda_0) > 0$ small enough so that we may apply Lemma \ref{lem:cov}. Next, choose a finite number of points 
\begin{align*}
Y_1, \dotsc, Y_N \in \R^d \cap B_1
\end{align*}
with $N = N(\eta')$ so that 
\begin{align*}
\dist(Y, \R^d) \ge \eta'/2, \text{ for }Y \in B_1 \setminus \cup_{i=1}^N B_{\eta'}(Y_i),
\end{align*}
and let us first cover $\sing(u) \cap B_{\eta'}(Y_i)$ for $1 \le i \le N$ fixed. 

First, notice that if $\Lambda_k \in \mathcal{F}$ is smallest frequency with $\Lambda_k > \Lambda_0$, then $B_{\eta'}(Y_i)$ is $(\epsilon, \Lambda_k, \Lambda_0)$-good, and so we may apply Lemma \ref{lem:cov} to find a finite collection of balls $\{B_{r_j}(X_j)\}_{j \in I_k}$ for which
\begin{align*}
(\R^d \cup \sing(u)) \cap  B_{\eta'}(Y_i) \subset \bigcup_{j \in I_k} B_{r_j}(X_j), \; \sum_{j \in I_k} r_j^{n-2} \le C (\eta')^{n-2}.
\end{align*}
Moreover, we also have for any index $j \in I_k$, either $C^{-1}s \le r_j \le C s$ or $2B_{r_j}(X_j) \subset \R^n \setminus \R^d$, (in which case we write $j \in I_k^G$) or $X_j \in \R^d$ and $B_{r_j}(X_j)$ is $(\epsilon, \Lambda_{k-1}, \Lambda_0)$-good (in which case we write $j \in I_k^B$). For each index $j \in I_k^B$ we iterate; applying Lemma \ref{lem:cov} again, we can find a collection of balls $\{B_{r_\ell}(X_\ell)\}$ indexed by $\ell \in I^j_{k-1}$ for which
\begin{align*}
(\R^d \cup \sing(u)) \cap B_{r_j}(X_j)  \subset \bigcup_{\ell \in I_{k-1}^j} B_{r_\ell}(X_\ell), \; \sum_{\ell \in I_{k-1}^j} r_\ell^{n-2} \le C (r_j)^{n-2}.
\end{align*}
In addition, the same alternative holds for $B_{r_\ell}(X_\ell)$; for each such ball, either $C^{-1} s \le r_\ell \le C s$ or $2 B_{r_\ell}(X_\ell) \subset \R^n \setminus \R^d$, or $B_{r_\ell}(X_\ell)$ is $(\epsilon, \Lambda_{k-2}, \Lambda_0)$-good. Notice that combining the two previous sums, we obtain 
\begin{align*}
\sum_{j \in I_k} \sum_{\ell \in I_{k-1}^j} r_{\ell}^{n-2} \le C^2 (\eta')^{n-2}.
\end{align*}

Repeating this finitely many times, this process must terminate until we find a finite union of balls $\{ B_{r_j}(X_j)\}_{j \in I}$ (relabeled for convenience) for which
\begin{align}\label{eqn:main_cov1}
(\R^d \cup \sing(u)) \cap B_{\eta'}(Y_i) \subset \bigcup_{j \in I} B_{r_j} (X_j), \; \, \; \sum_{j \in I} r_j^{n-2} \le C^k(\eta')^{n-2},
\end{align} 
and for each $j \in I$, one of the alternatives holds: either $C^{-1} s \le r_j \le C s$, $B_{2r_j}(X_j) \subset \R^n \setminus \R^d$, or $B_{r_j}(X_j)$ is $(\epsilon, \Lambda_1, \Lambda_0)$-good. If $B_{r_j}(X_j)$ is a ball for which the last alternative holds, notice we can apply Lemma \ref{lem:sing_empty} to further decompose each such ball $B_{r_j}(X_j)$ into a finite family of sub-balls of comparable radii which cover $\sing(u) \cap B_{r_j}(X_j)$ and instead, only one of the first two alternatives hold. Finally, for each $j \in I$ for which $ B_{2r_j}(X_j) \subset \R^n \setminus \R^d$, we may apply the Theorem \ref{thm:NV} of \cite{NV17} (using the fact that the frequency of $u$ in $B_{r_j}(X_j)$ is uniformly bounded by $\Lambda_0$, and that $u$ satisfies a uniformly elliptic equation with Lipschitz coefficients in $B_{r_i}(X_j)$ by Remark \ref{rmk:whitney_scaling}) to obtain a further sub-cover of finitely many balls $\{B_{t_\ell}(Z_\ell) \}$ for which $C^{-1} s \le t_\ell \le C s$ and 
\begin{align*}
\sing(u) \cap B_{r_j}(X_j) \subset \bigcup_{\ell} B_{t_\ell}(Z_\ell), \; \, \; \sum_{\ell} t_\ell^{n-2} \le C r_j^{n-2}. 
\end{align*}

Replacing the $B_{r_j}(X_j)$ by this sub-cover and again relabeling for convenience, we have found a collection of balls $B_{r_i}(X_i)$ for which
\begin{align*}
\sing(u) \cap B_{\eta'}(Y_i) \subset \bigcup_{j \in I} B_{r_j} (X_j), \; \, \; \sum_{j \in I} r_j^{n-2} \le C^k(\eta')^{n-2}, \; \; C^{-1} s \le r_j \le C s.
\end{align*} 
Summing over the finitely many points $Y_1, \dotsc, Y_N$, and taking a cover of $\sing(u) \cap \{\delta(Y) > \eta'/2\}$ using Whitney balls as we did previously with the main result of \cite{NV17}, we thus can cover
\begin{align}\label{eqn:main_cov2}
\sing(u) \cap B_1 \subset \bigcup_{j \in J} B_{r_j} (X_j),\; \, \; \sum_{j \in J} r_j^{n-2} \le C, \; \; C^{-1} s \le r_j \le C s,
\end{align} 
which one readily sees implies the first main conclusion of the Theorem once recalling the definition of $u$.

To see that $\sing( u )$ is countably $(n-2)$ rectifiable in $B_{1}$, we simply recall Remark \ref{rmk:whitney_scaling} which implies that $u$ solves a uniformly elliptic equation with Lipschitz coefficients in Whitney balls $B_{\delta(Y)/4}(Y)$ away from $\R^d$. In particular, the main Theorem in \cite{NV17} implies that $\sing(u) \cap B_{\delta(Y)/4}(Y)$ is $(n-2)$-rectifiable, whenever $Y \in B_1 \setminus \R^d$. We then can cover $B_1\setminus \R^d$ by countably many Whitney balls, and cover the remaining piece $\sing(u) \cap \R^d \subset \R^{n-2}$ to conclude, since $d \le n-2$.
\end{proof}

As a Corollary of the main result, Theorem \ref{thm:main_flat}, we obtain singular set estimates for solutions to the degenerate elliptic equation outside of $C^{1,1}$ $d$-dimensional graphs. In what follows, as in Appendix \ref{sec:app_gr}, when $\Gamma$ is $d$-Ahlfors regular\footnote{any graph of a $d$-Lipschitz function will be $d$-Ahlfors regular}, we write 
\begin{equation}\label{e:regdist}
D_{\beta}(X) \coloneqq \left(  \int_\Gamma \frac{1}{\abs{X-Y}^{d+\beta}} \;  d\HD^d(Y) \right)^{-1/\beta}, \qquad X \not \in \Gamma
\end{equation}
for the regularized distance, and $D_\infty(X) \coloneqq \dist(X, \Gamma)$ for the euclidean distance to $\Gamma$.
\begin{cor}[The main result]\label{cor:main}
Let $\beta \in (1, \infty]$ and assume the following:
\begin{enumerate}[(i)]
\item $0 \in \Gamma \subset \R^n$ is a $d$-dimensional, $d$-Ahlfors regular $C^{1,1}$ graph parametrized by $\phi \in C^{1,1}(\R^d ; \R^{n-d})$ with $\norm{  \nabla \phi }_{\mathrm{Lip}(\R^d)} \le \cphii$,
\item $u \in W_r(B_1) \cap C(B_1)$ is a weak solution of $$-\mathrm{div}(a(X) D_\beta(X)^{-n+d+1}\nabla u(X))= 0, \qquad  X\in B_1$$ and $u|_{\Gamma \cap B_1} =0$ for $a(X)\in C^{0,1}(B_1)$ and $\Lambda \geq a(x) \geq \lambda$
\end{enumerate}
Then there exists $r_0  \in (0,1)$ small enough and $C >1$ (both depending on $n, \cphii, \beta, \|a\|_{\mathrm{Lip}}, \lambda, \Lambda$ and $C$ depending also on $N_u(0, 100r_0)$) so that 
\begin{align*}
\L^n(B_s(\sing(u)) \cap B_{r_0}) \le C s^2, \qquad \HD^{n-2}(\sing(u) \cap B_{r_0}) \le C.
\end{align*}
Moreover, $\sing(u) \cap B_{r_0}$ is countably $(n-2)$-rectifiable: there exists countably many $(n-2)$-dimensional Lipschitz graphs $\Sigma_i \subset \R^n$, $i \in \N$ so that 
\begin{align*}
	\HD^{n-2} \left(  \sing(u) \cap B_{r_0}  \setminus \bigcup_{i \in \N} \Sigma_i \right) = 0.
\end{align*}
\end{cor}

\begin{proof}
This Corollary is a simple consequence of Theorems \ref{thm:main_flat} and \ref{thm:change_var}. (For details regarding the case $\beta = \infty$, see Remark \ref{rmk:nonsmooth_op}). First, Theorem \ref{thm:change_var} implies that there is $r_0 = r_0(n,d, C_2) \in (0,1)$ and a bi-Lipschitz change of variables (with constant at most $2$) 
\begin{align*}
\rho: B_{r_0} \subset \R^n \ra \rho(B_{r_0}) 
\end{align*}
which maps $\R^d \cap B_{r_0}(0)$ bijectively onto $\Gamma \cap \rho(B_{r_0})$. By \cite[Section 4]{DFMDAHL}, we know that $$-\mathrm{div}(a(X) D_\beta(X)^{-n+d+1}\nabla u(X))= 0$$ in $B_{r_0}$ implies that $w \coloneqq u \circ \rho$ is a solution of 
\begin{align*}
- \divv( a\circ \rho A_{\rho, \beta} \nabla w) & = 0 \text{ in } \rho(B_{r_0}), \\
w &= 0 \text{ on } \R^d \cap \rho(B_{r_0}),
\end{align*}
where $A_{\rho, \beta}$ is symmetric, degenerate elliptic matrix defined by $\rho$ in \eqref{eqn:A_rho}. Moreover, Theorem \ref{thm:change_var} also guarantees that $A_{\rho, \beta}$ satisfies the higher co-dimensional condition in $B_{r_0}$ with constants $\lambda = \lambda(n,d, \beta)$ and $C_0 = C_0(n,d,\beta, C_2) >0$. Furthermore, since $a$ is Lipschitz and $a\simeq 1$ and $\rho$ is Bi-Lipschitz, it follows immediately that $a\circ \rho A_{\rho, \beta}$ also satisfies the higher co-dimension condition with perhaps slightly larger constants that depend on the properties of $a$. Applying Theorem \ref{thm:main_flat} to $w$ gives us the estimate
\begin{align*}
\L^n(B_s(\sing(w)) \cap B_{r_0^2}) \le C' s^2
\end{align*}
where $C' = C'(n,d,C_2, N_w(0, 100r_0^2))$. Since $\rho$ is bi-Lipschitz with constant $\le 2$, we have that $\sing(u) = \rho^{-1} ( \sing(w))$ and so the above implies 
\begin{align*}
\L^n(B_s(\sing(u)) \cap B_{r_0^2/2}) \le C C' s^2,
\end{align*}
which is almost our desired Minkowski estimate. Moreover, since $\rho$ is bi-Lipschitz, then the $(n-2)$-rectifiability of $\sing(w)$ implies the same for $\sing(u)$.

The only thing left to do is replace the constant $C'$ which depends on $N_w(0, 100 r_0^2)$ with the frequency for $N_u$. However, recalling the doubling inequalities from Corollary \ref{cor:doubling_var}, we can estimate the frequency $N_w(0, r)$ by a doubling index, which transfers easily to $u$ through the change of variables $\rho$: 
\begin{align}\label{eqn:N_doubling}
N_w(0, r) \le C \log \left(\dfrac{ \int_{E_{2r}}  w^2 \; dm  }{ \int_{E_r} w^2 \; dm }\right) + 1, \,  r < r_0^2.
\end{align}
It is easy to check that the change of variables $\rho$, which is centered about the point $0 \in \Gamma$, asymptotically maps round balls to round balls near the origin, and so in particular, $\mathcal{A}_{\rho, \beta}(0) = a\circ \rho(0) I$ and thus $E_r = B_{cr}$ for some $c > 0$. Combining this with \eqref{eqn:N_doubling} and the fact that $\rho$ is bi-Lipschitz, it is easy to check then that for $0 < r < r_0^2$ and $r_0$ small enough,
\begin{align*}
N_w(0,r) \le C\log \left(\dfrac{ \int_{B_{4cr}}  u^2 \; dm  }{ \int_{B_{cr/2}} u^2 \; dm }\right) + 1 = C \mathscr{N}_u(cr/2) + 1,
\end{align*}
and so the claim is proved.

\end{proof}

\begin{rmk}[The case $d < n-2$: a simpler proof]
There is a much simpler proof of the $\HD^{n-2}$-measure bound from Theorem \ref{thm:main} when $d < n-2$, which simply uses the subcritical scaling of all of the quantities. The main point is that one may apply Naber-Valtorta in each Whitney ball $B_{\delta(Y)/4}(Y)$ for $Y \not \in \R^d$ (as per Remark \ref{rmk:whitney_scaling}) to deduce that 
	\begin{align*}
		\HD^{n-2}\left( \sing(u) \cap B_{\delta(Y)/4}(Y)  \right) \le C({N}_u(0, 100r_0)) \delta(Y)^{n-2},
	\end{align*}
	and then we sum over a countable collection of such balls covering $B_{r_0} \setminus \Gamma$. This sum converges when $d < n-2$, since in the ball there are  $\approx 2^{dk}$ such Whitney balls of radius $2^{-k}$ each of which contains a singular set with $\mathcal H^{n-2}$-measure no larger than $C2^{-k(n-2)}$.  Summing over all such Whitney balls, and all such scales gives us the estimate
	\begin{align*}
		\HD^{n-2} \left( \sing(u) \cap B_{r_0}  \right) \le C(N_u(0, 100 r_0)) \sum_{k \ge 1} 2^{dk -k(n-2)} \le C(N_u(0, 100 r_0)),
	\end{align*}
	when $d < n-2$. Of course this argument still requires one to prove uniform boundedness of the frequency function outside of some ball (i.e., Corollary \ref{cor:uniform_var_off}).

 Given the above, $d = n-2$ is really the critical case of the conclusion of the main Theorem. In some sense this suggests that the exact statement as in Theorem \ref{thm:main} is maybe the wrong one to consider when $d < n-2$, but the discussion at the end of Section \ref{sec:examples} also suggests that it is not so trivial to find an alternative statement. 
\end{rmk}

\appendix

\section{Regularity theory for elliptic equations with boundaries of high co-dimension}\label{sec:app_reg}
In this section, we recall (and prove) several PDE estimates for solutions to equations of the form \eqref{eqn:soln_var} for matrices $A$ with enough regularity. As usual, we consider the domain $\Omega \coloneqq \R^n \setminus \R^d$ and set $\Gamma \coloneqq \partial \Omega = \R^d $.  The next result comes from \cite{DFM23}, and although the results stated there are global, the arguments and conclusion can be made local. The main premise is that for matrices $A$ which satisfy the higher co-dimensional $C^{0,1}$ condition, one has additional regularity of $\nabla_x u$ for solutions $u$ to \eqref{eqn:soln_var}. In fact, one really only needs to know boundedness of $\nabla_x \A$, where $\A$ is the matrix coming from Definition \ref{defn:c1a}.

\begin{lemma}[Propositions 7.3, 7.5 in \cite{DFM23}]\label{prop:c1a_dxu}
Suppose that $X_0 \in \Gamma$, $R >0$, and $u \in W_r(B_R(X_0))$ solves \eqref{eqn:soln_var} in $B_R(X_0)$, and $u$ vanishes continuously on $\Gamma \cap B_R(X_0)$. If the matrix $A$ satisfies the higher co-dimensional $C^{0,1}$ condition in $B_R(X_0)$ with constants $C_0, \lambda >0$, then we have $\nabla_x u \in W_r(B_{R}(X_0))$ (i.e., $\nabla \nabla_x u \in L^2_\loc(B_R(X_0), dm)$) with the estimate
\begin{align*}
\int_{B_{R/2}(X_0)} \abs{\nabla \nabla_x u}^2 dm \lesssim_{\lambda, C_0} R^{-2} \int_{B_R(X_0)} \abs{\nabla u}^2 \; dm .
\end{align*}
Moreover, $\mathrm{Tr}(\nabla_x u) = 0$ almost everywhere on $\Gamma \cap B_R(X_0)$, and 
\begin{align}\label{eqn:u_x_tr}
\lim_{\epsilon \da 0} \int_{B_{R/2}(X_0) \cap \{ \epsilon/2 < \delta < \epsilon\} } \abs{\nabla_x u}^2 \delta^{-n+d} \; dX = 0.
\end{align}
\end{lemma}

As a consequence of the above, we conclude with some ``soft'' estimates on higher order derivatives of $u$ near $\Gamma$ for solutions of \eqref{eqn:soln}.

\begin{prop}\label{prop:higher_order_moser}
Suppose that $X_0 \in \Gamma$, $R >0$, $u \in W_r(B_R(X_0))$ solves \eqref{eqn:soln} in $B_R(X_0)$, and $u$ vanishes continuously on $\Gamma \cap B_R(X_0)$. For $r < R/2$, $k\in \N$, and $\alpha, \beta$ multi-indices with $k = \abs{\beta}$, define the function
\begin{align*}
v_{\alpha, \beta}(X) \coloneqq \delta(X)^{k} \sup_{Y \in W_X} \abs{\partial_x^\alpha \partial_t^\beta u(Y)}
\end{align*}
where $W_X = B(X, \delta(X)/8)$ is a Whitney ball. Then one has the estimate
\begin{align*}
\int_{B_r(X_0)} v_{\alpha, \beta}(X)^2 \; dm \lesssim_k  R^{-2 \abs{\alpha} + 2} \int_{B_R(X_0)}  \abs{\nabla u}^2 \; dm.
\end{align*}
\end{prop}
\begin{proof}
The argument is a simple consequence of Schauder theory, and the ``classic elliptic theory'' that extends to the degenerate elliptic operators we consider.
Fix $X \in B_r(X_0)$, set $\rho = \delta(X)/4$, and re-scale $u$ by
\begin{align*}
w(Y) \coloneqq (\partial_x^\alpha u)(X + \rho Y).
\end{align*}
We remark that since $u$ is a solution to the equation \eqref{eqn:soln}, then so is $\partial_x^\alpha u $ since $\partial_x^\alpha (\abs{t}) = 0$. 
Moreover, Proposition \ref{prop:c1a_dxu} iterated $\abs{\alpha} -1$ times gives that $\partial_x^\alpha u \in W_r(B_R(X_0))$ with 
\begin{align}\label{eqn:dux}
\int_{B_r(X_0)} \abs{ \partial_x^\alpha u}^2 \; dm \lesssim R^{-2\abs{\alpha}+2} \int_{B_R(X_0)} \abs{\nabla u}^2 \; dm.
\end{align}

Now, setting $\tilde{\delta}(Y) \coloneqq \rho^{-1} \delta(X + \rho Y) $, we see that $w$ solves the equation
\begin{align*}
- \divv( \tilde{\delta}^{-n+d+1} \; \nabla w) = 0 \text{ in } B_1.
\end{align*}
The classical Schauder Theory (see for example, \cite[Corollary 2.29]{FRRO22} implies that for any $\gamma \in (0,1)$,
\begin{align}\label{eqn:hom1}
\norm{w}_{C^{k, \gamma}(B_{1/2})} \lesssim_{n,d,k} \norm{w}_{L^\infty(B_1)},
\end{align}
with constant depending only on $n, d,$ and $k$ since $\tilde{\delta}$ is smooth in $B_1$, bounded above and below in $B_1$ (with controlled constant), and since $\|\tilde{\delta} \|_{C^m(B_1)} \le C_m$ for all $m \in \N$. The inequality \eqref{eqn:hom1} clearly gives that 
\begin{align*}
v_{\alpha, \beta}(X) & \lesssim \sup_{B_{\delta(X)/4}(X)} \abs{ \partial_x^\alpha u},
\end{align*}
so that since $(\partial_x^\alpha u)^2 \ge 0$ has $-\divv(\delta^{-n+d+1} \nabla((\partial_x^\alpha u)^2)) = - 2 \delta^{-n+d+1} \abs{\nabla \partial_x^\alpha u}^2 \le 0$, the interior Moser estimate (\cite[Lemma 8.7]{DFM21AMS} ) gives that 
\begin{align*}
v_{\alpha, \beta}(X)^2 \lesssim \fint_{B_{\delta(X)/2}(X)} (\partial_x^\alpha u)^2 \; dm.
\end{align*}
Since on $B_{\delta(X)/2}(X)$, $\delta(Y) \simeq \delta(X)$, the conclusion of the Proposition holds after a suitable covering argument, by estimate \eqref{eqn:dux}.
\end{proof}

It turns out that in the proof of the monotonicity of the frequency function for solutions in Lemma \ref{lem:monotone_var}, we need $L^\infty$ estimates on $\nabla u$. In the setting of the matrices $A$ which satisfy the higher co-dimensional $C^{0,1}$ condition, we can indeed provide such an estimate. 
\begin{lemma}[Boundedness of $\nabla u$]\label{lem:w2_u}
Suppose that $u \in W_r(B_R(X_0))$ solves \eqref{eqn:soln_var} for $X_0 \in \Gamma = \R^d \subset \R^n$ and $\Omega = \R^n \setminus \R^d$, and $A$ satisfies the higher co-dimensional $C^{0,1}$ condition in $B_R(X_0)$ with constants $C_0, \lambda >0$. Then there is a constant $C$ depending only on the constants $n$, $d$, $C_0$, $\lambda$, and the scale $R >0$ for which 
\begin{align*}
\abs{u} + \delta \abs{\nabla u} \le C \delta \sup_{B_R(X_0)} \abs{u}
\end{align*}
holds in $B_{R/2}(X_0)$. In particular, $\abs{\nabla u} \in L^\infty(B_{R/2}(X_0))$.
\end{lemma}

\begin{proof}
First, we claim that for $M= M (n,d,C_0, \lambda, R) >0$ chosen large enough and $r_0 = r_0(n,d,C_0, \lambda, R) >0$ small enough, the function $g(x,t) \coloneqq \abs{t} - M \abs{t}^2$ is a supersolution of \eqref{eqn:soln_var} in $B_{r_0}(Y)$ for any $Y \in B_{R/2}(X_0) \cap \Gamma$. This is merely a computation which we outline the details of now.

Write $A = \abs{t}^{-n+d+1} \A$ as in Definition \ref{defn:c1a}, and denote $\A = (\ol{a}_{ij})$ with the block structure
\begin{align*}
\A = \begin{pmatrix}
J_1 & \A_1 \\
J_2 & \A_2 \\
\end{pmatrix},
\end{align*}
where $J_1$ is a $d \times d$ matrix, and $\A_2$ is an $(n-d) \times (n-d)$ matrix. We compute that $\nabla g = t/\abs{t} - 2 M t$, and thus 
\begin{align*}
\abs{t}^{-n+d+1} \A \nabla g = \abs{t}^{-n+d+1} \left( \sum_{k=1}^{n-d} \ol{a}_{i(k+d)} \left( t_k/\abs{t} - 2 M t_k \right) \right).
\end{align*}
Further computation shows
\begin{align*}
\divv( \abs{t}^{-n+d+1} \A \nabla g ) & = (-n+d+1) \abs{t}^{-n+d-1} \dotp{t, \A_2 (t/\abs{t} - 2Mt)}  \\
& \qquad + \abs{t}^{-n+d+1} \Tr(\A_2 D^2(\abs{t} - M\abs{t}^2 )) + \abs{t}^{-n+d+1} O(C_0 (1 + M\abs{t})),
\end{align*}
where the error comes from differentiating the coefficients of $\A$, which are Lipschitz with constant $\le C_0$. Next, recalling the block structure of $\A_2(x,0) = h(x,0) I_{n-d}$ at the boundary, we can rewrite the above (again using the Lipschitz nature of $\A$) after multiplying by $\abs{t}^{n-d-1}$:
\begin{align*}
\abs{t}^{n-d-1} & \divv(    \abs{t}^{-n+d+1}  \A \nabla g )  \\
&  =  h(x,0) \left( (-n+d+1)\abs{t}^{-2} \dotp{ t, (t/\abs{t} -2Mt) }  + \Tr (D^2(\abs{t} - M\abs{t}^2) \right)  +O(C_0(1+M\abs{t}))   \\
& =  h(x,0) \left( (-n+d+1)  (\abs{t}^{-1} - 2M )    + (n-d-1)\abs{t}^{-1} -2M(n-d) \right) +O(C_0(1+M\abs{t})) \\
& =  - 2M h(x,0) + O(C_0(1 + M \abs{t}))  \le -2M \lambda + O(C_0(1 + M \abs{t})),
\end{align*}
since $\Tr(D^2\abs{t}) = (n-d-1)/\abs{t}$ and $\Tr(D^2 \abs{t}^2) = 2(n-d)$. For $Y \in B_{R/2}(X_0) \cap \Gamma$, the above clearly implies that for $M$ sufficiently large and $r_0$ sufficiently small, $\divv(\abs{t}^{-n+d+1} \A \nabla g) \le 0$ in $B_{r_0}(Y)$ which is to say that $g$ is a supersolution of \eqref{eqn:soln_var}. By taking $M$ larger and $r_0$ smaller if necessary, notice that we also have the estimate
\begin{align}\label{eqn:comp2p}
g(Y) \simeq \abs{t} = \delta(Y) \text{ for all } Y \in B_{R/2}(X_0), \, \delta(Y) < r_0,
\end{align}
with implicit constant depending on $M$ and $r_0$.

Now, fix some $Y \in B_{R/2}(X_0) \cap \Gamma$. Let $v$ be a solution to the problem
\begin{align*}
-\divv(A \nabla v) & = 0 \text{ in }B_{r_0}(Y) \cap \Omega, \\
v & = \abs{u} \text{ on } \partial (B_{r_0}(Y) \cap \Omega).
\end{align*}
The existence of such a solution is guaranteed, for example, by \cite[Lemma 12.13]{DFM20MIXED}, and the fact that $B_{r_0}(Y) \setminus \R^d$ belongs to one of the permissible domains for which the elliptic theory of \cite{DFM20MIXED} holds (we omit the details). Since $v$ is a solution, and $\abs{u} = u_+ + u_-$ is a subsolution by \cite[Lemma 8.5]{DFM21AMS}, the comparison principle \cite[Lemma 11.14]{DFM21AMS} then says that $0 \le \abs{u} \le v$ in $B_{r_0}(Y) \cap \Omega$. Essentially the same comparison principle argument allows us to find a nonnegative solution $G$ which satisfies $0 \le G \le g$ in $B_{r_0}(Y) \cap \Omega$.

Since $v$ and $g$ are positive solutions of \eqref{eqn:soln_var} which vanish on $B_{r_0}(Y) \cap \Gamma$, the boundary comparison principle \cite[Theorem 11.17]{DFM21AMS} tells us that 
\begin{align}\label{eqn:comp1p}
v(X) \le C G(X) \left(  \dfrac{v(Z)}{G(Z)}  \right) \text{ in } B_{r_0/2}(Y)
\end{align}
for some constant $C = C(n,d,\lambda)  >0$, where $Z = Z(Y)$ is any corkscrew point associated to the boundary ball $B_{r_0}(Y)$, say, $Y + (r_0/2)e$ where $e$ is any unit vector orthogonal to $\R^d$. Notice that $\tau \coloneqq \inf_{Y \in B_{R/2}(X_0) \cap \Gamma}  G(Z(Y)) > 0$ is some positive constant again depending on $n, d, C_0, \lambda$ and $R>0$, but independent of $u$ (here, one can show by compactness that $\tau$ depends on the operator $A$ only through the associated constants $\lambda$ and $C_0$). Along with our previous estimate $v \le \sup_{B_R(X_0)} \abs{u}$, \eqref{eqn:comp2p}, \eqref{eqn:comp1p} imply for every $X \in B_{r_0/2}(Y)$,
\begin{align*}
0 \le \abs{u(X)} \le v(X) \lesssim G(X) \sup_{B_R(X_0)} \abs{u} \le  g(X) \sup_{B_R(X_0)}  \abs{u} \lesssim \delta(X) \sup_{B_R(X_0)} \abs{u}.
\end{align*}
This gives the first estimate of the Lemma in $B_{r_0/2}$, and the second follows quickly from standard Schauder theory. Finally, applying this estimate to smaller balls centered in $\R^d \cap B_{R}$ allows one to obtain the full estimate claimed in $B_{R/2}$.
\end{proof}

Finally, the following standard estimate says that solutions to elliptic equations of the type \eqref{eqn:soln_var} where the coefficients $A$ satisfy the higher co-dimensional $C^{0,1}$ condition are well-approximated by solutions to constant coefficient equations at small scales.
\begin{lemma}\label{lem:approx}
	Suppose that $u \in W_r(B_R(X_0))$ solves \eqref{eqn:soln_var} for $X_0 \in \Gamma = \R^d \subset \R^n$ and $\Omega = \R^n \setminus \R^d$, and $A$ satisfies the higher co-dimensional $C^{0,1}$ condition in $B_R(X_0)$ with constants $C_0, \lambda>0$, and $\A(X_0) = I$. Then there is a constant $C$ depending only on $n$, $d$, $\lambda$ for which the following holds. If $v \in W_r(B_r(X_0))$ with $r < R$ is the solution of the boundary value problem
	\begin{equation}
		\begin{aligned}
			-\divv(\abs{t}^{-n+d+1}  \nabla v   ) & = 0  && \text{ in } B_r(X_0) \cap \Omega, \\
			v & = u && \text{ on } \partial (B_r(X_0) \setminus \Gamma),
		\end{aligned}
	\end{equation}
	then we have the energy estimate
	\begin{align*}
		\fint_{B_r(X_0)} \abs{ \nabla(u-v)  }^2 \; dm \le C C_0 r \min \left \{  \fint_{B_r(X_0)} \abs{\nabla u}^2 \; dm, \fint_{B_r(X_0)} \abs{\nabla v}^2 \; dm    \right \}.
	\end{align*}
\end{lemma}
\begin{proof}
	Let $v$ be such a solution. Integrating by parts and the fact that $u = v$ on $\partial B_r(X_0)$ gives us
	\begin{align*}
		\int_{B_r(X_0)} \nabla (v-u) \cdot \nabla v \; dm & = 0,
	\end{align*}
	or in other words, 
	\begin{equation}\label{eqn:uv_orth}
		\int_{B_r(X_0)} \abs{\nabla v}^2 \; dm \equiv \int_{B_r(X_0)}  \nabla u \cdot \nabla v \; dm.
	\end{equation}
	
	Since $A$ satisfies the higher co-dimensional $C^{0,1}$ condition in $B_R(X_0)$, we may write $A(X) = \abs{t}^{-n+d+1} \A(X)$, with $\abs{\A(X) - I} \le C_0 \abs{X - X_0}$. It follows that 
	\begin{align*}
		\int_{B_r(X_0)} \A \nabla u \cdot \nabla u \; dm & = \int_{B_r(X_0)} \abs{\nabla u}^2 \; dm + \int_{B_r(X_0)} (\A(X) - I) \nabla u \cdot \nabla u \; dm \\
		& = ( 1 + O(C_0 r))\int_{B_r(X_0)} \abs{\nabla u}^2 \; dm,
	\end{align*}
	and similarly,
	\begin{align*}
		\int_{B_r(X_0)} \abs{\nabla v}^2 \; dm = (1 + O(C_0 r)) \int_{B_r(X_0)} \A \nabla v \cdot \nabla v \; dm.
	\end{align*}
	Recall that $\A$ is uniformly elliptic, and so we see that $\int_{B_r(X_0)} \abs{\nabla u}^2 \; dm \simeq \int_{B_r(X_0)} \abs{\nabla v}^2 \; dm$ rather quickly from the fact that 
	\begin{align*}
		\int_{B_r(X_0)} \A \nabla u \cdot \nabla u \; dm \le \int_{B_r(X_0)} \A \nabla v \cdot \nabla v \; dm, \\
			\int_{B_r(X_0)}  \abs{\nabla v}^2 \; dm \le \int_{B_r(X_0)} \abs{\nabla u}^2 \; dm,
	\end{align*}
	since $u$, $v$ are both solutions of (different) elliptic equations with the same trace on $\partial (B_r(X_0) \setminus \Gamma)$. Using this energy minimization again,
	we also deduce that the weighted Dirichlet energies of $u$ and $v$ are sufficiently close:
	\begin{align*}
		  \int_{B_r(X_0)} \abs{\nabla u}^2 \; dm & \le \int_{B_r(X_0)} \A \nabla u \cdot \nabla u \; dm + C C_0 r \int_{B_r(X_0)} \abs{\nabla u}^2 \; dm \\
		& \le \int_{B_r(X_0)} \A \nabla v \cdot \nabla v \; dm + C C_0 r \int_{B_r(X_0)} \abs{\nabla u}^2 \; dm \\
		& \le \int_{B_r(X_0)}  \abs{\nabla v}^2 \; dm + C C_0 r \left( \int_{B_r(X_0)} \abs{\nabla u}^2 + \abs{\nabla v}^2 \; dm \right).
	\end{align*}
	Finally, we combine the previous inequality with \eqref{eqn:uv_orth} to see
	\begin{align*}
		\int_{B_r(X_0)} \abs{ \nabla (v-u)}^2 \; dm & = \int_{B_r(X_0)} \abs{\nabla v}^2 - 2 \nabla v \cdot \nabla u + \abs{\nabla u}^2 \; dm \\
		& = \int_{B_r(X_0)}  \abs{\nabla u}^2 - \abs{\nabla v}^2 \; dm \\
		& \le C C_0 r \left( \int_{B_r(X_0)} \abs{\nabla u}^2 + \abs{\nabla v}^2 \; dm \right),
	\end{align*}	
	which concludes the proof since $\int_{B_r(X_0)} \abs{\nabla u}^2 \; dm \simeq \int_{B_r(X_0)} \abs{\nabla v}^2 \; dm$.
\end{proof}

Finally, we shall need in Section \ref{sec:pinch} a Moser-type inequality for solutions to equations of the type \eqref{eqn:soln_var} with a right-hand side. Since the proof uses the classical Moser iteration, we omit it, but instead point the reader to the proof of the Moser-type inequality in this higher co-dimension setting without right-hand side found in \cite[Lemma 8.12]{DFM21AMS}
\begin{lemma}\label{lem:moser}
	There is an exponent $q_0 > 1$ so that the following holds. Let $u$ be a solution of 
	\begin{equation}	
	\begin{alignedat}{4}
		- \divv( A \nabla u ) & = -\mathrm{div}(\delta^{-n + d+1}f) &  \text{ in }  & B_{2r} \cap (\R^n \setminus \R^d), \\
		u & = 0 & \text{ on } & B_{2r} \cap \R^d,
	\end{alignedat}
	\end{equation}
	where $\A = A \abs{t}^{n-d-1}$ is symmetric and uniformly elliptic with constant $\lambda >0$, and $f \in L^{q_0}(B_{2r}, dm)$. Then there is a constant $C = C(n,d, q_0, \lambda) > 1$ for which 
	\begin{align*}
		\sup_{B_r} \abs{u} \le C \left \{ \left(  \fint_{B_{2r}} u^2 \; dm \right)^{1/2} + r \left( \fint_{B_{2r}} \abs{f}^q \; dm \right)^{1/q} \right \}.
	\end{align*}
\end{lemma}

\subsection{\texorpdfstring{Dahlberg's Theorem in $B_1 \setminus \R^d$ and $L^2$ solvability of the Dirichlet problem}{Dahlberg's Theorem in the higher co-dimension ball}}
In this section, we justify the use of Theorem \ref{thm:dahl_mixed}, which was applied in Section \ref{sec:pinch} to obtain quantitative approximations for solutions to \eqref{eqn:soln_var} with almost constant frequency. First, let us introduce some notation.

We consider the domain $\Omega \coloneqq (\R^n \setminus \R^d) \cap B_1$ whose boundary is composed of two pieces:
\begin{align*}
	\partial \Omega = (\partial \Omega \cap \R^d) \cup (\partial \Omega \setminus \R^d) \eqqcolon \Gamma \cup \Sigma.
\end{align*}
Moreover, we equip the boundary with the natural choice of surface measure $\mu$ on $\partial \Omega$ defined by
\begin{align*}
	d\mu(X) \coloneqq d\HD^d|_\Gamma(X) + w(X) d\HD^{n-1}|_\Sigma(X),
\end{align*}
where as previously, $w(X) = \delta(X)^{-n+d+1}$. Along with absolutely continuous measure $dm(X) = w(X) dX$ on $\Omega$, such a choice of a boundary measure allows one to apply the elliptic theory developed in \cite{DFM20MIXED} for solutions to the degenerate elliptic equation
\begin{align*}
	Lu = -\divv(\delta^{-n+d+1} \nabla u) & = 0 \text{ in } \Omega.
\end{align*}
In particular, for each $X \in \Omega$ we have the existence of a (unique) probability measure $\omega_L^X$ on $\partial \Omega$ for which the function
\begin{align*}
	u_f(X) \coloneqq \int_{\partial \Omega} f \; d\omega^X_L
\end{align*}
is the unique weak solution to the continuous Dirichlet problem for $f \in C(\partial \Omega)$:
\begin{equation}\label{eqn:D2}
	\begin{split}
		Lu_f & = 0 \text{ in } \Omega, \\
		u_f & = f \text{ on } \partial \Omega.
	\end{split}
\end{equation}
See \cite{DFM20MIXED} Section 3.5 and Lemma 12.15. 

To state the conclusion of the Theorem we need a few more definitions. Define the non-tangential approach regions $\gamma_a(X)$ and the non-tangential maximal function $N_a(u)(X)$ defined for $X \in \partial \Omega$ and $a \in (0,1)$ by
\begin{align*}
	\gamma_a(X) \coloneqq \{ Y \in \Omega \; : \; \delta(Y) > a \abs{Y-X} \}, \qquad N_a(u)(X) \coloneqq \sup_{\gamma_a(X)} \abs{u(Y)},
\end{align*} 
The main result we have for this domain is the following.

\begin{thm}\label{thm:dahl_mixed}
	For any $X \in \Omega$, $\mu \ll \omega_L^X \ll \mu$. Moreover, there is some $\epsilon = \epsilon(n,d) >  0$ so that the density $\tfrac{d \omega_L^X}{d\mu}$ belongs to the reverse H\"{o}lder class $\mathrm{RH}_{2 + \epsilon}(\mu)$ in the sense that 
	\begin{align*}
		\left( \fint_{ B} \left (\dfrac{d\omega_L^X}{d \mu} \right )^{2 + \epsilon} d\mu \right)^{1/(2 + \epsilon)} \le C_X \fint_{B} \dfrac{d\omega_L^X}{d \mu} \; d\mu 
	\end{align*}
	holds for each boundary ball $B_r(X_0)$, $X_0 \in \partial \Omega$ and $0 < r \le 1$. 
	
	Moreover, the Dirichlet problem \eqref{eqn:D2} is $L^2(\mu)$-solvable in the sense that for each $f \in L^2(\mu)$, the function $u_f \in W_r(\Omega)$ defined by
	\begin{align*}
		u_f(X) \coloneqq \int_{\partial \Omega} f \; d\omega_L^X
	\end{align*}
	satisfies $Lu_f = 0$ in $\Omega$ with the corresponding estimates
	\begin{align*}
		\norm{ N_a(u_f)  }_{L^2(\mu)} \le C_a \norm{f}_{L^2(\mu)}, \qquad \limsup_{Y \ra X, \; : \; Y \in \gamma_a(X)} \abs{u(Y) - f(X)} = 0 \text{ for } \mu \text{ a.e. } X \in \partial \Omega.
	\end{align*}
\end{thm}
\begin{proof}
	The proof of this Theorem follows from well-developed techniques in elliptic boundary value problems, and the analogous statement for harmonic measure in the upper half ball, $B_1^+ \subset \R^{d+1}_+$. To avoid taking us too far from the current paper, let us provide a sketch of the main steps of the proof of the Theorem.
	
	For $y \in B_1^+$, we use the notation $\omega_{-\Delta}^y$ to denote the harmonic measure in $B_1^+ \subset \R^{d+1}_+$ with pole at $y$, i.e., the unique probability measure on $\partial B_1^+$ for which solution to the continuous Dirichlet problem for $v_f \in C^2(B_1^+) \cap C( \overline{\Omega} )$
	\begin{align*}
		\begin{split}
			-\Delta v_f & = 0 \text{ in } B_1^+, \\
			v_f & = f \text{ on } \partial B_1^+,
		\end{split}
	\end{align*}
	is given by the integral representation $v_f(y) = \int_{\partial B_1^+} f \; d\omega_{-\Delta}^y$. Since $B_1^+$ is a bounded Lipschitz domain, Dahlberg's Theorem implies that $\omega_{-\Delta}^y \ll \sigma \coloneqq \HD^d|_{\partial B_1^+} \ll \omega_{-\Delta}^y$, and in fact the density $k^y(z) \coloneqq \tfrac{ d\omega^y_{-\Delta} }{d\sigma} (z)$ is in the reverse H\"{o}lder class $\mathrm{RH}_{2+\epsilon}(\sigma)$ for some $\epsilon > 0$ for each $y \in B_1^+$. More specifically, from the presentation of Dahlberg's Theorem in \cite[Theorem 10.1]{PTHM24} followed by Gehring's Lemma on the Ahlfors regular set $\partial B_1^+$, we know that for each boundary ball $B_r(y_0)$ with $y_0 \in \partial B_1^+$ and $ 0 < r < 1$ and each $y \in B_1^+$ for which $\dist(y, 2B_r(y_0) \cap \partial B_1^+) \ge \Lambda^{-1} r$, we have that 
	\begin{align*}
		\left(\fint_{B_r(y_0) \cap \partial \Omega} \left(\dfrac{d\omega^y_{-\Delta}}{d\sigma} \right)^{2 + \epsilon} \;d\sigma \right)^{1/ (2+\epsilon)} \le C_\Lambda \fint_{B_r(y_0) \cap \partial \Omega} \dfrac{d\omega^y_{-\Delta}}{d\sigma} \; d\sigma,
	\end{align*}
	with $C_\Lambda >1$ depending only on the dimension and $\Lambda >1$. Finally, from the uniform doubling property of harmonic measure in $B_1^+$, \cite[Lemma 8.17]{PTHM24}, the above improves to the reverse H\"{o}lder condition
	\begin{align} \label{eqn:rh_hm}
		\left(\fint_{B_{\Lambda r}(y_0) \cap \partial \Omega} \left(\dfrac{d\omega^y_{-\Delta}}{d\sigma} \right)^{2 + \epsilon} \;d\sigma \right)^{1/ (2+\epsilon)} \le C_\Lambda \fint_{B_r(y_0) \cap \partial \Omega} \dfrac{d\omega^y_{-\Delta}}{d\sigma} \; d\sigma,
	\end{align}
	uniformly for $y \in B_1^+$ for which $\dist(y, \partial B_1^+) \ge \Lambda^{-1} r$, where $\Lambda >1$ is any fixed number. It is precisely condition \eqref{eqn:rh_hm} that we wish to transfer to the density $\tfrac{d\omega_L^Y}{d\mu}$. In fact we prove a much stronger estimate saying these two densities are essentially bounded equivalent (see estimate \eqref{eqn:rn_deriv}). This equivalence shall be a somewhat straightforward consequence of the fact that the measures $\mu$ and $\omega^Y_L$ are in some sense just rotations of the measures $\sigma$ and $\omega^y_{-\Delta}$ respectively, which we now justify.
	
	For subsets $F \subset \R^{d+1}_+$, define the rotation $R(F) \subset \R^n$ by
	\begin{align*}
		R(F) \coloneqq \{ (x,t) \in \R^d \times \R^{n-d} \; : \; (x,\abs{t}) \in F\}.
	\end{align*}
	Similarly, if $E \subset \R^n$, we abuse notation and still write $R(E)$ for the rotation of $E$ in $\R^n$, 
	\begin{align*}
		R(E) \coloneqq \{ (x, Bt) \; : \; (x,t) \in E, \; B \in O(n-d)\},
	\end{align*}
	where $O(n-d)$ is the orthogonal group on $\R^{n-d}$. Associate to each $y \in \R^{d+1}_+$ a well-defined choice of a point $\tilde{y} \in \R^n \cap R(\{y\})$, and associate for $Y \in \R^n$ the unique point $\tilde{Y} \in \R^{d+1}_+$ for which $Y \in R(\{\tilde{Y}\})$. It is a straight-forward consequence of Remark \ref{rmk:harm_rot} and the definition of the measures $\omega_{-\Delta}^y$ and $\omega_{L}^{Y}$ that the two are essentially the same:
	\begin{equation}\label{eqn:hm_rot2}
		\omega_{L}^{Y}(R(F)) \equiv \omega_{-\Delta}^{\tilde{Y}}(F), \qquad  \; Y \in \Omega, \; F \subset \partial B_1^+ \text{ Borel.}
	\end{equation}
	Indeed, given $f \in C(\partial B_1^+)$, let $v_f$ be its harmonic extension inside $B_1^+$. Defining $v(x,t) = v_f(x,\abs{t})$, we see that Remark \ref{rmk:harm_rot} implies that $v \in C(\overline{\Omega})$ is a solution of \eqref{eqn:D2} with boundary data $\tilde{f}(x,t) = f(x, \abs{t})$, and so if $Y = (x,t)$, then $\tilde{Y} = (x, \abs{t})$ and we have
	\begin{align*}
		\int_{\partial \Omega} \tilde{f} \; d\omega^{Y}_L  = v(Y) = v_f(\tilde{Y}) & = \int_{\partial B_1^+} f \; d\omega^{\tilde{Y}}_{-\Delta}.
	\end{align*}
	Choosing an appropriate sequence of continuous functions, one can use the above to show that \eqref{eqn:hm_rot2} holds for $F$ open or compact. The claim then follows for Borel $F$ using the fact that $\omega^{\tilde{Y}}_{-\Delta}, \omega^{Y}_{L}$ are inner and outer regular and a routine approximation argument.
	
	Next, we claim that for some dimensional constant $C >1$, we have that 
	\begin{align}\label{eqn:sm_rot}
		C^{-1} \mu(R(B_r(y_0))) \le \sigma(B_r(y_0)) \le C \mu(R(B_r(y_0))), \qquad y_0 \in \partial B_1^+, 0 < r < 1.
	\end{align}
	Notice that when $B_r(y_0) \cap \partial B_1^+ \subset \partial B_1^+ \cap \partial \R^{d+1}_+$, the above as actually an equality just by the definitions of $\mu$ and $\sigma$. On the other hand, for $y_0 \in \partial B_1^+ \setminus \partial \R^{d+1}_+$ and $r > 0$ sufficiently small, it is easy to check that $\HD^{n-1}(R(B_r(y_0)) \cap \Sigma ) \simeq r^d \dist(y_0, \partial \R^{d+1}_+)^{n-d-1}$ based on different considerations for when $\dist(y_0, \partial \R^{d+1}_+)$ is small or large. From the definition of $\mu$, we then see that for $r$ sufficiently small, $\mu(R(B_r(y_0))) \simeq r^d \simeq \sigma(B_r(y_0))$ in this case. Finally, the last piece of information we need is the rotation invariance of $\omega^Y_L$ and $\mu$; whenever $F \subset \partial \Omega$ is Borel and $T$ is a linear isometry which fixes $\R^d$, then 
	\begin{align}\label{eqn:hm_rot3}
		\mu(T(F)) = \mu(F), \qquad \omega^{T(Y)}_L(T(F)) = \omega^Y_L(F).
	\end{align}
	For the measure $\mu$ this property is immediate from the definition, while for $\omega^Y_L$ it follows again from Remark \ref{rmk:harm_rot}, since the operator \eqref{eqn:D2} is invariant under rotations in the $t$ variable. In particular, we always know from uniqueness of solutions (Lemma 12.2 \cite{DFM20MIXED}) that for $f \in C(\partial \Omega)$, we have that the solutions $u_f$ and $u_{f \circ T}$ as in the notation from \eqref{eqn:D2} are related by $u_f \circ T = u_{f \circ T}$, since $u_f \circ T$ is a solution as well, and has boundary data $f \circ T$. Evaluating this equality at $Y$ means in terms of elliptic measures that
	\begin{align*}
		\int_{\partial \Omega}  f \; d\omega^{T(Y)}_L =   \int_{\partial \Omega} f \circ T \; d\omega_L^Y,
	\end{align*}
	from which the second equality in \eqref{eqn:hm_rot3} follows easily by taking a sequence of continuous functions $f_k$ approximating $\chi_F$ and recalling that $T$ is invertible.

	At this stage, we can compare the densities $\tfrac{d\omega_L}{d\mu}$ and $\tfrac{d \omega_{-\Delta}}{d\sigma}$. Fix a pole $Y_0 = (0, \dotsc, 0, 1/2) \in \Omega$, $Z \in \partial \Omega$ and $0 < r \ll 1$. Notice that as $Y \ra \omega_L^Y(B_r(Z))$ is a nonnegative solution of the elliptic equation \eqref{eqn:D2} and $R(\{Y_0\})$ is uniformly bounded away from $\partial \Omega$, Harnack's inequality \cite[Lemma 11.35]{DFM20MIXED} implies that for each $Y \in R(\{Y_0\})$ we have  $\omega^Y_L(B_r(Z)) \simeq \omega_L^{Y_0}(B_r(Z))$. Moreover, we may use Vitali's covering Lemma to cover $R(B_r(Z))$ with $N_{r,Z}$ balls $\{5 B_i\}_{i=1}^{N_{r,Z}}$ so that the $B_i$ are disjoint, and moreover, such that the $B_i$ have radius $r$ and are centered in $R(\{Z\})$. Using the doubling property of elliptic measure \cite[Lemma 15.43]{DFM20MIXED} which is valid for, say, $0 < r < 1/16$, it is easy to check that from our covering that
	\begin{align*}
		\omega_L^{Y_0}(R(B_r(Z))) \simeq \sum_{i=1}^{N_{r,Z}} \omega^{Y_0}_L(B_i) = \sum_{i=1}^{N_{r,Z}} \omega_L^{Y_0}(T_i(B_r(Z))),
	\end{align*}
	where $T_i$ is an appropriately chosen linear isometry mapping $B_r(Z)$ to $B_i$ fixing $\R^d$ pointwise. Recalling \eqref{eqn:hm_rot3} and that $\omega_L^Y \simeq \omega_L^{Y_0}$ on $R(\{Y_0\})$, the above gives us the comparability
	\begin{align*}
		\omega_L^{Y_0}(B_r(Z)) \simeq N_{r,Z}^{-1} \omega_L^{Y_0}(R(B_r(Z))),
	\end{align*}
	and of course the same argument applies to $\mu$:
	\begin{align*}
		\mu(B_r(Z)) \simeq N_{r,Z}^{-1} \mu(R(B_r(Z))).
	\end{align*}
	Combining the previous two inequalities with \eqref{eqn:hm_rot2} and \eqref{eqn:sm_rot}, we conclude for $Z \in \partial \Omega$ and all $r >0$ small that 
	\begin{align*}
		\dfrac{ \omega_L^{Y_0}(B_r(Z)) }{\mu(B_r(Z))} \simeq \dfrac{ \omega_L^{Y_0}(R(B_r(Z)))}{\mu(R(B_r(Z)))} \simeq \dfrac{ \omega_{-\Delta}^{\tilde{Y_0}}(B_r(\tilde{Z})) }{\sigma(B_r(\tilde{Z}))}.
	\end{align*}
	
	With the same boundary ball $B_r(Z)$ with $Z \in \partial \Omega$ and $r > 0$ sufficiently small, let $\Lambda > 1$, and suppose that $Y \in B_{2\Lambda r}(Z)$ with $\dist(Y, \partial \Omega) > (2\Lambda)^{-1} r$. Notice that by definition, $Y$ is a corkscrew point for $B_{4\Lambda r}(Z)$ in the sense that $B_{(4\Lambda)^{-1}r}(Y) \subset \Omega \cap B_{4\Lambda r}(Y)$. Hence by the Harnack inequality and the change of pole property for elliptic measure,  \cite[Lemma 15.61]{DFM20MIXED}, we see that whenever $Q \in \partial \Omega \cap B_{\Lambda r}(Z)$ and $0 < \rho < r$,
	\begin{align*}
		\omega_L^{Y_0}(B_{\Lambda 	r}(Z))  \dfrac{\omega_L^Y(B_\rho(Q))}{\mu(B_\rho(Q))} \simeq_\Lambda    \dfrac{\omega^{Y_0}_L(B_\rho(Q))}{\mu(B_\rho(Q))}  \simeq_\Lambda \dfrac{ \omega_{-\Delta}^{\tilde{Y_0}}(B_\rho(\tilde{Q})) }{\sigma(B_\rho(\tilde{Q}))}.
	\end{align*}
	Upon sending $\rho \da 0$, we conclude that with the constant $c_{\Lambda, Z, r} = \omega_L^{Y_0}(B_{\Lambda r}(Z))$, 
	\begin{align}\label{eqn:rn_deriv}
		c_{\Lambda, Z, r} \dfrac{d\omega_L^{Y}}{d\mu}(Q) \simeq_\Lambda \dfrac{ d\omega_{-\Delta}^{\tilde{Y}_0} }{ d\sigma }(\tilde{Q}), \qquad Q \in  \partial \Omega \cap B_{\Lambda r}(Z).
	\end{align}
	It is a straight-forward consequence of \eqref{eqn:rn_deriv}, the comparability of $\mu$ and $\sigma$ as in \eqref{eqn:sm_rot}, and the previous Vitali-type argument above that \eqref{eqn:rh_hm} (which is invariant under multiplication by constants) transfers to $\tfrac{d\omega_L}{d\mu}$ to give us the reverse-H\"{o}lder estimate
	\begin{align*}
		\left(   \fint_{B_{\Lambda r}(Z)} \left(  \dfrac{d \omega_L^{Y}}{d\mu} \right)^{2+\epsilon} \; d\mu \right)^{1/(2+\epsilon)} \le C_\Lambda \fint_{B_r(Z) \cap \Omega} \dfrac{d\omega_L^Y}{d\mu} \; d\mu
	\end{align*}
	uniformly for $Z \in \partial \Omega$, $ 0 < r < r_0 \ll 1$, and $Y \in \Omega$ such that $\dist(Y, \partial \Omega) \ge \Lambda^{-1} r$. That the above reverse-H\"{o}lder type estimate implies that $L^2(\mu)$ solvability of the Dirichlet problem follows verbatim from the argument in \cite[Lemma 10.8]{PTHM24}, where boundary H\"{o}lder continuity of solutions with zero trace is guaranteed by \cite[Lemma 11.32]{DFM20MIXED}.
\end{proof}

In what follows, as in Section \ref{sec:pinch}, we order the discrete set of frequencies $\mathcal{F} = \{ \Lambda_1, \Lambda_2, \dotsc, \}$ with $\Lambda_1 = 1$ and $\Lambda_k < \Lambda_{k+1}$ for each $k \in \N$. In addition, for each $\Lambda_k$ with $k \in \N$, denote the linear space of $\Lambda_k$-homogeneous solutions $u \in W_r(\R^n) \cap C(\R^n)$ of \eqref{eqn:soln} which vanish on $\R^d$ by $\Hom_k$. We shall abuse notation and also write $\Hom_k$ for the restriction of such functions to $\S^{n-1}$. With this notation, the following result is analogous to the fact that the restrictions of homogeneous harmonic polynomials to $\S^{n-1}$ (i.e., spherical harmonics) are dense in $L^2(d\sigma)$.
\begin{cor}\label{cor:homogen_dense}
	The subspaces $\{\Hom_k\}_{k=1}^\infty$ of $L^2(\S^{n-1}, d\mu)$ are finite-dimensional, pairwise orthogonal, and have dense span. In particular, we have the isomorphism 
	\begin{align*}
		L^2(\S^{n-1}, d\mu) \simeq \bigoplus_{k\ge 1} \Hom_k,
	\end{align*}
	and so each element $f \in L^2(\S^{n-1}, d\mu)$ may be written as a series
	\begin{align*}
		f & = \sum_{k \ge 1} a_k \phi_k
	\end{align*}
	with $\{a_k\} \in \ell^2$, $\phi_k \in \Hom_k$, and $\norm{\phi_k}_{L^2(d\mu)} = 1$.
\end{cor}
\begin{proof}
	The fact that $\Hom_k$ are finite-dimensional was pointed out previously in Remark \ref{rmk:h_compt}. First, fix $f \in L^2(\S^{n-1}, d\mu)$. By Theorem \ref{thm:dahl_mixed}, we know that there exists a solution $u_g \in W_r(\Omega)$ satisfying $L u_g = 0$ in $\Omega$ with boundary data $g= g_f \in L^2(\partial \Omega, d\mu)$ defined by
	\begin{align*}
		g(X) = \begin{cases} f(X) & X \in \Sigma, \\
			0 & X \in \Gamma.
		\end{cases}
	\end{align*}
	Of course here, we mean that $u_g$ has boundary data $g$ in the sense that $N_a(u_g) \in L^2(\partial \Omega, d\mu)$ with the non-tangential convergence as in Theorem \ref{thm:dahl_mixed}. Next, for $k \in \N$ and any fixed $0 \ne \phi_k \in \Hom_k$ nonzero, define the function
	\begin{align*}
		\psi_f^k(r) \coloneqq \fint_{\partial B_r} u_g \phi_k \; d\mu \equiv \fint_{\partial B_1} u_g(rY) \phi_k(rY) \; d\mu(Y).
	\end{align*}
	A straight-forward computation shows that for $0 < r < 1$, 
	\begin{align*}
		\dfrac{d}{dr} \psi_f^k(r) & = \fint_{\partial B_r} (\partial_n u_f) \psi_k + u_f (\partial_n \psi_k) \; d\mu \\
		& = 2 \fint_{\partial B_r} u_g (\partial_n \phi_k) \; d\mu  = \dfrac{2 \Lambda_k}{r} \fint_{\partial B_r} u_g \phi_k \; d\mu \equiv \dfrac{2 \Lambda_k}{r} \psi^k_f(r).
	\end{align*}
	Indeed the second equality follows from the fact that 
	\begin{align*}
		\int_{\partial B_r} (\partial_n u_f) \psi_k - u_f (\partial_n \psi_k) \; d\mu = 0
	\end{align*}
	since $u_f, \phi_k$ are solutions of \eqref{eqn:soln} with zero trace on $\R^d \cap B_r$, and the second from the fact that $\phi_k$ is homogeneous of degree $\Lambda_k$.
	
	In particular, notice that if we take $f = \phi_j \in \Hom_j$ with $j \ne k$, then $\psi^k_{\phi_j} = \psi^j_{\phi_k}$, and the computation above shows us that 
	\begin{align*}
		\dfrac{r}{2} \dfrac{d}{dr} \psi_{\phi_j}^k(r) = \Lambda_k \psi_{\phi_j}^k(r) = \Lambda_j \psi_{\phi_j}^k(r).
	\end{align*}
	This shows that $\psi_{\phi_j}^k(r) \equiv 0$ since $\Lambda_k \ne \Lambda_j$, and in particular this proves the first claim that the spaces $\Hom_k$ are pairwise orthogonal in $L^2(\S^{n-1}, d\mu)$. Now let us show by contradiction that the $\{\Hom_k\}_{k \ge 1}$ have dense span in $L^2(\S^{n-1}, d\mu)$. Assume on the contrary that there is $0 \ne f \in L^2(\S^{n-1}, d\mu)$ for which 
	\begin{align*}
		\int_{\partial B_1} f \phi_k \; d\mu = 0, \text{  for all } k \in \N, \, \phi_k \in \Hom_k.
	\end{align*}
	With the notation as before, we recall that 
	\begin{align*}
		\dfrac{d}{dr} \psi^k_f(r) = \dfrac{2 \Lambda_k}{r} \psi_f^k(r),
	\end{align*}
	while $\lim_{r \ra 1^{-}} \psi^k_f(r) = \fint_{\partial B_1} f \phi_k \; d\mu = 0$. This limit can be justified from the fact that the functions $u_g(r \, \cdot\, ) \ra f$ in $L^2(\S^{n-1}, d\mu)$ as $r \ra 1^{-}$ since they are bounded by $N_a(u_g) \in L^2(\S^{n-1}, d\mu)$ and converge pointwise $\mu$ almost everywhere to $g_f=f$ on $\partial B_1$, as per Theorem \ref{thm:dahl_mixed}. It follows from the Picard-Lindel\"{o}f Theorem that $\psi^k_f(r) \equiv 0$, and thus 
	\begin{align*}
		\fint_{\partial B_r} u_f \phi_k \; d\mu = 0, \text{  for all } k \in \N, \, \phi_k \in \Hom_k, \, 0 < r < 1.
	\end{align*}
	However, the above is a clear contradiction to Corollary \ref{cor:tangent_maps}, and so the Corollary is proved.
\end{proof}

\begin{rmk}\label{rmk:sph_rep}
	By Corollary \ref{cor:homogen_dense}, we know that $L^2(\S^{n-1}, d\mu) \simeq \bigoplus_{k \ge 0} \Hom_k$, and each $\Hom_k$ is finite-dimensional. Thus we may choose for each $k \in \N$ a finite-dimensional orthonormal basis of $\Hom_k$, $\{\phi_j^k\}_{1 \le j \le N_k}$ where $N_k = \dim \Hom_K$. Thus each $f \in L^2(\S^{n-1}, d\mu)$ may be expressed uniquely by a sum
	\begin{align}
		f = \sum_{k \ge 1} \sum_{j=1}^{N_k} a_j^k \phi_j^k
	\end{align}
	with convergence in $L^2(d\mu)$.
\end{rmk}

\subsection{More spherical harmonics: some computations}

Much like in the case of harmonic functions, there are more homogeneous solutions to the ``constant-coefficient'' equation \eqref{eqn:soln} if we consider solutions away from the origin; in such cases, we obtain more solutions which are homogeneous of negative degree. Moreover, such solutions can easily be computed in terms of those we characterized in Theorem \ref{thm:homogen_sol}, which shall turn out to be useful in studying solutions of \eqref{eqn:soln} in annuli. In this section we write (in usual polar coordinates, as opposed to section \ref{sec:const}), $r = \abs{X}$ and $\theta = X / \abs{X}$.

\begin{lemma}\label{lem:neg_homogen_soln}
	Given $c_R, c_\rho \in \R$ and $\phi \in \Hom_k$, the unique solution of the boundary value problem
	\begin{equation}\label{eqn:annulus_1}
		\begin{alignedat}{3}
			-\divv(\abs{t}^{-n+d+1} \nabla u)&  = 0 & \, \,  \text{ in }  &  (B_R \setminus B_\rho) \setminus \R^d, \\
			u  & = c_R \phi & \,\, \text{ on } & \partial B_R, \\
			u  & = c_\rho \phi &  \, \, \text{ on } & \partial B_\rho, \\
			u & = 0 & \, \, \text{ on } & \R^d \cap (B_R \setminus B_\rho),
		\end{alignedat}
	\end{equation}
	for $u \in W_r(B_R \setminus B_\rho)$ is given by 
	\begin{align*}
		u(r, \theta) & = (a r^{\Lambda_k} + b r^{-\Lambda_k - d + 1} ) \phi(\theta),
	\end{align*}
	with 
	\begin{align*}
		a =  \dfrac{c_R - \left( \tfrac{\rho}{R} \right)^{\Lambda_k + d -1}  c_\rho }{R^{\Lambda_k} - \left(\tfrac{\rho}{R} \right)^{\Lambda_k + d - 1}  \rho^{\Lambda_k}}, \qquad 
		b  = \dfrac{c_\rho - \left( \tfrac{\rho}{R} \right)^{\Lambda_k}  c_R }{\rho^{-\Lambda_k - d + 1} - \left(\tfrac{\rho}{R} \right)^{\Lambda_k}  R^{-\Lambda_k -d  + 1}}.
	\end{align*}
\end{lemma}
\begin{proof}
	Write $\Lambda = \Lambda_k$, and $u_\Lambda(r, \theta) = r^\Lambda \phi(\theta)$. Recalling Remark \ref{rmk:strong_soln}, we know that $u_\Lambda$ is a $\Lambda$-homogeneous solution of \eqref{eqn:soln} classically, and a simple computation reveals that $v_m(r, \theta) \coloneqq r^{m} u_\Lambda(r,\theta) = r^{\Lambda + m} \phi(\theta)$ also is a solution away from the origin for $m = -2 \Lambda - d + 1$. Indeed, using the fact that $\Delta (r^m) = m(m+n-2) r^{m-2}$ and the computation of the operator $\divv(\abs{t}^{-n+d+1} \nabla \, \cdot \,)$ as in the first line of \eqref{eqn:cyl}, we see that 
	\begin{align*}
		\abs{t}^{n-d+1} & \divv(\abs{t}^{-n+d+1} \nabla v_m) \\
		&  = \abs{t}^2  \Delta v_m + (-n+d+1) \nabla v_m \cdot t \\
		& = \abs{t}^2 \left(  r^m \Delta u_\Lambda  + 2 \nabla u_\Lambda \cdot \nabla (r^m) + u_\Lambda \Delta (r^m) \right) + (-n+d+1) ( r^m \nabla u_\Lambda + u_\Lambda \nabla (r^m)) \cdot t \\
		& = \abs{t}^2 \left( 2 m r^{m-1} \partial_r u + m(m+n-2) r^{m-2} u_\Lambda \right)   +  (-n+d+1)m  \abs{t}^2 r^{m-2}  u_\Lambda  \\
		& = \abs{t}^2 r^{m-2}  u_\Lambda  m  \left( m + 2\Lambda + d-1 \right) 
	\end{align*}
	where in the second to last line we use that $u_\Lambda$ is a solution, and in the last the fact that $u_\Lambda$ is $\Lambda$-homogeneous. Hence we see that $v_m$ is a solution for $m = -2\Lambda -d + 1$, which is to say that $\tilde{u}_\Lambda(r,\theta) \coloneqq v_{-2\Lambda - d + 1}(r,\theta) = r^{-\Lambda - d + 1} \phi(\theta)$ is a $(-\Lambda-d+1)$-homogeneous solution of \eqref{eqn:soln} away from $0$. Notice now that $u_\Lambda, \tilde{u}_\Lambda$ are two linearly independent solutions of \eqref{eqn:soln} in $B_R \setminus B_\rho$ whose boundary data is proportional to $\phi$ on each $\partial B_r$, $r \in (0, \infty)$. In particular, a solution of \eqref{eqn:annulus_1} is given by 
	\begin{align*}
		u(r,\theta) \coloneqq a r^\Lambda + b r^{-\Lambda - d +1}
	\end{align*}
	where $a,b$ are chosen to satisfy the linear system
	\begin{align*}
		a R^{\Lambda} + b R^{-\Lambda -d + 1} & = c_R \\
		a \rho^\Lambda + b \rho^{-\Lambda -d +1} & = c_\rho,
	\end{align*}
	or in other words,
	\begin{align*}
		a & =  \dfrac{c_R - \left( \tfrac{\rho}{R} \right)^{\Lambda + d -1}  c_\rho }{R^\Lambda - \left(\tfrac{\rho}{R} \right)^{\Lambda + d - 1}  \rho^\Lambda}, \\
		b & = \dfrac{c_\rho - \left( \tfrac{\rho}{R} \right)^{\Lambda}  c_R }{\rho^{-\Lambda - d + 1} - \left(\tfrac{\rho}{R} \right)^{\Lambda}  R^{-\Lambda -d  + 1}}.
	\end{align*}
	By the maximum principle (for elliptic equations, or \cite[Lemma 12.8]{DFM20MIXED}), $u$ is the only such solution.
\end{proof}

Applying the previous Lemma to each component of $L^2(d\mu)$ boundary data, we obtain the following Corollary.
\begin{cor}\label{cor:ann_soln}
	Given $f_\rho, f_R \in L^2(\partial B_1, d\mu)$, write
	\begin{align*}
		f_t & = \sum_{k \ge 1} \sum_{j=1}^{N_k} c_{j,t}^k \phi_j^k 
	\end{align*}
	for $t = \rho, R$. Then the unique solution of the boundary value problem
	\begin{equation}\label{eqn:annulus_2}
		\begin{alignedat}{3}
			-\divv(\abs{t}^{-n+d+1} \nabla u)&  = 0 & \, \,  \text{ in }  &  (B_R \setminus B_\rho) \setminus \R^d, \\
			u  & = f_t & \,\, \text{ on } & \partial B_t,  \, \, t = \rho, R, \\
			u & = 0 & \, \, \text{ on } & \R^d \cap (B_R \setminus B_\rho),
		\end{alignedat}
	\end{equation}
	for $u \in W_r(B_R \setminus B_\rho)$ is given by 
	\begin{align*}
		u(r, \theta) & = \sum_{k \ge 1} \sum_{j=1}^{N_k} (a_j^k r^{\Lambda_k} + b_j^k r^{-\Lambda_k -d +1}) \phi_j^k(\theta)
	\end{align*}
	with
	\begin{align*}
		a_j^k =  \dfrac{c_{j,R}^k - \left( \tfrac{\rho}{R} \right)^{\Lambda_k + d -1}  c_{j,\rho}^k }{R^{\Lambda_k} - \left(\tfrac{\rho}{R} \right)^{\Lambda_k + d - 1}  \rho^{\Lambda_k}}, \qquad 
		b_j^k  = \dfrac{c_{j, \rho}^k - \left( \tfrac{\rho}{R} \right)^{\Lambda_k}  c_{j,R}^k }{\rho^{-\Lambda_k - d + 1} - \left(\tfrac{\rho}{R} \right)^{\Lambda_k}  R^{-\Lambda_k -d  + 1}}.
	\end{align*}
	Here the boundary value attainment of $f_\rho, f_R$ is in the sense that $u(r  \, \cdot \, ) \ra f_t$ in $L^2(\partial B_1, d\mu)$, as $r \ra t$ for $t = \rho, R$.
\end{cor}
\begin{proof}
	One simply applies Lemma \ref{lem:neg_homogen_soln} to each component $\phi_j^k$ individually to see that the representation given in the Corollary is a formally a solution of \eqref{eqn:annulus_2}. Of course, one ought to check that the sum giving $u$ actually converges to a solution, that $u \in W_r(B_R \setminus B_\rho)$, and the boundary data is attained in $L^2(\partial B_1,  d\mu)$. This can be checked using the definition of the $a_j^k, b_j^k$. 
	
	To check this claim, by scale invariance we may assume $1 = \rho < R$. Then, from their definitions, we readily see that 
	\begin{align*}
		\abs{a_j^k} &  \lesssim R^{-\Lambda_k} \abs{c_{j, R}^k} + R^{-2\Lambda_k} \abs{ c_{j, \rho}^k}, \\
		\abs{b_j^k} & \lesssim \abs{c_{j, \rho}^k} + R^{-\Lambda_k} \abs{c_{j,R}^k}.
	\end{align*}
	It follows that $\{R^{\Lambda_k} a_j^k\}, \{b_j^k\} \in \ell^2$. Using the fact that the $\phi_j^k$ are orthonormal and $a_{j,1}^k + b_{j,1}^k = c_{j,1}^k$, it is easy to check from the dominated convergence theorem that for $r > 1$,
	\begin{align*}
		\int_{\partial B_1} \left(  u(r, \theta) - f_1  \right)^2 \; d \mu & = \sum_{k \ge 1} \sum_{j =1}^{N_k} \left(a_j^k r^{\Lambda_k} + b_j^k r^{-\Lambda_k -d +1} - c_{j,1}^k \right)^2  \\
		& = \sum_{k \ge 1} \sum_{j=1}^{N_k} \left( a_j^k (r^{\Lambda_k} - 1) + b_j^k (r^{-\Lambda_k -d +1} - 1) \right)^2 \\
		& \le 2 \sum_{k \ge 1} \sum_{j=1}^{N_k}  (a_j^k)^2  (r^{\Lambda_k} - 1)^2 + (b_j^k)^2 (r^{-\Lambda_k -d +1} - 1)^2 \ra 0
	\end{align*}
	as $r \ra 1^+$. A similar computation shows that $\lim_{r \ra R^-} u(r, \theta) = f_R$ in $L^2(\partial B_1, d\mu)$.
	
	To see that $u \in W_r(B_R \setminus B_\rho)$ and is a solution of \eqref{eqn:annulus_2}, we simply need to show uniform energy estimates for $u^M(r,\theta) = \sum_{k = 1}^M \sum_{j=1}^{N_k} (a_j^k r^{\Lambda_k} + b_j^k r^{-\Lambda_k -d + 1}) \phi_j^k(\theta)$ on compact subsets of $B_R \setminus B_\rho$ that are independent of $M$:
	\begin{align}\label{eqn:energy_est}
		\sup_{M \in \N} \int_K \abs{\nabla u^M}^2 \; dm  \le C_K < \infty
	\end{align}
	for $K \subset \subset B_R \setminus B_\rho$. Indeed, each $u^M$ is a solution (being a finite sum of solutions, by 
	Lemma \ref{lem:neg_homogen_soln}), and usual, solutions are maintained under strong limits. To this end, we in fact only need to obtain uniform estimates on
	\begin{align*}
		\int_{K} (u^M)^2 dm
	\end{align*}
	independent of $M$, since by the Caccioppoli inequality, such estimates on $u^M$ automatically imply \eqref{eqn:energy_est}. However, such estimates again follow easily from the bounds on $a_j^k, b_j^k$ and the orthogonality of $\{\phi_j^k\}$, so we omit the details.
\end{proof}

With the representation of solution coming from Corollary \ref{cor:ann_soln}, we have a simple representation of the frequency $N_u$ for solutions to constant-coefficient equations. Since the proof is essentially the same as the one for harmonic functions (see, for example \cite{HLNODAL}), we omit it.
\begin{lemma}\label{lem:N_polar}
	If $u \in W_r(B_R)$ is a solution of \eqref{eqn:soln}, then as per Corollary \ref{cor:ann_soln}, we know that $u$ is given in polar coordinates by
	\begin{align*}
		u(r,\theta) & = \sum_{k \ge 1} \sum_{j=1}^{N_k} j_k^k r^{\Lambda_k} \phi_j^k(\theta),
	\end{align*}
	where the $\{\phi_j^k\}$ are an orthonormal basis of $L^2(\partial B_1, d\sigma_w)$. With this notation, we have that 
	\begin{align*}
		r^{-d-1} \int_{B_r} \abs{\nabla u}^2 \; dm &= \sum_{k \ge 1} \sum_{j=1}^{N_k} \Lambda_k (a_j^k)^2 r^{2\Lambda_k}, \\
		r^{-d} \int_{\partial B_r} u^2 \; d\sigma_w &= \sum_{k \ge 1} \sum_{j=1}^{N_k} (a_j^k)^2 r^{2\Lambda_k}, 
	\end{align*}
	and consequently,
	\begin{align*}
		N_u(0,r) & = \dfrac{ \sum_{k \ge 1} \sum_{j=1}^{N_k} \Lambda_k (a_j^k)^2 r^{2\Lambda_k} }{ \sum_{k \ge 1} \sum_{j=1}^{N_k} (a_j^k)^2 r^{2\Lambda_k}  }.
	\end{align*}
\end{lemma}

\section{The change of variables}\label{sec:app_gr}																				 
In the theory of degenerate elliptic operators outside of sets of low dimension, there are a family of operators that behave in some sense like the Laplacian does in co-dimension one. Namely, if $\Gamma$ is $d$-Ahlfors regular set in $\R^n$, we define the regularized distance (for a fixed parameter $\beta >0$) by 
\begin{align}\label{eqn:reg_dist}
D_{\beta}(X) \coloneqq \left(  \int_\Gamma \dfrac{1}{\abs{X-Y}^{d+\beta} }  \; d\HD^d(Y)\right)^{-1/\beta}.
\end{align}
A straight-forward computation shows that $D_\beta \simeq_{n,d,\beta} \delta$, and $D_\beta \in C^\infty(\R^n \setminus \Gamma)$. With this distance, we define the degenerate elliptic operator
\begin{align}\label{eqn:L_beta}
L_\beta \coloneqq - \divv( D_{ \beta}^{-n+d+1} \nabla \, \cdot \,),
\end{align}
which falls into the category of the elliptic operators as in Section \ref{sec:notation}.

In the end, our goal in this section is to demonstrate that estimates on the singular set will hold for solutions associated to these operators $L_\beta$, $\Gamma$ is a $d$-Ahlfors regular graph of $C^{1,1}$ function, and $\beta >1$ (the restriction on $\beta$ will later become clear in the proof of Theorem \ref{thm:change_var}). Afterwords, we will pass such estimates on to the operator $L_\infty$, in order to prove Theorem \ref{thm:main} (see Remark \ref{rmk:nonsmooth_op}). In practice though, computations are made simpler by flattening the graph $\Gamma$ with a bi-Lipschitz change of variables $\rho$ and instead performing all of our analysis for solutions associated to a conjugated degenerate elliptic operator $L_{\rho, \beta} = - \divv( A_{\rho, \beta} \nabla \, \cdot \, )$ (i.e., the one obtained through the change of variables) in the flat space $\R^n \setminus \R^d$ (see \eqref{eqn:A_rho}). In the classical setting (i.e., when $d =n-1$), this change of variables is also $C^{1,1}$, and since the operator $L_{\rho, \beta}$ involves the Jacobian of $\rho$, the coefficients of $L_{\rho, \beta}$ will be Lipschitz. Hence for this operator, one expects unique continuation results and estimates on the nodal and critical sets of solutions. In the higher co-dimension setting, we have an analogous change of variables, but since the space normal to $\Gamma$ is at least 2-dimensional, the estimates are more technical to write down. Let us describe some aspects of this change of variables introduced in \cite{DFMDAHL}, with the goal of sketching a proof of the fact that the conjugated operator $L_{\rho, \beta}$ associated to the change of variables satisfies the higher co-dimension $C^{0,1}$ condition. See Theorem \ref{thm:change_var}

In what follows, suppose that $\Gamma$ is a $d$-dimensional, $d$-Ahlfors regular $C^{1,1}$ graph parametrized by $\Phi(x) = (x, \phi(x))$, with $\phi \in C^{1,1}(\R^d; \R^{n-d})$ satisfying the estimate
\begin{align}\label{eqn:phi_est}
\norm{\nabla \phi}_{\mathrm{Lip}(\R^d)} \le \cphii.
\end{align}
For $\epsilon > 0$ small, notice that \eqref{eqn:phi_est} implies that for $0 < r_0 \le \cphii^{-1} \epsilon/M$ and $X_0 = (x_0, t_0) \in \Gamma$, the graph $\Gamma$ in $B_{r_0}(X_0)$ can be rewritten as the graph of a $C^{1,1}$ function $\psi: V \ra V^\perp$ with $\norm{\psi}_{\mathrm{Lip}} \le \epsilon$ for some appropriately chosen $d$-plane $V \subset \R^n$. Here $M \gg 1$ is some constant which shall freely make large, but shall depend only on the dimensions $n$ and $d$.

Upon translating, localizing, rotating, and relabeling, we abuse notation and rewrite $\psi$ as $\phi$. Moreover we assume that $X_0 = (0, t_0)$. Hence we have that $\Gamma \cap B_{r_0}(X_0)$ coincides with the graph of a $C^{1,1}$ function, $\phi: B_{r_0} \cap \R^d \ra \R^{n-d}$ with $\nabla \phi(x_0) = 0$ and the estimates
\begin{align}\label{eqn:phi_est_2}
\norm{\nabla \phi}_{L^\infty(B_{r_0})} \le \epsilon,  \; \norm{\nabla \phi}_{\mathrm{Lip}(B_{r_0})} \le \cphii.
\end{align}
Of course, this comes at the expense of losing the representation of $\Gamma$ as the graph over a single function defined on $\R^d$, but this will not bother us too much. Now from the estimates \eqref{eqn:phi_est_2} one can follow the procedure in \cite{DFMDAHL} to construct a bi-Lipschitz mapping \[\rho: B_{R_0} \subset \R^n \ra \rho(B_{R_0}) \supset B_{r_0}(X_0)\] with $R_0 \simeq r_0$ that maps $\R^d \cap B_{R_0}$ to $\Gamma \cap \rho(B_{R_0})$. Here, the importance of localizing is that the smallness of $\norm{\nabla \phi}_{L^\infty(B_{r_0})}$ in \eqref{eqn:phi_est_2} allows one to guarantee that the mapping $\rho$ constructed is bi-Lipschitz. 

It is easy to see by direct computation (see for example \cite[Section 4]{DFMDAHL}) that $u$ is a solution of \eqref{eqn:L_beta} if and only if $v = u \circ \rho$ solves $-\divv( A_{\rho, \beta} \nabla v) = 0$ in the weak sense, where $A_{\rho, \beta}$ is defined by $A_{\rho, \beta} = \abs{t}^{-n+d+1} \A_{\rho, \beta}$, and 
\begin{equation}
\label{eqn:A_rho}
\begin{split}
\A_{\rho, \beta}(x,t) & = \left( \dfrac{\abs{t}}{D_{\beta}(\rho(x,t))} \right)^{n-d-1} \abs{\det \Jac(x,t)} (\Jac(x,t)^{-1})^T \Jac(x,t)^{-1}.
\end{split}
\end{equation}
Here, $\Jac(x,t) = \nabla \rho(x,t)$ is the Jacobian. In view of these facts above, our main goal is to show that the matrix $\A_{\rho, \beta}$ satisfies the higher co-dimensional $C^{0,1}$ condition for some nice bi-Lipschitz map $\rho$ (i.e., Theorem \ref{thm:change_var}). A proof of this fact unfortunately requires the recounting the details of the construction and properties of $\rho$, which takes us a bit far from the goals of the current paper. As such we rapidly outline what we need below.

\subsection{\texorpdfstring{The higher co-dimension $C^{0,1}$ condition for $A_{\rho, \beta}$}{The higher co-dimension Lipschitz condition for the change of variables}}

Let us first recount in detail the construction of the mapping $\rho$, assuming that we have the condition \eqref{eqn:phi_est_2}. In fact, by taking $M$ large enough, notice we may assume \eqref{eqn:phi_est_2} holds in an enlarged neighborhood of $B_{r_0}(x_0) \cap \R^d$, which allows us to perform the following construction for points $x \in B_{r_0}$ and scales $0 < s < r_0$. We also adopt the notation in \cite{DFMDAHL} and write $\abs{t} = r$ when defining functions depending on $t$.

Take $\eta \in C_c^\infty(\R^d)$ to be a smooth, nonnegative, radial bump function supported $B_1(0)$ so that $\int \eta(x) \;  dx = 1$. Its dilations $\eta_s$ for $s>0$ are defined by $\eta_s(x)  = s^{-d} \eta \left( \tfrac{x}{s} \right)$
and the functions $\phi_s :B_{r_0} \cap \R^d \ra \R^{n-d}$ by $\phi_s(x)  = \eta_s * \phi(x)$ for $s >0$ small enough so that $\phi$ is defined in $B_s(x) \cap \R^d$. In addition, set $\Phi_s(x) = (x, \phi_s(x))$. Let $P(x,s)$ denote the tangent $d$-plane to $\phi_s(x)$ at $\Phi_s(x)$, and set $P'(x,s) = P(x,s) - \Phi_s(x)$ (so that $P'(x,s)$ contains the origin). Finally, the mapping $\rho: B_{R_0} \subset \R^n \ra \R^n$ is defined by
\begin{align}\label{eqn:rho}
\rho(x,t) & = \Phi_{r}(x) + h(x,r) R_{x,r}(0, t),
\end{align}
where $h(x,r) \simeq 1$ is a scalar and $R_{x,r}$ is an orthogonal transformation mapping $\R^d$ to $P'(x,r)$ defined as follows.

Fix a bump function $\theta \in C_c^\infty(\R^n)$ with $\chi_{B_{1/2}} \le \theta \le \chi_{B_1}$ and define
\begin{align}\label{eqn:lambda_def}
\theta_{x,r}(Y) \coloneqq \theta \left(  \dfrac{\Phi_r(x) - Y}{r} \right),  \qquad \lambda(x,r) \coloneqq  a_0^{-1} r^{-d}   \int _\Gamma \theta_{x,r}(Y) \; d\HD^d(Y),
\end{align}
where $a_0 \simeq_{n,d} 1$ is some constant depending on the dimensions and the particular choice of $\theta$. Notice that $\lambda(x,r) \simeq 1$ for $x \in B_{r_0}$ and $0 < r < r_0$ since $\Gamma$ is a $d$-dimensional $C^{1,1}$ graph. Finally, $h$ is defined by 
\begin{equation}\label{eqn:h_def}
h(x,t) = h(x,r) \coloneqq (c_\beta \lambda(x, r) )^{1/\beta},
\end{equation}
where $c_\beta \coloneqq \int_{\R^d} (1 + \abs{y}^2)^{(-d-\beta)/2} \; dy > 0$ has the defining property that if $V \subset \R^n$ is a $d$-plane and then one has 
\begin{align}\label{eqn:cbeta_prop}
D_{V, \beta}(X)^{-\beta} \coloneqq \int_{V} \abs{X-Y}^{d+\beta} \; d\HD^d(Y) \equiv c_\beta \dist(X, V)^{-\beta}, \; X \not \in V.
\end{align}
See for example \cite[Section 5.3]{DFMDAHL}.

As for the rotation $R_{x,r}$, we denote by $\{e^i\}_{i=1}^n$ the standard basis vectors, and set
\begin{equation}\label{eqn:v_def}
\begin{split}
\hat{v}^i(x,r) &  = \partial_{x_i} \Phi_r(x) = (e^i, \partial_{x_i} \phi_r(x)),  \; \; i = 1, \dotsc, d,\\
\hat{w}^j(x,r) & = ((-\nabla_x \phi_r^j)^T, e^j),  \; \; j = d+1, \dotsc, n,
\end{split}
\end{equation}
so that the $\{\hat{v}^i(x,r)\}_{i=1}^d$ and $\{\hat{w}^i(x,r)\}_{i=d+1}^n$ form bases for $P'(x,r)$ and $P'(x,r)^\perp$ respectively. Finally, $\{v_i(x,r)\}_{i=1}^d$ and $\{w_i(x,r)\}_{i=d+1}^n$ denote orthonormal bases for $P'(x,r)$ and $P'(x,r)^\perp$ obtained from $\{\hat{v}^i(x,r)\}_{i=1}^d$ and $\{\hat{w}^i(x,r)\}_{i=d+1}^n$ respectively via the Gram-Schmidt procedure (for details, see \cite[Section 2]{DFMDAHL}). With this basis, then we denote by $R_{x,r}$ the linear orthogonal mapping that maps the ordered basis $\{e^1, \dotsc, e^n\}$ to the ordered basis $\{v^1(x,r), \dotsc, v^d(x,r), w^{d+1}(x,r), \dotsc, w^n(x,r)\}$.

The authors in \cite{DFMDAHL} show that if $\epsilon >0$ is sufficiently small (or in otherwords, as long as $r_0$ is sufficiently small) then these choices of $h(x,r)$ and $R_{x,r}$ give that 
\begin{align}\label{eqn:rho_bilip}
\rho \text{ is bi-Lipschitz with constant } 1 + C \epsilon
\end{align}
from its domain $B_{R_0}$ onto its image. Moreover, they give estimates on $\Jac(x,t)$ which can be decomposed as 
\begin{align}\label{eqn:jax_decomp}
\nabla \rho(x,t) = \mathrm{Jac}(x,t) & = J(x,t) Q(x,t)^{T},
\end{align}
where $Q(x,t)^T$ is the orthogonal matrix giving the rotation $R_{x,r}$ in the sense that $R_{x,r}(y,u) = (y,u)Q(x,t)^T$ and $J(x,t)$ has good block-form structure\footnote{See the equations (3.9), (3.10) in \cite{DFMDAHL}.}. We should remark that in the definition above, we follow the convention in \cite{DFMDAHL} that Jacobian and matrix $Q(x,t)^T$ act by \textit{left} multiplication.

As for the matrix block structure of $J(x,t)$, one can prove that there is $d \times d$ matrix $J_1'(x,t)$ so that $J(x,t)$ asymptotically approaches
\begin{align} \label{eqn:J'}
J'(x,t) & = \begin{pmatrix}
J_1'(x,t) & 0 \\
0 & h(x,t) I_{n-d},
\end{pmatrix} 
\end{align}
near the boundary, $\R^d$. More specifically, one can decompose $J = J' + M + H$, and by \cite[Lemmas 3.15, 3.26, 5.49]{DFMDAHL} we have the estimates
\begin{equation}\label{eqn:m_h_est}
\begin{split}
\abs{J'_1(x,t) - I_{d} } & \le C \epsilon^2, \\
\abs{M(x,t)} & \le C \left( \abs{ \partial_r \phi_r} + r\abs{\nabla_{x,r} \nabla_x \phi_r} \right), \\
\abs{H(x,t)} & \le C r\abs{\nabla_{x,t} h(x, t)} \le C \alpha_\mu(\Phi(x), C r),
\end{split}
\end{equation}
where $\alpha_\mu$ are the Tolsa $\alpha$ numbers associated to $\mu = \HD^d|_\Gamma$ defined as follows (see also \cite{DFMDAHL} and \cite{TOLSAALPHA}).

\begin{defn}\label{defn:alpha}
Denote by $\mathrm{Flat}(n,d)$ the set of measures of the form $c \HD^d|_P$ where $c > 0$ and $P \subset \R^n$ is a $d$-plane. 
If $\mu$ is a $d$-Ahlfors regular measure, then define for $x \in \R^n$ and $r >0$, 
\begin{align*}
\alpha_{\mu}(x,r) & \coloneqq \inf_{\nu \in \mathrm{Flat}(n,d)} \left( \sup_{f \in \Lambda(B_r(x))} r^{-d-1}\abs{ \int f (d \mu - d\nu) } \right)
\end{align*}
where $\Lambda(B_r(x))$ is the space of $1$-Lipschitz functions with compact support in $B_r(x)$. 
\end{defn}

When $\phi$ is merely Lipschitz, \eqref{eqn:m_h_est} can be used to show that $\abs{M}$ and $\abs{H}$ satisfy Carleson measure estimates. In our setting, since $\phi$ is of class $C^{1,1}$, we can obtain higher order decay.

\begin{lemma}\label{lem:alpha_c11}
In the setting above, if $\mu = \HD^d|_\Gamma$, then we have the estimates
\begin{align*}
\alpha_\mu(X,r) \le C \cphii \epsilon r
\end{align*}
for $X \in B_r(X_0) \cap \Gamma$, $r < r_0$ where $C = C(n,d) > 0$.
\end{lemma}
\begin{proof}
Let $f$ be a  $1$-Lipschitz function with compact support in $B_r(X)$, and set $U = \pi(\Gamma \cap B_r(X)) \subset \R^d$. Then the area formula gives us
\begin{align*}
\int f \; d\mu & = \int_{U} f \circ \phi(y) \sqrt{1 + \abs{\nabla \phi(y)}^2} \; d\HD^d(y),
\end{align*}
while the estimates on $\nabla^2 \phi$ in \eqref{eqn:phi_est_2} and the mean value theorem for the function $g(s) = \sqrt {1 + s^2}$ for $0 \le s < \epsilon$ gives us 
\begin{align*}
\abs{ \sqrt{1 + \abs{\nabla \phi(y)}^2} - \sqrt{1 + \abs{ \nabla \phi(x)}^2 }}  \le  C \epsilon \abs{ \abs{\nabla \phi(y)} - \abs{\nabla \phi(x)}} \le C \cphii \epsilon \abs{x-y} \le C \cphii \epsilon r.
\end{align*}
Recalling that $f$ is $1$-Lipschitz, then $\norm{f}_{\infty} \le C r$, and thus for the flat measure $d\nu(y) = \left(1 + \abs{\nabla \phi(x)}^2 \right)^{1/2} \; dy$, we have 
\begin{align*}
r^{-d-1} \abs{ \int f \; (d\mu - d\nu)} \le r^{-d-1} \norm{f}_{\infty} \int_U \abs{ \sqrt{1 + \abs{\nabla \phi(y)}^2} - \sqrt{1 + \abs{\nabla \phi(x)}^2 }} \; d\HD^d(y) \le C \cphii \epsilon r,
\end{align*}
which is exactly the claim.
\end{proof}

The following estimates are straight-forward consequences of the definition of $\phi_r(x) = \eta_r * \phi(x)$, the fact that $\eta$ is radial, and the $C^{1,1}$ regularity of $\phi$. Notice that this one degree of regularity of an improvement of \cite[Lemma 3.17]{DFMDAHL}, since our graph $\Gamma$ is locally $C^{1,1}$.
\begin{lemma}\label{lem:phi_higher}
In the setting above, we have the estimates
\begin{align*}
r^{-1} \abs{\partial_r \phi_r(x)}  + \abs{\nabla_{x,r}^2  \phi_r(x)} + r \abs{ \nabla_{x,r}^3 \phi_r}  \le C \cphii, \; \abs{\nabla_x \phi_r} \le C \epsilon
\end{align*}
for $x \in B_{r_0}$ and $0 < r < r_0$ with $C = C(n,d)$.
\end{lemma}
\begin{proof}
Notice that we may write (by changing variables),
\begin{align}\label{eqn:phi_r}
\phi_r(x) & = \int \eta(z) \phi(x - rz) \; dz,
\end{align}
so differentiating in $r$ and using the fact that $\int \eta(z) z \; dz = 0$ since $\eta$ is radial, we see
\begin{align*}
\partial_r \phi_r(x) & = \int \eta(z) \dotp{\nabla \phi(x-rz), -z} \; d z \\
& = \int \eta(z) \dotp{ \nabla \phi(x -rz) - \nabla \phi(x), -z} \; dz.
\end{align*}
Using the Lipschitz bound on $\nabla \phi$ gives us the desired estimate on $\abs{\partial_r \phi_r}$. Differentiating \eqref{eqn:phi_r} similarly gives us the bounds on $\nabla_x \phi_r$ and $\nabla_{x,r}^2 \phi_r$ since $\abs{\nabla \phi} \le \epsilon$ and $\abs{\nabla^2 \phi} \le  \cphii$.

As for the estimate on $\nabla_{x,r}^3 \phi_r$, let us demonstrate just how to obtain the estimate on $\nabla_x \partial_r^2 \phi_r$, since the other derivatives are handled in a similar manner. Using \eqref{eqn:phi_r}, we can write
\begin{align*}
\partial_r^2 \phi_r(x) & = \sum_{i,j=1}^d \int \eta(z)z_i z_j \phi_{x_i x_j}(x - rz) \; dz \\
& = \sum_{i,j=1}^d \int \eta^{ij}_r(y) \phi_{x_i x_j}(x-y) \; dy = \eta^{ij}_r  * \phi_{x_i x_j}(x)
\end{align*}
for $\eta^{ij}(y) = \eta(y) y_i y_j$ and $\eta^{ij}_r = r^{-d} \eta^{ij}(y/r)$. Computing one more derivative in $x$ then gives us 
\begin{align*}
\nabla_x  \partial_r^2 \phi_r (x) = ( \nabla_x \eta^{ij}_r ) * \phi_{x_i x_j}(x),
\end{align*}
and so since $\abs{ \phi_{x_i x_j}} \le \cphii$ and $\abs{ \nabla_x \eta^{ij}_r} \le C r^{-1}$, this gives us the desired bound on $\nabla_x \partial_r^2 \phi_r(x)$. 
\end{proof}

In view of the previous Lemmas and \eqref{eqn:m_h_est}, we get the estimate $ \abs{M(x,t)} + \abs{H(x,t)} \le C \cphii r$, and so the errors $M$ and $H$ vanish linearly at the boundary $\R^d$:
\begin{cor}\label{cor:J_decomp}
In the setting above, there exists a decomposition of $J(x,t) = \Jac(x,t) Q(x,t)$ of the form $ J = J' + M + H$ where $J'$ is a uniformly elliptic matrix,
\begin{align*}
J'(x,t) & = \begin{pmatrix}
J_1'(x,t) & 0 \\
0 & h(x,t) I_{n-d} \\
\end{pmatrix},
\end{align*}
and $\abs{M(x,t)} + \abs{H(x,t)} \le C \cphii r$, where $C = C(n,d) > 0$.

\end{cor}

With these estimates, we can now give a sketch of proof of the claim we made earlier that the conjugated matrix $A_{\rho, \beta}$ satisfies the higher co-dimensional $C^{0,1}$ condition. Unfortunately, rigorously writing down the appropriate estimates for the matrix $A_{\rho, \beta}$ requires us to dig into some of the estimates and techniques from \cite{DFMDAHL}. We stress that the work done in \cite{DFMDAHL} provides most of the framework for the Theorem, and that one just needs to check that the improved estimates coming from Lemmas \ref{lem:alpha_c11} and \ref{lem:phi_higher} lead to higher regularity on the matrix $A_{\rho, \beta}$.

\begin{thm}\label{thm:change_var}
Let $\Gamma$ be the graph of $d$-dimensional $C^{1,1}$ function $\phi \in C(\R^d ; \R^{n-d})$ so that \eqref{eqn:phi_est} holds, and let $\beta > 1$. Then for each $X _0 = (x_0, t_0) \in \Gamma$ and $r_0$ sufficiently small (depending on $\cphii$, $n$, and $d$), the change of variables 
\begin{align*}
\rho:B_{r_0}((x_0, 0))\subset \R^n \ra \rho(B_{r_0}((x_0, 0))
\end{align*}
detailed as in \eqref{eqn:rho} is bi-Lipschitz (with constant at most $2$), and moreover, the corresponding conjugated matrix $A_{\rho, \beta}$ defined in \eqref{eqn:A_rho} satisfies the higher co-dimensional $C^{0,1}$ condition in $B_{r_0}((x_0, 0))$ with constants $\lambda = \lambda(n,d,\beta) >0$ and $C_0 =C_0(n,d,\beta, \cphii) >0$.
\end{thm} 
\begin{proof}[Sketch of proof]
First, one needs to check the Lipschitz nature of the uniformly elliptic part of the matrix $A_{\rho, \beta}$, i.e., $\A_{\rho, \beta} = \abs{t}^{n-d-1} A_{\rho, \beta}$ which is
\begin{align*}
\A_{\rho, \beta}(x,t) = \left( \dfrac{\abs{t}}{D_{\beta}(\rho(x,t))} \right)^{n-d-1} \abs{\det \Jac(x,t)} (\Jac(x,t)^{-1})^T \Jac(x,t)^{-1}.
\end{align*}
Lemma 3.26 in \cite{DFMDAHL} says that $\det{\Jac(x,t)} \simeq h(x,t)^{n-d} \simeq 1$ as long as $\epsilon$ in \eqref{eqn:phi_est_2} is small enough, so clearly $\Jac(x,t)$ is invertible. In addition, we have that $\abs{t} \simeq D_{\beta}(\rho(x,t))$ by \eqref{eqn:rho_bilip}, since $\abs{t} = \dist((x,t), \R^d)$, $D_{ \beta}(\rho(x,t)) \simeq \dist(\rho(x,t), \Gamma)$ by \cite[Lemma 5.1]{DFMDAHL}, and since $\rho$ maps $\R^d$ to $\Gamma$. Moreover, it is easy to see from Corollary \ref{cor:J_decomp} that $\Jac(x,t)$ is uniformly bounded and elliptic, and so the same holds for $\A_{\rho, \beta}$. In view of these facts, it then suffices to show just that 
\begin{align}\label{eqn:lip_reduc}
(x,t) \ra \Jac(x,t), \qquad  (x,t) \ra \dfrac{\abs{t}}{D_{\beta}(\rho(x,r))},
\end{align}
are Lipschitz functions in the variables $(x,t)$ in order to prove that $\A_{\rho, \beta}$ is Lipschitz. Let us start by showing $\Jac(x,t) = \nabla \rho(x,t)$ is Lipschitz by estimating $\nabla_{x,t}^2 \rho(x,t)$.

To this end, recalling the definition $\rho(x,t) = \Phi_r(x) + h(x,r) R_{x,r}(0,t)$ with $\Phi_r(x) = (x, \phi_r(x))$, Lemma \ref{lem:phi_higher} implies that $\abs{\nabla_{x,r}^2 \Phi_r(x)} \le C \cphii$. Thus we need only show that $\abs{\nabla_{x,t}^2 (h(x,r) R_{x,r}(0,t))} \le C$, which is a consequence of the inequalities
\begin{align}
\abs{\nabla_{x,t} h(x,r)} + r \abs{\nabla_{x,t}^2 h(x,r)} +  & \le C, \label{eqn:h_est}  \\
\abs{ \nabla_{x,t} (R_{x,r})}  + r \abs{ \nabla_{x,t}^2 ( R_{x,r})}& \le C \label{eqn:R_est}.
\end{align}
To be clear, in \eqref{eqn:R_est} we view $R_{x,r}$ as a matrix valued function. 

Recalling the definition of $h \simeq 1$ as in \eqref{eqn:h_def}, the estimates \eqref{eqn:h_est} follow from the same inequalities with $\lambda$ in place of $h$. The first estimate is then a direct consequence of \cite[Lemma 5.49]{DFMDAHL} which says $\abs{\nabla_{x,r} \lambda(x,r)} \le C \alpha_\mu(\Phi(x),Cr)/r$ and Lemma \ref{lem:alpha_c11}. One may also deduce the appropriate estimates on $\nabla_{x,r}^2 \lambda(x,r)$ using the same arguments as in the proof of that Lemma. Namely, one computes directly the components of 
\begin{align*}
\nabla_{x,r}^2 \left( r^{-d} \theta_{x,r}(Y) \right)  = \nabla_{x,r}^2 \left( r^{-d} \theta \left( \dfrac{\Phi_r(x) - Y}{r} \right) \right),
\end{align*}
each of which can be seen to be a $(\tilde{C} r^{-2})$-Lipschitz function by the estimates on $\phi_r$ in Lemma \ref{lem:phi_higher}, where $\tilde{C}$ is allowed to depend on $\cphii$. Comparing $\nabla_{x,r}^2 \lambda(x,r)$ with the corresponding quantity $\nabla_{x,r}^2 \tilde{\lambda}(x,r)$,
\begin{align*}
\tilde{\lambda}(x,r) \coloneqq a_0^{-1} r^{-d} \int_{P(x,r)} \theta_{x,r}(Y) \; d\mu_{x,r}(Y)
\end{align*}
where $\mu_{x,r}$ is an appropriately chosen flat measure $\mu_{x,r} = \lambda(x,r) \HD^d|_{P(x,r)}$ allows one to obtain the estimate
\begin{align*}
\abs{\nabla_{x,r}^2 \lambda(x,r)} \le \tilde{C} r^{-2} \alpha_{\mu}(\Phi(x), Cr) \le \tilde{C} r^{-1} \cphii \epsilon. 
\end{align*}
See the arguments in \cite[Lemmas 5.22, 5.49]{DFMDAHL}. By virtue of Lemma \ref{lem:alpha_c11}, this gives the second estimate in \eqref{eqn:h_est}.

The estimates on $R_{x,r}$ are slightly simpler. Recall that the rotation $R_{x,r} = R_{x,r}$ is given by mapping the standard orthonormal basis $\{e^i\}_{i=1}^n$ to another one, \[\{v^1(x,r), \dotsc, v^d(x,r), w^{d+1}(x,r) , \dotsc, w^{n}(x,r)\},\]
as in \eqref{eqn:v_def}. From the fact that $\abs{\hat{v}^i}, \abs{\hat{w}^i} \simeq 1$ (from the bound \eqref{eqn:phi_est_2} on $\nabla \phi$), and that the $v^i$, $w^i$ are obtained from the $\hat{v}^{i}, \hat{w}^i$ from a bounded number of algebraic manipulations, it is easy to see that $R_{x,r}$ is Lipschitz provided that $\hat{v}^i, \hat{w}^i$ are. This, however, is an immediate consequence of their definition in \eqref{eqn:v_def}, and Lemma \ref{lem:phi_higher}. Altogether, these arguments show that $\Jac(x,t)$ is Lipschitz.

Finally, this leaves us to showing that $(x,t) \ra \abs{t}/D_{\beta}(\rho(x,t))$ is Lipschitz. To ease notation, write $D = D_{\beta}$, and notice that it suffices to show
\begin{align}
\abs{ (D(\rho(x,t))^{-\beta} -  \abs{t}^{-\beta}} & \le C \abs{t}^{-\beta + 1} \label{eqn:dest_1}\\
\abs{ \nabla D(\rho(x,t))^{-\beta} -   \nabla \abs{t}^{-\beta}} & \le C \abs{t}^{-\beta}. \label{eqn:dest_2}
\end{align}
Indeed, the fact that $\abs{t} \simeq D \circ \rho(x,t)$, $\abs{\nabla \abs{t}} + \abs{\nabla D} \lesssim 1$, and the quotient rule readily imply that $\abs{t}/(D \circ \rho(x,t))$ is Lipschitz when \eqref{eqn:dest_1} and \eqref{eqn:dest_2} hold. To this end, most of our work is done for us already; notice that since $R_{x,r}$ is an isometry and $P(x,r)$ is the tangent plane to $\phi_r$ at $\Phi_r(x)$, then the definition of $\rho$ readily gives
\begin{align*}
\dist(\rho(x,t), P(x,r)) = \abs{ \rho(x,t) - \Phi_r(x)} & = h(x,t) \abs{t},
\end{align*}
or in other words, $\abs{t}^{-\beta} \equiv c_\beta \lambda(x,r) \dist(\rho(x,t), P(x,r))^{-\beta}$. Lemma 5.59 in \cite{DFMDAHL} then gives
\begin{align}\label{eqn:D_est}
\abs{ (D(\rho(x,t))^{-\beta} - c_\beta \lambda(x,r) \dist(\rho(x,t), P  (x,r))^{-\beta}} \le C r^{-\beta} \sum_{k \ge 0} 2^{-k \beta} \alpha_\mu(\Phi(x), C 2^k r),
\end{align}
where again, the $\alpha_\mu$ are the Tolsa $\alpha$ numbers associated to the measure $\mu = \HD^d|_\Gamma$. Using Lemma \ref{lem:alpha_c11} and the fact that $\alpha_\mu \lesssim 1$, we have that $\alpha_\mu(\Phi(x), C 2^k r) \le C 2^k r$ for all $k \in \N$ (indeed if $C 2^k r > r_0$, we bound crudely $\alpha_{\mu}(\Phi(x), C2^k r) \le C \le C r_0^{-1} 2^k r$). Hence  \eqref{eqn:D_est} then gives us 
\begin{align*}
\abs{ (D(\rho(x,t))^{-\beta} - \abs{t}^{-\beta}} \le C r^{-\beta + 1} \sum_{k \ge 0} 2^{-k (\beta - 1)} \le C r^{- \beta + 1},
\end{align*}
since $\beta > 1$, which is our first estimate.

As for the term with the gradients, we compute (recalling that the Jacobian acts on the left)
\begin{align}\label{eqn:d_deriv}
\nabla( D(\rho(x,t))^{-\beta}) & =  (-d-\beta) \int  \dfrac{ \nabla \rho(x,t)  (\rho(x,t) - Y)   }{  \abs{\rho(x,t) - Y}^{d+\beta+2}  } \; d \mu(Y).
\end{align}
On the other hand, we have from \eqref{eqn:cbeta_prop} and again the fact that $h(x,t) \abs{t} = \dist(\rho(x,t), P(x,r))$, 
\begin{align*}
\abs{t}^{-\beta - 2} & = h(x,t)^{\beta+2} \dist(\rho(x,t), P(x,r))^{-\beta -2}   = \int_{P(x,r)} \dfrac{c_{\beta + 2}^{-1} h(x,t)^{\beta + 2}}{\abs{\rho(x,t) - Y}^{d+\beta + 2}} \; d\HD^d(Y),
\end{align*}
and thus,
\begin{align}\label{eqn:t_deriv}
\nabla \abs{t}^{-\beta} & = - \beta t  \int_{P(x,r)} \dfrac{c_{\beta + 2}^{-1} h(x,t)^{\beta + 2}}{\abs{\rho(x,t) - Y}^{d+\beta + 2}} \; d\HD^d(Y).
\end{align}
Recalling the definition of $Q(x,t)^T$ as the matrix that maps $\R^d$ to $P'(x,r)$ (by left multiplication), then $h(x,t)t = Q(x,t)^T (\rho(x,t) - \Phi_r(x))$ and so since $\rho(x,t) - \Phi_r(x)$ is orthogonal to $P(x,r)$, we obtain by symmetry that 
\begin{equation}\label{eqn:t_dir}
\begin{split}
\int_{P(x,r)} \dfrac{h(x,t)t}{\abs{\rho(x,t) - Y}^{d+\beta + 2}} \; d\HD^d(Y) & = \int_{P(x,r)}  \dfrac{ Q(x,t)^T(\rho(x,t) - \Phi_r(x))}{\abs{\rho(x,t) - Y}^{d+\beta+2}} \; d\HD^d(Y) \\
& = \int_{P(x,r)}  \dfrac{ Q(x,t)^T(\rho(x,t) - Y)}{\abs{\rho(x,t) - Y}^{d+\beta+2}} \; d\HD^d(Y).
\end{split}
\end{equation}
Finally, recalling the decomposition $J = J' + H + M$ and Corollary \ref{cor:J_decomp}, we notice that $J'(x,t)t = ht$, and thus $\abs{J(x,t)t - h(x,t)t} \le C \abs{t}^2$. Hence we combine \eqref{eqn:t_deriv}, \eqref{eqn:t_dir} , the definition \eqref{eqn:jax_decomp} and the bound $\int_{P(x,r)} \abs{\rho(x,t) - Y}^{-d-\beta-2} \; d\HD^d(Y) \le C \abs{t}^{-\beta - 2}$ to conclude
\begin{align*}
\abs{ \int_{P(x,r)} \dfrac{h(x,t)^2 t}{\abs{\rho(x,t) - Y}^{d+\beta+2}} \; d\HD^d(Y) - \int_{P(x,r)} \dfrac{\nabla \rho(x,t) (\rho(x,t) - Y)}{\abs{\rho(x,t) - Y}^{d+\beta+2}} \; d\HD^d(Y)  } \le C \abs{t}^{-\beta}.
\end{align*}
It can be shown \cite[(3.30)]{Fen_22} that $(\beta + d) c_{\beta+2} = \beta c_\beta$, and thus from \eqref{eqn:t_deriv}, the above estimate, and the fact that $h^\beta \equiv c_\beta \lambda$, we obtain the estimate
\begin{align}\label{eqn:t_deriv_2}
\abs{ \nabla \abs{t}^{-\beta} - (-d - \beta) \lambda(x,r) \int_{P(x,r)} \dfrac{\nabla \rho(x,t) (\rho(x,t) - Y)}{\abs{\rho(x,t) - Y}^{d+\beta+2}} \; d\HD^d(Y)  } \le C \abs{t}^{-\beta}.
\end{align}

As a final step, one can use the same arguments as in Lemm 5.59 of \cite{DFMDAHL}, or more directly, \cite[Corollary 3.8]{FL23} to obtain the estimate
\begin{align*}
\abs{\int \dfrac{\nabla \rho(x,t) (\rho(x,t) - Y)}{\abs{\rho(x,t) - Y}^{d+\beta+2}} \; \left ( d\mu - \lambda(x,r) d\HD^d|_{P(x,r)} \right) } & \le C r^{-\beta-1} \sum_{k \ge 0} 2^{-k(\beta + 1)} \alpha_\mu(\Phi(x), C 2^k r) \\ 
&  \le C r^{-\beta},
\end{align*}
since $\alpha_\mu(\Phi(x), C r) \le C r$ as argued previously. Indeed, the only essential difference in the proof of the first inequality above and the estimate \eqref{eqn:D_est} in \cite{DFMDAHL} is the fact that in the annulus $A_k = B_{2^k r}(\Phi(x)) \setminus B_{2^{k-1}r}(\Phi(x))$, the integrand $f(Y) =   \abs{\rho(x,t) - Y}^{-d-\beta-2}(\rho(x,t) - Y)$ satisfies 
\begin{align*}
\abs{f(Y)} + 2^k r \abs{\nabla f(Y)} \le C (2^k r)^{-d - \beta -1} \text{ in } A_k,
\end{align*}
whereas the integrand $g(Y) = \abs{\rho(x,t) - Y}^{-d-\beta}$ satisfies 
\begin{align*}
\abs{g(Y)} + 2^k r \abs{\nabla g(Y)} \le C (2^k r)^{-d - \beta}, \text{ in } A_k.
\end{align*}
Following the exact argument in \cite{DFMDAHL} Lemma 5.59 with this replacement then gives the estimate. In conjunction with \eqref{eqn:d_deriv} and the estimate \eqref{eqn:t_deriv_2}, our desired estimate \eqref{eqn:dest_2} then follows.

To finish the proof, we need to verify the existence of the approximating block-form matrix $\mathcal{B}$ for $\A_{\rho, \beta}$ as in Definition \ref{defn:c1a}. Recalling that $\Jac(x,t) = J(x,t)Q(x,t)^T$ with $Q(x,t)$ an orthogonal matrix, we may write
\begin{align*}
\A_{\rho, \beta}(x,t) = \left( \dfrac{\abs{t}}{D_{\beta}(\rho(x,t))} \right)^{n-d-1} \abs{\det \Jac(x,t)} (J(x,t)^{-1})^T J(x,t)^{-1}.
\end{align*}
By Corollary \ref{cor:J_decomp}, we know that there is a decomposition $J = J' + M + H$ where $J'$ is uniformly elliptic with appropriate block structure, and $\abs{M(x,t)} + \abs{H(x,t)} \le C \cphii r$. In particular, notice that 
\begin{align*}
J^{-1} & = \left( J' + M + H \right)^{-1} \\
& = (J')^{-1} \left( I + (J')^{-1} M +  (J')^{-1} H \right)^{-1},
\end{align*}
and so the fact that $J'$ is uniformly elliptic and the bounds on $M$ and $H$ then give us that 
\begin{align*}
J(x,t)^{-1} & = (J'(x,t))^{-1} + E(x,t)
\end{align*}
with error satisfying $\abs{E(x,t)} \le C \cphii r$. Recalling the block structure on $J'$, we compute
\begin{equation*}
\begin{split}
(J(x,t)^{-1})^{T} J(x,t)^{-1}  & = ((J'(x,t))^{-1})^T (J'(x,t))^{-1}  + \tilde{E}(x,t)   \\
& = \begin{pmatrix}
((J_1'(x,t))^{-1})^{T} (J_1'(x,t))^{-1} & 0 \\
0 & h(x,t)^{-2} I \\
\end{pmatrix} + \tilde{E}(x,t), 
\end{split}
\end{equation*}
where again the estimate $\abs{\tilde{E}(x,t)} \le C \cphii r$ holds. 

We now have a natural candidate for the matrix $\mathcal{B}$ required in the higher co-dimension $C^{0,1}$ condition: define
\begin{align*}
\mathcal{B}(x,t) & \coloneqq \left( \dfrac{\abs{t}}{D_{\beta}(\rho(x,t))} \right)^{n-d-1} \abs{\det \Jac(x,t)} ((J'(x,t))^{-1})^T (J'(x,t))^{-1}.
\end{align*}
Notice that since $D_{\beta}(\rho(x,t)) \simeq t$ and $\det \Jac(x,t) \simeq h^{n-d} \simeq 1$, the above estimate on $\tilde{E}$ gives us that 
\begin{align*}
\abs{\A_{\rho, \beta}(x,t) - \mathcal{B}(x,t)} \le C \cphii r = C \cphii \delta,
\end{align*}
which is the one of the required conditions for $\mathcal{B}$. The second required condition is that $\mathcal{B}$ be Lipschitz and uniformly elliptic. Uniform ellipticity is clear, again by Corollary \ref{cor:J_decomp} and the fact that $h\simeq 1$. On the other hand, we see that $\mathcal{B}$ is Lipschitz provided the entries of $J_1'(x,t)$ are Lipschitz, since we already have Lipschitz bounds on the other terms defining $\mathcal{B}$ in \eqref{eqn:lip_reduc} and \eqref{eqn:h_est}. To this end, we invoke the definition of $(J_1'(x,t)_{k \ell} = \dotp{\hat{v}^k, v^\ell}$ from \cite[Definition 3.12]{DFMDAHL}, where the $\hat{v}^k$ and $v^k$ are those same vectors as coming from \eqref{eqn:v_def}. However, as mentioned there, these entries are Lipschitz since the $v^k$ and $\hat{v}^k$ are all bounded and Lipschitz by Lemma \ref{lem:phi_higher}. Altogether, this completes the proof of the Theorem.
\end{proof}

\begin{rmk}\label{rmk:nonsmooth_op}
	It turns out that a $C^{1,1}$ smoothness assumption on $\Gamma$ is also enough to guarantee that our main result, Theorem \ref{thm:main}, holds for operators of the form
	\begin{equation}\label{eqn:L_gen}
		L_a \coloneqq - \divv( a \, \delta_\Gamma^{-n+d+1} \nabla \, \cdot \,   ),
	\end{equation}
	whenever $a \simeq 1$ is uniformly Lipschitz, where $\delta_\Gamma(X) = \dist(X, \Gamma)$ again is the usual Euclidean distance function to the graph $\Gamma$. At the moment, though, we do not know how to prove this without passing through the same estimates of the proof of Theorem \ref{thm:change_var} (and the regularized distance $D_\beta$).
	
	In fact, in light of Corollary \ref{cor:main} it suffices to show that when $\Gamma$ is $C^{1,1}$ and $\beta > 2$ we have that $X\mapsto \frac{D_\beta(X)}{\delta(X)}$ is a Lipschitz function (we already know it is uniformly bound away from zero and infinity) in a neighborhood of the boundary. 
	
	First, since $\Gamma$ is given by the graph of the map $\phi \in C^{1,1} (\R^d ; \R^{n-d})$ with the bound \eqref{eqn:phi_est}, we know that there is a small ball $B_{R_0}$ centered on $\Gamma$ for which the tubular neighborhood of $\Gamma$ is well-defined (see for example, \cite[Remark 4.20]{Fed});  for each $Z \in B_{R_0}$, there exists a unique point $P(Z) \in \Gamma$ such that $\delta_\Gamma(Z) = \abs{ Z - P(Z)}$. Of course the scale $R_0$ depends on the bound \eqref{eqn:phi_est}. If we denote by $e_Z = (Z - P(Z)) / \delta_\Gamma(Z)$, then we know that $e_Z$ is a unit vector orthogonal to the tangent plane of $\Gamma$ at $P(Z)$, and moreover, 
	\begin{equation}\label{eqn:delta_struct}
		\begin{split}
		\nabla \delta_\Gamma(Z) & = e_Z,\\
		\delta_\Gamma(P(Z) + t e_Z) & = t \text{ for } t \in (0, \delta_\Gamma(Z)).
		\end{split}
	\end{equation}
	To show $(D_\beta /\delta_\Gamma)(Z)$ is Lipschitz, just like before, it suffices to prove
	\begin{align}
		\abs{ c_\beta \delta_\Gamma(Z)^{-\beta} -  D_\beta(Z)^{-\beta}  } & \le C \delta_\Gamma(Z)^{-\beta + 1} , \label{eqn:dest_1p} \\
		\abs{ \nabla( c_\beta \delta_\Gamma^{-\beta})(Z) -  D_\beta(Z)^{-\beta}  } & \le C \delta_\Gamma(Z).  \label{eqn:dest_2p}
	\end{align}
	which are analogous to \eqref{eqn:dest_1} and \eqref{eqn:dest_2} in the proof of Theorem \ref{thm:change_var}. For the sake of brevity, let us just prove \eqref{eqn:dest_1p}. 
	
	Fix $Z \in B_{R_0}$, and assume for the sake of notation that $P(Z) = 0$,  and additionally (up to rotation and translation),  $\abs{\nabla \phi(0)} = \phi(0) = 0 $. With this choice of notation, we see then that $\delta_\Gamma(Z) = \delta_{\R^d}(Z)$ coincides with the distance from $Z$ to the tangent plane to $\Gamma$ at $0$, which is $\R^d$. Recalling the defining property of $c_\beta$ as in \eqref{eqn:cbeta_prop} and the fact that $\Gamma$ is the graph of $\phi$ over $\R^d$, we can then write
	\begin{equation}\label{eqn:sum_spl}
		\begin{split}
		\abs{ c_\beta \delta_\Gamma(Z)^{-\beta} - D_\beta(Z)^{-\beta}} & = \abs{  \int_{\R^d}  \abs{Z-Y}^{-d-\beta} \; d \HD^d(Y)  - \int_{\Gamma} \abs{Z-Y}^{-d-\beta} \; d \HD^d(Y)} \\
		& = \abs{\int_{\R^d} \abs{Z- Y}^{-d-\beta} \left( 1 - \sqrt{1 + \abs{\nabla \phi(Y)}^2} \right) \; d\HD^d(Y)} \\
		& \le \sum_{k \ge 0} \int_{\R^d \cap A_k} \abs{Z-Y}^{-d-\beta} \abs{   1 - \sqrt{ 1 + \abs{\nabla \phi(y) }^2 }} \; d\HD^d(Y),
		\end{split} 
	\end{equation}
	where we set $r  = \delta(Z)$, $A_0 \coloneqq B_{2r}$, and $A_k \coloneqq B_{2^{k}r} \setminus B_{2^{k-1}r}$ for $k \ge 1$. Choose $k_0 \in \N$ to be the smallest integer with $2^{k_0} r > r^{1/2}$. Taking $R_0$ (and hence $r$) to be small, we know that the fact that $\nabla^2 \phi$ is bounded and $\nabla \phi(0) = 0$ implies that $\abs{\nabla \phi (Y)} \le C \abs{Y}$ for $\abs{Y}$ small, so that $\abs{ 1- (1 + \abs{\nabla \phi(Y)}^2)^{1/2}} \le C \abs{Y}^2$ for $Y \in \R^d \cap A_k$, $0 \le k \le k_0$. This allows us to easily estimate
	\begin{align*}
		\sum_{ 0 \le k \le k_0}   \int_{\R^d \cap A_k} \abs{Z-Y}^{-d-\beta} & \abs{   1 - \sqrt{ 1 + \abs{\nabla \phi(y) }^2 }} \;  d\HD^d(Y)  \\
		& \le C \sum_{0 \le k \le k_0} \HD^d(\R^d \cap A_k) (2^k r)^2 (2^k r)^{-d-\beta} \\
		& \le C r^{-\beta+2} \sum_{k \ge 0} 2^{k(2-\beta)} \le C r^{-\beta+2}
	\end{align*}
	provided that $\beta > 2$. On the other hand, we crudely estimate
	\begin{align*}
		\sum_{ k > k_0}   \int_{\R^d \cap A_k} \abs{Z-Y}^{-d-\beta}   \;  d\HD^d(Y)   \le C  	\sum_{ k > k_0} (2^k r)^{-\beta} \le C r^{-\beta} 2^{-k_0 \beta} \le C r^{-\beta + \beta/2} \le C r^{-\beta + 1}
	\end{align*}
	by the choice of $k_0$ and since $\beta >2$. Combining the above two estimates (and using that $\Gamma$ is $d$-Ahlfors regular) with \eqref{eqn:sum_spl} then shows \eqref{eqn:dest_1p}. The estimate for \eqref{eqn:dest_2p} is argued in a similar manner, using also \eqref{eqn:delta_struct} and the techniques used in proving \eqref{eqn:dest_2}. We omit the details for the sake of brevity.
\end{rmk}																																																																																																																													 

\bibliographystyle{alpha}
\bibliography{bibl}

\end{document}